\documentclass[11pt]{article}
\usepackage[utf8]{inputenc}

\usepackage{algorithm}
\usepackage{algorithmic}

\usepackage{pgf}%figure
\usepackage{tikz}
\usetikzlibrary{arrows, decorations.pathmorphing, backgrounds, positioning, fit, petri, automata}
\definecolor{yellow1}{rgb}{1,0.8,0.2}

\usepackage{makecell}

\usepackage{amsmath,amsthm,amsfonts,amssymb,amsbsy,amsopn,amstext}
\usepackage{graphicx,color}
\usepackage{mathbbol,mathrsfs}
\usepackage{float,ccaption}
\usepackage{indentfirst}
\usepackage{appendix}
\usepackage{subfigure}
\usepackage{parskip}

\usepackage{emptypage}

\usepackage{booktabs}
\usepackage{threeparttable}
\usepackage{multirow}
\usepackage{diagbox}
\usepackage[numbers,sort&compress]{natbib}

 \newtheorem{thm}{Theorem}
 
 \newtheorem{lem}{Lemma}

 \newtheorem{ass}{Assumption}
\newcommand{\define}{:=}%\newcommand{\define}{\stackrel{\triangle}{=}}
\newcommand{\normm}[1]{{\left\vert\kern-0.25ex\left\vert\kern-0.25ex\left\vert #1
		\right\vert\kern-0.25ex\right\vert\kern-0.25ex\right\vert}}

\usepackage{geometry}
\begin{document}
\title{Numerical Methods for  Distributed Stochastic Compositional Optimization Problems with Aggregative Structure}

\author{Shengchao Zhao and Yongchao Liu\thanks{ School of Mathematical Sciences, Dalian University of Technology, Dalian 116024, China, e-mail: zhaoshengchao@mail.dlut.edu.cn (Shengchao Zhao), lyc@dlut.edu.cn (Yongchao Liu) }}%The research is supported by the NSFC \#11971090. \\
\date{}
\maketitle
\noindent{\bf Abstract.} The paper studies the distributed stochastic compositional optimization  problems over networks,  where all the  agents' inner-level function    is the sum of  each agent's private expectation function.
 Focusing  on the aggregative structure of the inner-level function, we employ the
 hybrid variance reduction method to obtain the information on each agent's private expectation function, and apply the dynamic consensus mechanism   to track the information
on each agent's inner-level function. Then by combining with the standard  distributed stochastic gradient descent method,  we  propose a distributed aggregative stochastic compositional gradient descent method.
   When the objective function is smooth, the proposed method achieves the optimal convergence rate $\mathcal{O}\left(K^{-1/2}\right)$.
 We further combine the proposed method with the communication compression and propose    the communication compressed variant  distributed aggregative stochastic compositional gradient descent method. The compressed variant of the proposed method
 maintains the optimal convergence rate $\mathcal{O}\left(K^{-1/2}\right)$. Simulated experiments  on decentralized reinforcement learning verify the effectiveness of the proposed methods.

\section{Introduction}

Stochastic compositional optimization problem has been widely studied due to its extensively emerging applications in machine learning \cite{Dai2017Learning,Bottou2018opti,Bo2018SBEED,tianyi2021sol}, risk-averse portfolio optimization \cite{Shapiro2009lect} and adaptive simulation \cite{hu2014model}.
In this paper, we consider the distributed  stochastic compositional optimization  problem
\begin{equation}\label{model}
\min_{x\in\mathbb{R}^d} h(x)\define\frac{1}{n}\sum_{j=1}^n f_j\left(\frac{1}{n}\sum_{j=1}^ng_j(x)\right)
\end{equation}
over networks, where     $g_j(x)\define\mathbb{E}\left[G_j(x;\phi_j)\right]$ is the private expectation function and
 $f_j(y)\define \mathbb{E}\left[F_j(y;\zeta_j)\right]$  is the private outer-level function  of agent $j \in[1,2,\cdots,n]$,  $G_j(\cdot;\phi_j):\mathbb{R}^d\rightarrow \mathbb{R}^p$ and $F_j(\cdot;\zeta):\mathbb{R}^p\rightarrow \mathbb{R}$ are measurable functions parameterized by random variables $\phi_j$ and $\zeta_j$ respectively.
Since  the  inner level function $\frac{1}{n}\sum_{j=1}^ng_j(x)$  aggregates  each agent's private function $g_j(\cdot), j=1,\cdots, n$, we call problem (\ref{model})  distributed  stochastic compositional optimization  problem with aggregative structure \cite{Li2022Aggregative} and DSCA for short.

%of each agent's objective,   is expect-valued and aggregative  for all agents, the DSCO problem (\ref{model}) is significantly different from the distributed noncompositional optimization problem.
% Obviously, it is more  difficult to solve problem (\ref{model}) over networks since the inner level of each agent's objective function is expect-valued and aggregative  for all agents. without the aggregative structure like problem (\ref{model}),

In the past decades,   stochastic gradient decent methods have been well studied for solving stochastic compositional optimization problem, such as, two-timescale scheme method \cite{wang2016acce,wang2017stochastic}, sing-timescale scheme method \cite{ghadimi2019single,tianyi2021sol} and  variance reduction based method \cite{Shen2018Asynchronous,Huo18,Lian2017Finite}.
 More recently, Gao and Huang \cite{gao2021fast} study distributed  stochastic compositional optimization  problem over undirected communication networks, where a gossip-based distributed stochastic  gradient descent method and its  gradient-tracking version are proposed.   Zhao and Liu \cite{Zhao2022compositional} propose a push-pull based  distributed stochastic  gradient descent method for   distributed  stochastic compositional optimization  problem
over directed communication networks. Both of works \cite{gao2021fast,Zhao2022compositional} achieve the   optimal convergence rate $\mathcal{O}\left(K^{-1/2}\right)$.
 The deterministic distributed optimization problems with the   aggregative structure have been studied in \cite{Li2022Aggregative,Sundhar2012Anew}.
 Li et al. \cite{Li2022Aggregative} consider the distributed aggregative optimization problem,
 which allows each agent's
 objective function  to be dependent not only on their own decision
variables, but also on the average of summable functions of decision
variables of all other agents.
 Ram et al. \cite{Sundhar2012Anew} model the regression problem as   the distributed constrained optimization problem,  where all agents cooperatively minimize a nonlinear function of the sum of the individual agent’s objective functions over the constraint set. %To address the challenge caused by the  aggregative structure, two distributed optimization methods by combining gradient descent method with dynamic average consensus mechanism are developed in \cite{Li2022Aggregative} and \cite{Sundhar2012Anew} respectively.
Yang et al. \cite{yang2022bilevel} consider the distributed bilevel stochastic optimization problem, where
each agent's objective depends not only on its own decision variable but also  on the optimal solution of an sum of the  expect-valued functions held privately by all agents.
%Koshal et al. \cite{Jayash2016Agggame} consider the aggregative games, where {\color{brown} a player’s objective is a function of the aggregative of all the players’ decisions.  Every player maintains an estimate of this aggregative, and the players exchange this information with their local neighbors over a connected network.
The authors \cite{yang2022bilevel} propose a gossip-based distributed bilevel learning algorithm that allows
networked agents to solve both the inner and outer optimization problems in a
single timescale and share information via network propagation.

In this paper, we focus on the numerical methods for distributed stochastic compositional optimization problems with aggregative structure (\ref{model}). We    propose a distributed aggregative stochastic compositional gradient descent method (D-ASCGD for short), which combines the distributed stochastic gradient descent method with the hybrid variance reduction method \cite{Ashok2019Momentum} and the  dynamic average consensus mechanism \cite{ZHU2010Discrete,Saber2005Consensus}. Specifically, we first construct the estimators of  the values and the gradients of each agent's private expectation function  with diminishing bias via the hybrid variance reduction technique.
%, and track the global inner-level function value and its corresponding gradient based on the  dynamic average consensus mechanism.
Then  we track the values and the gradients of the inner-level    function   by employing dynamic average consensus mechanism.
 Finally,  combined with the standard distributed stochastic gradient method,  we  arrive at the  D-ASCGD.
 The proposed D-ASCGD achieves the convergence rate $\mathcal{O}\left(K^{-1/2}\right)$, which matches the optimal convergence rate of stochastic gradient descent  method \cite{Ghadimi2013nonconvex}.
We further combine the D-ASCGD with the compress procedure considered in \cite{liu2021linear}, which induces  the communication compressed variant of distributed  stochastic compositional gradient descent method (CD-ASCGD for short).
The  CD-ASCGD compresses the decision variables, the  trackers of the inner-level function value and its corresponding gradient  to provide a communication-efficient implementation. CD-ASCGD also achieves the  optimal convergence rate $\mathcal{O}\left(K^{-1/2}\right)$.

The rest of this paper is organized as follows. Section \ref{sec:D-method} introduces the proposed D-ASCGD and some standard
assumptions on problem (\ref{model}), communication graphs and weighted matrices. Section \ref{sec:D-analysis} focuses on the convergence
analysis of D-ASCGD. Section \ref{sec:CD-method} presents the communication compressed variant  CD-ASCGD and analyzes  its convergence rate.
Preliminary   numerical test is provided in section \ref{sec:num}. % to validate the theoretic results.

Throughout this paper, we use the following notation. $\mathbb{R}^d$ denotes  the d-dimension Euclidean space endowed with norm $\|x\|=\sqrt{\langle x,x\rangle}$. $\normm{\cdot}_F$ and $\normm{\cdot}$ denote the Frobenius norm and matrix norm introduced by $\|\cdot\|$ respectively.  $col(x_1,x_2,\cdots,x_n)$ denotes the
column vector by stacking up vectors $x_1,x_2,\cdots,x_n$,   $\mathbf{1}:=(1~1\dots1)^\intercal\in \mathbb{R}^{n}$ and $\mathbf{0}\define(0~0\dots0)^\intercal\in \mathbb{R}^{d}$.
$\mathbf{I}_{d}\in\mathbb{R}^{d\times d}$ stands for the identity matrix and
$\mathbf{A}\otimes \mathbf{B}$ denotes the Kronecker product of matrices $\mathbf{A}$ and $\mathbf{B}$. For any positive sequences $\{a_k\}$ and $\{b_k\}$,  $a_k=\mathcal{O}(b_k)$ if there exists $c>0$ such that $a_k\le c b_k$. The communication relationship between agents is characterized by a directed graph $\mathcal{G}=\left(\mathcal{V},\mathcal{E}\right)$, where $\mathcal{V}=\{1,2,...,n\}$ is the node set and $\mathcal{E}\subseteq\mathcal{V}\times\mathcal{V}$ is the edge set. For any $i\in\mathcal{V}$, $P_{\phi_i}$ and $P_{\zeta_i}$ are the distributions of random variables $\phi_i$ and $\zeta_i$ respectively.  For a set  $\mathcal{S}$, $|\mathcal{S}|$ denotes its cardinality.

\section{D-ASCGD method}\label{sec:D-method}

As we discussed in the introduction,  the D-ASCGD  uses hybrid variance reduction technique \cite{Ashok2019Momentum} to estimate  the value and the  gradient of each agent's private expectation function, and  uses
the  dynamic average consensus mechanism \cite{ZHU2010Discrete,Saber2005Consensus} to track the value and the gradient of the inner-level  function, which reads as follows.
%We propose a distributed stochastic gradient descent based method for solving problem (\ref{model}).
\begin{algorithm}[H]
	\caption{$\underline{\text{D}}$istributed $\underline{\text{A}}$ggregative $\underline{\text{S}}$tochastic $\underline{\text{C}}$ompositional $\underline{\text{G}}$radient $\underline{\text{D}}$escent (D-ASCGD):}\label{alg:SIA}
	\begin{algorithmic}[1]
		\REQUIRE initial values $x_{i,1}\in\mathbb{R}^{d}$,  $y_{i,1}=G_{i,1}\in\mathbb{R}^{p}$, $z_{i,1}=\hat{G}_{i,1}\in\mathbb{R}^{d\times p}$; stepsizes $\alpha_k>0$, $\beta_k>0,\gamma_k>0$; nonnegative weight matrix $\mathbf{W}=\{w_{ij}\}_{1\le i,j\le n}\in \mathbb{R}^{n\times n}.$        %%input
		%\ENSURE Optimum set $\mathbb{X}$  %%output
		\FOR {$k=1,2,\cdots$}
		\FOR {$i=1,\cdots,n$ in parallel}
		\STATE Draw $\phi_{i,k+1}\stackrel{iid}\sim P_{\phi_i}$, $\zeta_{i,k+1}\stackrel{iid}\sim P_{\zeta_i}$. Compute function values $G_i(x_{i,k};\phi_{i,k+1})$, $G_i(x_{i,k+1};\phi_{i,k+1})$ and gradients $\nabla F_i(y_{i,k};\zeta_{i,k+1})$, $\nabla G_i(x_{i,k};\phi_{i,k+1})$, $\nabla G_i(x_{i,k+1};\phi_{i,k+1}).$
		\STATE\label{gdstep} $x_{i,k+1}=\sum_{j=1}^n w_{ij} x_{j,k}-\alpha_k z_{i,k}\nabla F_i(y_{i,k};\zeta_{i,k+1})$
		\STATE\label{step:G} $G_{i,k+1}=(1-\beta_k)\left(G_{i,k}-G_i(x_{i,k};\phi_{i,k+1})\right)+ G_i(x_{i,k+1};\phi_{i,k+1})$
		\STATE \label{step:G-1}$\hat{G}_{i,k+1}=(1-\gamma_k)\left(\hat{G}_{i,k}-\nabla G_i(x_{i,k};\phi_{i,k+1})\right)+\nabla G_i(x_{i,k+1};\phi_{i,k+1})$
		\STATE\label{step:y}$y_{i,k+1}=\sum_{j=1}^n w_{ij} y_{i,k}+G_{i,k+1}-G_{i,k}$
		\STATE\label{step:z}$z_{i,k+1}=\sum_{j=1}^n w_{ij} z_{j,k}+\hat{G}_{i,k+1}-\hat{G}_{i,k}$
		\ENDFOR
		\ENDFOR
	\end{algorithmic}
\end{algorithm}

In Algorithm \ref{alg:SIA}, the key issue is  to estimate the stochastic gradient $$\frac{1}{n}\sum_{j=1}^n \nabla g_j(x)\nabla F_i\left(\frac{1}{n}\sum_{j=1}^n g_j(x);\zeta_{i,k+1}\right)$$
for each agent $j$. We employ the   hybrid variance reduction technique \cite{Ashok2019Momentum} to estimate local function value  $g_i(x)$ and  gradient $\nabla g_i(x)$ in Lines \ref{step:G} and \ref{step:G-1}. Taking $G_{i,k+1}$ as an example,  we have
\begin{align*}
  G_{i,k+1}=(1-\beta_k)\left( {G_{i,k}+ G_i(x_{i,k+1};\phi_{i,k+1})-G_i(x_{i,k};\phi_{i,k+1})} \right)+\beta_k  {G_i(x_{i,k+1};\phi_{i,k+1})}
\end{align*}
where the  first term plays the role of variance reduction and the second term  is the stochastic function value.  The  convex combination of the two terms may reduce the variance of    $G_{i,k+1}$ estimating $g(x)$  gradually.
In Lines \ref{step:y} and \ref{step:z}, we utilize the dynamic average consensus mechanism \cite{ZHU2010Discrete,Saber2005Consensus} to design the trackers of the  inner level function.  As a result of $\mathbf{W}$ being doubly stochastic, we have
\begin{equation*}
\frac{1}{n}\sum_{j=1}^n y_{j,k}=\frac{1}{n}\sum_{j=1}^n G_{j,k},\quad \frac{1}{n}\sum_{j=1}^n z_{j,k}=\frac{1}{n}\sum_{j=1}^n \hat{G}_{j,k}.
\end{equation*}
Then if  $G_{i,k}$ and $\hat{G}_{i,k}$ converge to $g_i(x)$ and $\nabla g_i(x)$ respectively,
$y_{i,k}$ and  $z_{i,k}$ could track   $\frac{1}{n}\sum_{j=1}^n g_j(x)$ and  $\frac{1}{n}\sum_{j=1}^n \nabla g_j(x)$ successfully.
%Indeed, the dynamic average consensus mechanism has been adopted in the distributed optimization algorithms \cite{Nedi2017Achie,Xin2020Gener,Xu2015aug}  widely, where it is used to track the average of gradients.

 For the sake of notational convenience, we denote
\begin{align*}
&G_{k+1,k+1}=col\left(G_1(x_{1,k+1};\phi_{1,k+1}),G_2(x_{2,k+1};\phi_{2,k+1}),\cdots,G_n(x_{n,k+1};\phi_{n,k+1})\right),\\
&G_{k+1,k}=col\left(G_1(x_{1,k};\phi_{1,k+1}),G_2(x_{2,k};\phi_{2,k+1}),\cdots,G_n(x_{n,k};\phi_{n,k+1})\right),\qquad\qquad\qquad\qquad\quad
\end{align*}
\vspace{-1cm}
\begin{align*}
&\nabla G_{k+1,k+1}=\left(\nabla G_1(x_{1,k+1};\phi_{1,k+1})^\intercal,\nabla G_2(x_{2,k+1};\phi_{2,k+1})^\intercal,\cdots,\nabla G_n(x_{n,k+1};\phi_{n,k+1})^\intercal\right)^\intercal,\\
&\nabla G_{k+1,k}=\left(\nabla G_1(x_{1,k};\phi_{1,k+1})^\intercal,\nabla G_2(x_{2,k};\phi_{2,k+1})^\intercal,\cdots,\nabla G_n(x_{n,k};\phi_{n,k+1})^\intercal\right)^\intercal,\qquad\qquad\quad\\
&\mathbf{G}_k=col\left(G_{1,k},G_{2,k},\cdots,G_{n,k}\right),~\hat{\mathbf{G}}_k=\left(\hat{G}_{1,k}^\intercal,\hat{G}_{2,k}^\intercal,\cdots,\hat{G}_{n,k}^\intercal\right)^\intercal,
\end{align*}
\vspace{-1cm}
\begin{align*}
&\mathbf{U}_{k+1}=col\left(z_{1,k}\nabla F_1(y_{1,k};\zeta_{1,k+1}),~z_{2,k}\nabla F_2(y_{2,k};\zeta_{2,k+1}),\cdots,z_{n,k}\nabla F_n(y_{n,k};\zeta_{n,k+1})\right),\\
&\mathbf{x}_k=col(x_{1,k},x_{2,k},\cdots,x_{n,k}),~\mathbf{y}_k=col(y_{1,k},y_{2,k},\cdots,y_{n,k}),~\mathbf{z}_k=(z_{1,k}^\intercal,z_{2,k}^\intercal,\cdots,z_{n,k}^\intercal)^\intercal,\\
&\mathbf{g}_{k}=col\left(g_1(x_{1,k}),g_2(x_{2,k}),\cdots,g_n(x_{n,k})\right),\\
&\nabla \mathbf{g}_{k}=\left(\nabla g_1(x_{1,k})^\intercal,\nabla g_2(x_{2,k})^\intercal,\cdots,\nabla g_n(x_{n,k})^\intercal\right)^\intercal,\\
&\tilde{\mathbf{W}}\define\mathbf{W}\otimes \mathbf{I}_d,~\bar{y}_k=\frac{1}{n}\sum_{j=1}^n y_{j,k},~\bar{z}_k=\frac{1}{n}\sum_{j=1}^n z_{j,k},~\bar{x}_k=\frac{1}{n}\sum_{j=1}^n x_{j,k}.\qquad\qquad\qquad\qquad\qquad\quad
\end{align*}
Then  steps \ref{gdstep}-\ref{step:z} in Algorithm \ref{alg:SIA} can be rewritten as
\begin{align}
\label{alg:x-1}&\mathbf{x}_{k+1}=\tilde{\mathbf{W}}\mathbf{x}_k-\alpha_k\mathbf{U}_{k+1}\\
\label{alg:G-1}&\mathbf{G}_{k+1}=(1-\beta_k)\left(\mathbf{G}_k-G_{k+1,k}\right)+G_{k+1,k+1},\\
&\hat{\mathbf{G}}_{k+1}=(1-\gamma_k)\left(\hat{\mathbf{G}}_k-\nabla G_{k+1,k}\right)+\nabla G_{k+1,k+1},\notag\\
\label{alg:z-1}&\mathbf{y}_{k+1}=\tilde{\mathbf{W}}\mathbf{y}_k+\mathbf{G}_{k+1}-\mathbf{G}_k,\\
&\mathbf{z}_{k+1}=\tilde{\mathbf{W}}\mathbf{z}_k+\hat{\mathbf{G}}_{k+1}-\hat{\mathbf{G}}_k.\notag
\end{align}

Throughout our analysis in the paper, we make the following
	two assumptions on  objective functions, networks and weight matrices.
\begin{ass}[Objective function]\label{ass-objective} {\rm
		Let $C_g,C_f,V_g,L_g$ and $L_f$  be positive scalars. For $\forall i\in\mathcal{V}$, $\forall x,x^{'}\in\mathbb{R}^d$, $\forall y,y^{'}\in\mathbb{R}^p$,
		\begin{itemize}
			\item[(a)]  functions $G_i(\cdot;\phi_i)$ and $F_i(\cdot;\zeta_i)$ are  smooth, that is,
			\begin{equation*}
		   \normm{\nabla G_i(x;\phi_i)-\nabla G_i(x^{'};\phi_i)}\le L_g\|x-x^{'}\|
			\end{equation*}
			and
			\begin{equation*}
			\|\nabla F_i(y;\zeta_i)-\nabla F_i(y^{'};\zeta_i)\|\le L_f\|y-y^{'}\|    ;
			\end{equation*}
			\item [(b)] 
			$
		    \mathbb{E}\left[G_i(x;\phi_{i})\big|\zeta_i\right]= g_i(x),~ \mathbb{E}\left[\nabla G_i(x;\phi_{i})\big|\zeta_i\right]= \nabla g_i(x)$, $\mathbb{E}\left[\nabla F_i(y;\zeta_{i})\big|\phi_i\right]= \nabla f_i(y);$
			\item [(c)] the   gradients of   $G_i(\cdot;\phi_i)$ and $F_i(\cdot;\zeta_i)$  are bounded in mean square sense, that is,
			\begin{equation*}
			\mathbb{E}\left[\normm{\nabla G_i(x;\phi_i)}^2\big|\zeta_i\right]\le C_g,\quad \mathbb{E}\left[\|\nabla F_i(y;\zeta_i)\|^2\big|\phi_i\right]\le C_f;
			\end{equation*}
			\item [(d)]   $G(x;\phi_i)$ and  $\nabla G(x;\phi_i)$  have bounded variances, $$\mathbb{E}\left[\|G_i(x;\phi_i)-g_i(x)\|^2\big|\zeta_i\right]\le V_g,\quad \mathbb{E}\left[\normm{\nabla G_i(x;\phi_i)-\nabla g_i(x)}^2\big|\zeta_i\right]\le V_g^{'}.$$
	\end{itemize}}
\end{ass}
Assumption \ref{ass-objective} is standard for stochastic compositional optimization problem \cite{tianyi2021sol,wang2017stochastic,gao2021fast}. Conditions (a) and (c) require both the functions and the gradients to be Lipschitz continuity, conditions  (b) and (d) are
analogous to the unbiasedness and bounded variance assumptions for non-compositional stochastic optimization problems.

\begin{ass}[weight matrices and networks]\label{ass:matrix} {\rm
		 The directed graph $\mathcal{G}$ is strongly connected and permits a nonnegative doubly stochastic weight matrix $\mathbf{W}=\{w_{ij}\}_{1\le i,j\le n}\in \mathbb{R}^{n\times n}$, i.e. $\mathbf{W}\mathbf{1}=\mathbf{1}$ and $\mathbf{1}^\intercal\mathbf{W}=\mathbf{1}^\intercal$.}
\end{ass}
Assumption \ref{ass:matrix}  implies  that $\left(\frac{1}{n}\mathbf{1}\mathbf{1}^\intercal\right)\mathbf{W}=\mathbf{W}\left(\frac{1}{n}\mathbf{1}\mathbf{1}^\intercal\right)=\frac{1}{n}\mathbf{1}\mathbf{1}^\intercal$ and the spectral norm $\rho\define \normm{\mathbf{W}-\frac{\mathbf{1}\mathbf{1}^\intercal}{n}}$ satisfies $\rho<1$ \cite[Lemma 4]{Li2022Aggregative}.

\section{Convergence analysis for D-ASCGD}\label{sec:D-analysis}

In this section, we study the convergence of D-ASCGD.
Before the presentation of the convergence rate of D-ASCGD in Theorem 1,
we first quantify the estimating errors   for  the value  and gradient of each agent's private expected function.    Throughout the paper,   the proof of all the lemmas are relegated to Appendix \ref{apd:D-NSCGD}.

% present the upper bounds of the estimating error of the local inner-level function in Lemma \ref{lem:tracking}, and the upper bounds of  the consensus error of the global inner-level function trackers and decision variable in Lemma \ref{lem:consensus}. By combining  \ref{lem:tracking} with Lemma \ref{lem:consensus}, we further give the specific coupling relationship between decision variable, the estimators of the local inner-level function and the trackers of the global inner-level function in Lemma \ref{lem:complex}. % The proofs of all lemmas are given in Appendix \ref{apd:D-NSCGD}.

\begin{lem}\label{lem:tracking}
	Under Assumptions \ref{ass-objective} and \ref{ass:matrix},
	\begin{align}
	\mathbb{E}\left[\|\mathbf{G}_{k+1}- \mathbf{g}_{k+1}\|^2\right]
	&\le(1-\beta_k)^2\mathbb{E}\left[\|\mathbf{G}_k-\mathbf{g}_k\|^2\right]+48(1-\beta_k)^2C_g\mathbb{E}\left[ \left\|\mathbf{x}_k-\mathbf{1}\otimes \bar{x}_k\right\|^2\right]\notag\\
	\label{GG-bound}&\quad+12(1-\beta_k)^2C_g\alpha_k^2C_f\mathbb{E}\left[\normm{\mathbf{z}_k}_F^2\right]+3\beta_k^2V_g
	\end{align}
	and
	\begin{align}
	\mathbb{E}\left[\normm{\hat{\mathbf{G}}_{k+1}- \nabla\mathbf{g}_{k+1}}_F^2\right]
	&\le(1-\gamma_k)^2\mathbb{E}\left[\normm{\hat{\mathbf{G}}_k-\nabla\mathbf{g}_k}_F^2\right]+48(1-\gamma_k)^2pL_g^2\mathbb{E}\left[ \left\|\mathbf{x}_k-\mathbf{1}\otimes \bar{x}_k\right\|^2\right]\notag\\
	\label{Gd-bound}&\quad+12(1-\gamma_k)^2pL_g^2\alpha_k^2C_f\mathbb{E}\left[\normm{\mathbf{z}_k}_F^2\right]+3\gamma_k^2V_g^{'}.
	\end{align}
\end{lem}

%Lemma {lem:tracking}   quantifies  estimating errors   for the local inner-level function value and its corresponding gradient.%the error of estimating $g_i(x_{i,k})$ and $\nabla g_i(x_{i,k})$ by $y_{i,k}$ and $z_{i,k}$.

The next lemma studies the asymptotic consensus of D-ASCGD.
\begin{lem}\label{lem:consensus}
Suppose (a) Assumptions \ref{ass-objective}-\ref{ass:matrix} hold; (b) $ \{\beta_k\}$ and $\{\gamma_k\}$ are nonincreasing sequences such that $\beta_1\le 1$ and $\gamma_1\le 1$. Then
	\begin{align}
	\mathbb{E}\left[\left\|\mathbf{x}_{k+1}-\mathbf{1}\otimes\bar{x}_{k+1}\right\|^2\right]
	\label{xx-bound}&\le \frac{1+\rho^2}{2}\mathbb{E}\left[\left\|\mathbf{x}_k-\mathbf{1}\otimes\bar{x}_k\right\|^2\right]+\frac{1+\rho^2}{1-\rho^2}\alpha_k^2C_f\mathbb{E}\left[\normm{\mathbf{z}_k}_F^2\right],
	\end{align}
	{\small\begin{align}
		\mathbb{E}\left[\left\|\mathbf{y}_{k+1}-\mathbf{1}\otimes\bar{y}_{k+1}\right\|^2\right]
		&\le \frac{1+\rho^2}{2}\mathbb{E}\left[\left\|\mathbf{y}_k-\mathbf{1}\otimes\bar{y}_k\right\|^2\right]+4\frac{1+\rho^2}{1-\rho^2}\left(\beta_k^2\mathbb{E}\left[\left\|\mathbf{G}_k-\mathbf{g}_k\right\|^2\right]\right.\notag\\
		\label{zz-bound-1}&\quad\left.+2C_g\left(4\mathbb{E}\left[ \left\|\mathbf{x}_k-\mathbf{1}\otimes \bar{x}_k\right\|^2\right]+\alpha_k^2C_f\mathbb{E}\left[\normm{\mathbf{z}_k}_F^2\right]\right)+\beta_k^2V_g\right)
		\end{align}}
	and
	{\small\begin{align}
		\mathbb{E}\left[\normm{\mathbf{z}_{k+1}-\mathbf{1}\otimes\bar{z}_{k+1}}^2\right]
		&\le \frac{1+\rho^2}{2}\mathbb{E}\left[\normm{\mathbf{z}_{k}-\mathbf{1}\otimes\bar{z}_{k}}^2\right]+3\frac{1+\rho^2}{1-\rho^2}\left(\gamma_k^2\mathbb{E}\left[\normm{\hat{\mathbf{G}}_k-\nabla\mathbf{g}_k}^2_F\right]\right.\notag\\
		\label{yy-bound}&\quad\left.+2L_g^2\left(4\mathbb{E}\left[ \left\|\mathbf{x}_k-\mathbf{1}\otimes \bar{x}_k\right\|^2\right]+\alpha_k^2C_f\mathbb{E}\left[\normm{\mathbf{z}_k}_F^2\right]\right)+\gamma_k^2V_g^{'}\right).
		\end{align}}
%	where the $C_g$, $C_f$, $V_g$, $V_g^{'}$ and $L_g$ are defined in Assumption \ref{ass-objective}.
\end{lem}

\begin{lem}\label{lem:complex}
Suppose (a) Assumptions \ref{ass-objective}-\ref{ass:matrix} hold; (b) $ \{\beta_k\}$ and $\{\gamma_k\}$ are nonincreasing sequences such that $\beta_1\le 1$ and $\gamma_1\le 1$.
	Denote
	\begin{equation}\label{notation-1-0}
	\begin{aligned}
	 &c_1=\frac{\frac{1-\rho^2}{24}}{32\frac{1+\rho^2}{1-\rho^2}C_g},~c_2=\frac{\frac{1-\rho^2}{24}}{32\frac{1+\rho^2}{1-\rho^2}L_g^2},~c_3=\frac{\frac{1-\rho^2}{24}}{48C_g},~c_4=\frac{\frac{1-\rho^2}{24}}{48pL_g^2},\\
	 &c_5=12C_gC_fc_3+12pL_g^2C_fc_4+8\frac{1+\rho^2}{1-\rho^2}C_gC_fc_1+8\frac{1+\rho^2}{1-\rho^2}L_g^2C_fc_2,\qquad\qquad
	\end{aligned}
	\end{equation}
	\begin{equation}\label{notation-1-1}
	\begin{aligned}
	&a_k=\max\left\{\frac{2+\rho^2}{3},\left(\frac{1+\rho^2}{2}+\frac{3pc_5}{c_2}\alpha_k^2\right),\left((1-\beta_k)^2+6\beta_k^2\right),\right.\\
	&\quad\quad\quad\quad~~~\left.\left((1-\gamma_k)^2+6p\gamma_k^2+\frac{3pc_5}{c_4}\alpha_k^2\right)\right\},\\
	&b_k=3npc_5C_g\alpha_k^2+\left(4\frac{1+\rho^2}{1-\rho^2}V_gc_1+3V_gc_3\right)\beta_k^2+\left(4\frac{1+\rho^2}{1-\rho^2}2V_g^{'}c_2+3V_g^{'}c_4\right)\gamma_k^2,
	\end{aligned}
	\end{equation}
	\begin{equation}\label{notation-1-2}
	\begin{aligned}
	 &V_k=\mathbb{E}\left[\left\|\mathbf{x}_k-\mathbf{1}\otimes\bar{x}_k\right\|^2\right]+c_1\mathbb{E}\left[\left\|\mathbf{y}_k-\mathbf{1}\otimes\bar{y}_k\right\|^2\right]+c_2\mathbb{E}\left[\normm{\mathbf{z}_{k}-\mathbf{1}\otimes\bar{z}_{k}}^2\right]\\
	&\quad+c_3\mathbb{E}\left[\|\mathbf{G}_k-\mathbf{g}_k\|^2\right]+c_4\mathbb{E}\left[\normm{\hat{\mathbf{G}}_k-\nabla\mathbf{g}_k}_F^2\right].\qquad\qquad\qquad\qquad\qquad\qquad
	\end{aligned}
	\end{equation}
	 Then
	\begin{equation}
	\label{V-bound}V_{k+1}\le a_k V_k+b_k.
	\end{equation}	
\end{lem}

Lemma \ref{lem:complex} is a technical result, which characterizes   a recursive inequality relation of the   estimating errors of   hybrid variance reduction method  and
the  dynamic average consensus method.

 We are ready to  present the convergence rate of D-ASCGD.
\begin{thm}\label{thm:rate} Let   $c_1,\cdots, c_5$ be defined in (\ref{notation-1-0}).
	Suppose that (a) Assumptions \ref{ass-objective}-\ref{ass:matrix} hold, (b) $\alpha_k=\frac{s_1}{\sqrt{K}}$,  $\beta_k=\frac{s_2}{\sqrt{K}}$, $\gamma_k=\frac{s_3}{\sqrt{K}}$, where positive constants $s_1,s_2, s_3$ are small enough such that $\alpha_k<1,\beta_k<1,\gamma_k<1$ and
	\begin{equation}\label{step-condi}
	\frac{s_1^2}{K}< \frac{(1-\rho^2)c_2}{6pc_5},~\frac{s_2}{\sqrt{K}}<\frac{2(1+\rho^2)}{7-5\rho^2},~\left(1+\frac{6pL_g(1-\rho^2)}{1+\rho^2}\right)\frac{s_3}{\sqrt{K}}+\frac{3pc_5}{c_4}\frac{s_1^2}{s_3\sqrt{K}}<2.
	\end{equation}
	 Then   for any $i\in\mathcal{V}$,
	\begin{align}\label{rate up}
	&\frac{1}{K}\sum_{k=1}^K\mathbb{E}\left[\|\nabla h(x_{i,k})\|^2\right]\le \frac{8\mathbb{E}\left[h(\bar{x}_1)\right]-\mathbb{E}\left[h(\bar{x}_{K+1})\right]}{s_1\sqrt{K}}+\frac{24pLC_fC_gs_1}{n\sqrt{K}}+\frac{8c_6(V_1+b)}{K(1-a_K)},
	\end{align}
	where  $a_k$ is defined in (\ref{notation-1-1}), $L=C_gL_f + C_f^{1/2}L_g$ and
	\begin{align}
	 \label{b}&b=3npc_5C_gs_1^2+3\left(\frac{1+\rho^2}{1-\rho^2}V_gc_1+V_gc_3\right)s_2^2+3\left(\frac{1+\rho^2}{1-\rho^2}2V_g^{'}c_2+V_g^{'}c_4\right)s_3^2,\\
	&c_6=\max\left\{2\left(\frac{C_g^2L_f^2}{n}+\frac{C_fL_g^2}{n}+\frac{L^2}{8}\right),~\frac{2 C_gL_f^2}{n\min\{c_1,c_3\}},~\frac{\left(4C_f+3LC_f\right)p}{2\min\{c_2,c_4\}n}\right\}.\notag
	\end{align}
\end{thm}
\begin{proof}
	We first provide an upper bound for $\mathbb{E}\left[\|\nabla h(\bar{x}_k)\|^2\right]$.
	Noting that $\nabla h(x)$ is {\small$L\left(\define C_gL_f + C_f^{1/2}L_g\right)$}-smooth \cite{Junyu2019SC},
	\begin{align*}
	h(\bar{x}_{k+1})&\le h(\bar{x}_k)+\langle \nabla h(\bar{x}_k),\bar{x}_{k+1}-\bar{x}_k\rangle+\frac{L}{2}\|\bar{x}_{k+1}-\bar{x}_k\|^2\notag\\
	&=h(\bar{x}_k)-\left\langle \nabla h(\bar{x}_k),\alpha_k\bar{U}_{k+1}\right\rangle+\frac{L}{2}\left\|\alpha_{k}\bar{U}_{k+1}\right\|^2\notag\\
	&=h(\bar{x}_k)-\alpha_{k}\|\nabla h(\bar{x}_k)\|^2+\frac{L}{2}\left\|\alpha_{k}\bar{U}_{k+1}\right\|^2+\left\langle \nabla h(\bar{x}_k),\alpha_{k}\left(\nabla h(\bar{x}_k)-\bar{U}_{k+1}\right)\right\rangle,
	\end{align*}
	where $\bar{U}_{k+1}=\left(\frac{\mathbf{1}^\intercal}{n}\otimes\mathbf{I}_{d}\right)\mathbf{U}_{k+1}$, the first equality follows from the fact that $\bar{x}_{k+1}=\bar{x}_k-\alpha_k\bar{U}_{k+1}$.
	Take expectation on both sides of above inequality,
	\begin{align}
	\mathbb{E}\left[h(\bar{x}_{k+1})\right]
	&\le \mathbb{E}\left[h(\bar{x}_k)\right]-\alpha_{k}\mathbb{E}\left[\|\nabla h(\bar{x}_k)\|^2\right]+\frac{L}{2}\mathbb{E}\left[\left\|\alpha_{k}\bar{U}_{k+1}\right\|^2\right]\notag\\
	\label{grad-bound-1}&\quad+\mathbb{E}\left[\left\langle \nabla h(\bar{x}_k),\alpha_{k}\left(\nabla h(\bar{x}_k)-\bar{U}_{k+1}\right)\right\rangle\right].
	\end{align}
	
	For the third term on the right hand side of (\ref{grad-bound-1}),
	\begin{align*}
	\frac{L}{2}\mathbb{E}\left[\left\|\alpha_{k}\bar{U}_{k+1}\right\|^2\right]&\le \frac{L\alpha_{k}^2}{2n}\mathbb{E}\left[\left\|\mathbf{U}_{k+1}\right\|^2\right]\le \frac{L\alpha_{k}^2}{2n}C_f\mathbb{E}\left[\normm{\mathbf{z}_{k}}_F^2\right],
	\end{align*}
	where the second inequality follows from the definition of  $\mathbf{U}_{k+1}$ and Assumption \ref{ass-objective}(c). For the term $\mathbb{E}\left[\normm{\mathbf{z}_k}_F^2\right]$,  it is  easy to observe that
	\begin{align*}
	 \mathbb{E}\left[\normm{\mathbf{z}_k}_F^2\right]\notag
	&=\mathbb{E}\left[\normm{\mathbf{z}_k-\mathbf{1}\otimes\bar{z}_k+\mathbf{1}\otimes\bar{z}_k-\mathbf{1}\otimes\left(\frac{1}{n}\sum_{j=1}^n\nabla g_j(x_{j,k})\right)+\mathbf{1}\otimes\left(\frac{1}{n}\sum_{j=1}^n\nabla g_j(x_{j,k})\right)}_F^2\right]\notag\\
	&\le 3\left(\mathbb{E}\left[\normm{\mathbf{z}_k-\mathbf{1}\otimes\bar{z}_k}_F^2\right]+n\mathbb{E}\left[\normm{\bar{z}_k-\frac{1}{n}\sum_{j=1}^n\nabla g_j(x_{j,k})}_F^2\right]+\sum_{j=1}^n\mathbb{E}\left[\normm{\nabla g_j(x_{j,k})}_F^2\right]\right)\notag\\
	&\le 3\left(\mathbb{E}\left[\normm{\mathbf{z}_k-\mathbf{1}\otimes\bar{z}_k}_F^2\right]+\mathbb{E}\left[\normm{\hat{\mathbf{G}}_k-\nabla\mathbf{g}_k}_F^2\right]+npC_g\right)\notag\\
	&\le 3p\left(\mathbb{E}\left[\normm{\mathbf{z}_k-\mathbf{1}\otimes\bar{z}_k}^2\right]+\mathbb{E}\left[\normm{\hat{\mathbf{G}}_k-\nabla\mathbf{g}_k}^2\right]+nC_g\right).
	\end{align*}
    Then
    \begin{align}\label{U-bound}
    \frac{L}{2}\mathbb{E}\left[\left\|\alpha_{k}\bar{U}_{k+1}\right\|^2\right]
    &\le 3p\frac{L\alpha_{k}^2}{2n}C_f\left(\mathbb{E}\left[\normm{\mathbf{z}_k-\mathbf{1}\otimes\bar{z}_k}^2\right]+\mathbb{E}\left[\normm{\hat{\mathbf{G}}_k-\nabla\mathbf{g}_k}^2\right]+nC_g\right).
    \end{align}

	For the fourth term on the right hand side of (\ref{grad-bound-1}),
	{\small\begin{align}
		&\mathbb{E}\left[\left\langle \nabla h(\bar{x}_k),\alpha_{k}\left(\nabla h(\bar{x}_k)-\bar{U}_{k+1}\right)\right\rangle\right]\notag\\
		&=\mathbb{E}\left[\left\langle \alpha_{k}\nabla h(\bar{x}_k),P_1+P_2+P_3+P_4+P_5+P_6+P_7\right\rangle\right]\notag\\
		&\le \frac{\alpha_k^2}{2\tau}\mathbb{E}\left[\|\nabla h(\bar{x}_k)\|^2\right]+3\tau\left(\mathbb{E}\left[\|P_1\|^2\right]+\mathbb{E}\left[\|P_2\|^2\right]+\mathbb{E}\left[\|P_3\|^2\right]+\mathbb{E}\left[\|P_4\|^2\right]+\mathbb{E}\left[\|P_5\|^2\right]+\mathbb{E}\left[\|P_6\|^2\right]\right)\notag\\
		\label{grad-bound-3-0}&\quad+\mathbb{E}\left[\left\langle \nabla h(\bar{x}_k),P_7\right\rangle\right],
		\end{align}}
	where  $\tau$ is any positive scalar,
	{\small\begin{align*}
		&P_1=\nabla h(\bar{x}_k)-\frac{1}{n}\sum_{j=1}^n\left(\frac{1}{n}\sum_{j=1}^n\nabla g_j(\bar{x}_k)\right)\nabla f_j\left(\frac{1}{n}\sum_{j=1}^ng_j(x_{j,k})\right),\\
		&P_2=\frac{1}{n}\sum_{j=1}^n\left(\frac{1}{n}\sum_{j=1}^n\left(\nabla g_j(\bar{x}_k)-\nabla g_j(x_{j,k})\right)\right)\nabla f_j\left(\frac{1}{n}\sum_{j=1}^ng_j(x_{j,k})\right),\qquad\qquad\qquad\qquad\quad
		\end{align*}}
	\vspace{-0.5cm}
	{\small\begin{align*}
		&P_3=\frac{1}{n}\sum_{j=1}^n\left(\frac{1}{n}\sum_{j=1}^n\nabla g_j(x_{j,k})\right)\left(\nabla f_i\left(\frac{1}{n}\sum_{j=1}^ng_j(x_{j,k})\right)-\nabla f_j\left(\frac{1}{n}\sum_{j=1}^nG_{j,k}\right)\right),\\
		&P_4=\frac{1}{n}\sum_{j=1}^n\left(\frac{1}{n}\sum_{j=1}^n\nabla g_j(x_{j,k})\right)\left(\nabla f_j(\bar{y}_k)-\nabla f_j(y_{j,k})\right),\qquad\qquad\qquad\qquad\quad\qquad\qquad\qquad
		\end{align*}}
	\vspace{-0.5cm}
	{\small\begin{align*}
		&P_5=\frac{1}{n}\sum_{j=1}^n\left(\frac{1}{n}\sum_{j=1}^n\nabla g_j(x_{j,k})-\frac{1}{n}\sum_{j=1}^n\hat{G}_{j,k}\right)\nabla f_j(y_{j,k}),~P_6=\frac{1}{n}\sum_{j=1}^n\left(\bar{z}_k-z_{j,k}\right)\nabla f_j(y_{j,k}),\\
		&P_7=\frac{1}{n}\sum_{j=1}^nz_{j,k}\left(\nabla f_j(y_{j,k})-\nabla F_j(y_{j,k};\zeta_{j,k+1})\right),
		\end{align*}}\footnote{By the iterations of $y_{i,k}$, $z_{i,k}$ in Algorithm \ref{alg:SIA} and the definitions of $\bar{y}_k$ and $\bar{z}_k$, we have $\bar{y}_k=\frac{1}{n}\sum_{j=1}^nG_{j,k}$ and $\bar{z}_k=\frac{1}{n}\sum_{j=1}^n\hat{G}_{j,k}$}
	the inequality follows from Cauchy-Schwartz inequality and the fact $ab\le \frac{1}{2\tau}a^2+\frac{\tau}{2}b^2$. Defining 
	\begin{equation*}
	\begin{aligned}
	&\mathcal{F}_1=\sigma\left(x_{i,1}, y_{i,1},z_{i,1},G_{i,1},\hat{G}_{i,1}:i\in\mathcal{V}\right),\\
	&\mathcal{F}_k=\sigma\{x_{i,1},y_{i,1},z_{i,1},G_{i,1},\hat{G}_{i,1}, \phi_{i,t},\zeta_{i,t}:i\in\mathcal{V}, 2\le t\le k\}(k\ge2),
	\end{aligned}
	\end{equation*}
	we have $\mathbb{E}\left[\nabla F_j(y_{j,k};\zeta_{j,k+1})\big|\mathcal{F}_k\right]=\nabla f_j(y_{j,k})$ and then the third term on the right hand side of inequality (\ref{grad-bound-3-0}) is equal to 0. Moreover, by Assumption \ref{ass-objective}(a) and (c),
	{\small\begin{align}
		&\mathbb{E}\left[\left\langle \nabla h(\bar{x}_k),\alpha_{k}\left(\nabla h(\bar{x}_k)-\bar{U}_{k+1}\right)\right\rangle\right]\notag\\
		&\le \frac{\alpha_k^2}{2\tau}\mathbb{E}\left[\|\nabla h(\bar{x}_k)\|^2\right]+3\tau\left(\frac{C_g^2L_f^2}{n}+\frac{C_fL_g^2}{n}\right)\mathbb{E}\left[\|\mathbf{x}_k-\mathbf{1}\otimes\bar{x}_k\|^2\right]+\frac{3\tau C_gL_f^2}{n}\mathbb{E}\left[\|\mathbf{g}_k-\mathbf{G}_k\|^2\right]\notag\\
		\label{grad-bound-3}&+\frac{3\tau C_gL_f^2}{n}\mathbb{E}\left[\|\mathbf{y}_k-\mathbf{1}\otimes \bar{y}_k\|^2\right]+\frac{3\tau C_f}{n}\mathbb{E}\left[\normm{\nabla \mathbf{g}_k-\hat{\mathbf{G}}_k}_F^2\right]+\frac{3\tau p C_f}{n}\mathbb{E}\left[\normm{ \mathbf{z}_k-\mathbf{1}\otimes \bar{z}_k}^2\right].
		\end{align}}
	
	Plug (\ref{U-bound}), (\ref{grad-bound-3}) into (\ref{grad-bound-1}) and set $\tau=\frac{2\alpha_k}{3}$,
	{\small\begin{align}
		&\mathbb{E}\left[h(\bar{x}_{k+1})\right]\notag\\
		&\le \mathbb{E}\left[h(\bar{x}_k)\right]-\frac{\alpha_k}{4}\mathbb{E}\left[\|\nabla h(\bar{x}_k)\|^2\right]+2\alpha_k\left(\frac{C_g^2L_f^2}{n}+\frac{C_fL_g^2}{n}\right)\mathbb{E}\left[\|\mathbf{x}_k-\mathbf{1}\otimes\bar{x}_k\|^2\right]\notag\\
		&\quad+\frac{2\alpha_k C_gL_f^2}{n}\mathbb{E}\left[\|\mathbf{g}_k-\mathbf{G}_k\|^2\right]+\frac{2\alpha_k C_gL_f^2}{n}\mathbb{E}\left[\|\mathbf{y}_k-\mathbf{1}\otimes \bar{y}_k\|^2\right]+3p\frac{L\alpha_{k}^2}{n}C_fC_g\notag\\
		\label{h-bound}&\quad+\frac{\left(4 C_f+3pLC_f\alpha_k\right)\alpha_k}{2n}\mathbb{E}\left[\normm{\nabla \mathbf{g}_k-\hat{\mathbf{G}}_k}_F^2\right]+\frac{\left(4C_f+3LC_f\alpha_k\right)p\alpha_k}{2n}\mathbb{E}\left[\normm{ \mathbf{z}_k-\mathbf{1}\otimes \bar{z}_k}^2\right].
		\end{align}}

	Next, we  provide an  upper bound of $\mathbb{E}\left[\nabla h(x_{i,k})\right]$. By the Lipschitz continuity of $\nabla h(\cdot)$,
	\begin{align*}
	\frac{1}{2}\mathbb{E}\left[\|\nabla h(x_{i,k})\|^2\right]&\le \mathbb{E}\left[\|\nabla h(x_{i,k})-\nabla h(\bar{x}_k)\|^2\right]+\mathbb{E}\left[\|\nabla h(\bar{x}_k)\|^2\right]\\
	&\le L^2\mathbb{E}\left[\|x_{i,k}-\bar{x}_k\|^2\right]+\mathbb{E}\left[\|\nabla h(\bar{x}_k)\|^2\right]\\
	&\le L^2\mathbb{E}\left[\|\mathbf{x}_{k}-\mathbf{1}\otimes\bar{x}_k\|^2\right]+\mathbb{E}\left[\|\nabla h(\bar{x}_k)\|^2\right],
	\end{align*}
	where the second inequality follows from the Lipschitz continuity of $\nabla h(\cdot)$.
	Then
	\begin{align*}
	-\mathbb{E}\left[\|\nabla h(\bar{x}_k)\|^2\right]\le -\frac{1}{2}\mathbb{E}\left[\|\nabla h(x_{i,k})\|^2\right]+L^2\mathbb{E}\left[\|\mathbf{x}_{k}-\mathbf{1}\otimes\bar{x}_k\|^2\right].
	\end{align*}
	Substitute the above inequality into (\ref{h-bound}),
	{\small\begin{align}
		&\mathbb{E}\left[h(\bar{x}_{k+1})\right]\notag\\
		&\le \mathbb{E}\left[h(\bar{x}_k)\right]-\frac{\alpha_k}{8}\mathbb{E}\left[\|\nabla h(x_{i,k})\|^2\right]+2\alpha_k\left(\frac{C_g^2L_f^2}{n}+\frac{C_fL_g^2}{n}+\frac{L^2}{8}\right)\mathbb{E}\left[\|\mathbf{x}_k-\mathbf{1}\otimes\bar{x}_k\|^2\right]\notag\\
		&\quad+\frac{2\alpha_k C_gL_f^2}{n}\mathbb{E}\left[\|\mathbf{g}_k-\mathbf{G}_k\|^2\right]+\frac{2\alpha_k C_gL_f^2}{n}\mathbb{E}\left[\|\mathbf{y}_k-\mathbf{1}\otimes \bar{y}_k\|^2\right]+3p\frac{L\alpha_{k}^2}{n}C_fC_g\notag\\
		&\quad+\frac{\left(4 C_f+3pLC_f\alpha_k\right)\alpha_k}{2n}\mathbb{E}\left[\normm{\nabla \mathbf{g}_k-\hat{\mathbf{G}}_k}_F^2\right]+\frac{\left(4C_f+3LC_f\alpha_k\right)p\alpha_k}{2n}\mathbb{E}\left[\normm{ \mathbf{z}_k-\mathbf{1}\otimes \bar{z}_k}^2\right]\notag\\
		\label{grad-bound-5}&\le \mathbb{E}\left[h(\bar{x}_k)\right]-\frac{\alpha_k}{8}\mathbb{E}\left[\|\nabla h(x_{i,k})\|^2\right]+3p\frac{L\alpha_{k}^2}{n}C_fC_g+c_6\alpha_k V_k,
		\end{align}}
	where
	\begin{align*}
	c_6=\max\left\{2\left(\frac{C_g^2L_f^2}{n}+\frac{C_fL_g^2}{n}+\frac{L^2}{8}\right),~\frac{2 C_gL_f^2}{n\min\{c_1,c_3\}},~\frac{\left(4C_f+3LC_f\right)p}{2\min\{c_2,c_4\}n}\right\},
	\end{align*}
	 $c_1, c_3, c_3, c_4$ are defined in (\ref{notation-1-0}) and $V_k$ is defined in (\ref{notation-1-2}). Reordering the terms of (\ref{grad-bound-5}) and summing over $k$ from 1 to $K$,
	\begin{align}
	\sum_{k=1}^K\frac{\alpha_k}{8}\mathbb{E}\left[\|\nabla h(x_{i,k})\|^2\right]
	\label{grad-bound-6}&\le \mathbb{E}\left[h(\bar{x}_1)\right]-\mathbb{E}\left[h(\bar{x}_{K+1})\right]+\frac{3pLC_fC_g}{n}\sum_{k=1}^K\alpha_k^2+c_6\sum_{k=1}^K\alpha_{k}V_k.
	\end{align}
	Note that definitions of $\alpha_k$, $\beta_k$, $\gamma_k$ guarantee $a_k=a<1$,\footnote{For any fixed $K$, $a_k$ is a constant dependent on $K$.} $b_k=b/K$ where $b$ is defined in (\ref{b}). Then by Lemma \ref{lem:complex},
	\begin{equation*}
	V_k\le a_{k-1}V_{k-1}+b_{k-1}=aV_{k-1}+b/K\le \cdots \le a^{k-1} V_1+(b/K)\sum_{t=0}^{k-2} a^t \le a^{k-1} V_1+\frac{b}{K(1-a)}.
	\end{equation*}
	Substitute the above inequality into (\ref{grad-bound-6}) and multiply both sides of (\ref{grad-bound-6}) by $\frac{8}{s_1\sqrt{K}}$,
	{\small		
		\begin{align}
		\frac{1}{K}\sum_{k=1}^K\mathbb{E}\left[\|\nabla h(x_{i,k})\|^2\right]
		&\le \frac{8\mathbb{E}\left[h(\bar{x}_1)\right]-\mathbb{E}\left[h(\bar{x}_{K+1})\right]}{s_1\sqrt{K}}+\frac{24pLC_fC_gs_1}{n\sqrt{K}}+\frac{8c_6}{K}\sum_{k=1}^K\left(a^{k-1} V_1+\frac{b}{K(1-a)}\right)\notag\\
		\label{rate-1}&\le \frac{8\mathbb{E}\left[h(\bar{x}_1)\right]-\mathbb{E}\left[h(\bar{x}_{K+1})\right]}{s_1\sqrt{K}}+\frac{24pLC_fC_gs_1}{n\sqrt{K}}+\frac{8c_6(V_1+b)}{K(1-a)}.
		\end{align}
	}
	The proof is complete.
\end{proof}
By the definitions of $a,b,V_1$, the third term on the right hand side of (\ref{rate-1}) could be expressed more exactly,  
\begin{equation*}
\mathcal{O}\left(\frac{\max\left\{\frac{1}{1-\rho^2},\sqrt{K}\right\}}{K}\right).
\end{equation*}
 Then similar to the more recent work \cite{yang2022bilevel}, 
the order of magnitude of inequality (\ref{rate-1}) is $\mathcal{O}(1/\sqrt{K})$, which  achieves the optimal convergence rate of  stochastic gradient descent  method \cite{Ghadimi2013nonconvex}. 
   The DSBO method \cite{yang2022bilevel} combines the gossip communication with weighted average stochastic approximation and  D-ASCGD combines the  distributed stochastic gradient descent method with hybrid variance reduction method  and dynamic consensus mechanism.

\section{Compressed D-ASCGD method}\label{sec:CD-method}

In recent years, various  techniques have been developed to reduce communication costs \cite{xu2020compressed,Tang2020survey}. They are extensively incorporated into centralized optimization methods \cite{Dan2017QSGD,seide20141,Guyon2017TernGrad} and decentralized methods \cite{Thinh2021random,Thinh2021fast,liu2021linear,Yi2022Nonconvex,Liao2022compGT}.
 This motivates us to   provide an extension of D-ASCGD by combining it with communication compressed method, which reads as follows.

 % However, all the  methods listed above focus on improving the communication efficiency for the  determine or stochastic non-compositional optimization problem.

% We develop a communication compressed version of D-NSCGD method for solving stochastic compositional optimization problem.

\begin{algorithm}[H]
	\caption{$\underline{\text{C}}$ompressed $\underline{\text{D}}$istributed $\underline{\text{A}}$ggregative $\underline{\text{S}}$tochastic $\underline{\text{C}}$ompositional $\underline{\text{G}}$radient $\underline{\text{D}}$escent (CD-ASCGD)\protect\footnotemark: }\label{alg:c-SIA}
	\begin{algorithmic}[1]
		\REQUIRE  initial values {\small$\mathbf{x}_{1} $,$ \mathbf{G}_{1},\hat{\mathbf{G}}_{1},\mathbf{H}_1^{x},\mathbf{H}_1^{y},\mathbf{H}_1^{z}$,$~\mathbf{y}_{1}=\mathbf{G}_{1}, \mathbf{z}_{1}=\hat{\mathbf{G}}_{1}$}; stepsizes {\small$\alpha_k>0$, $\beta_k>0,\gamma_k>0$}; scaling parameters {\small$\alpha_w\in(0,1)$}; nonnegative weight matrix {\small$\mathbf{W}\in \mathbb{R}^{n\times n}$, $\tilde{\mathbf{W}}=\mathbf{W}\otimes\mathbf{I}_d$}    %%input
		%\ENSURE Optimum set $\mathbb{X}$  %%output
		\STATE $\mathbf{H}_1^{x,w}=\tilde{\mathbf{W}}\mathbf{H}_1^{x},\mathbf{H}_1^{y,w}=\tilde{\mathbf{W}}\mathbf{H}_1^{y},\mathbf{H}_1^{z,w}=\tilde{\mathbf{W}}\mathbf{H}_1^{z}$
		\FOR {$k=1,2,\cdots$}
		\STATE $\breve{\mathbf{x}}_k,\breve{\mathbf{x}}_k^w,\mathbf{H}_{k+1}^{x},\mathbf{H}_{k+1}^{x,w}=\text{C}\text{OMM}\left(\mathbf{x}_k,\mathbf{H}_{k}^{x},\mathbf{H}_{k}^{x,w}\right)$
		\STATE $\breve{\mathbf{y}}_k,\breve{\mathbf{y}}_k^w,\mathbf{H}_{k+1}^{y},\mathbf{H}_{k+1}^{y,w}=\text{C}\text{OMM}\left(\mathbf{y}_k,\mathbf{H}_{k}^{y},\mathbf{H}_{k}^{y,w}\right)$
		\STATE $\breve{\mathbf{z}}_k,\breve{\mathbf{z}}_k^w,\mathbf{H}_{k+1}^{z},\mathbf{H}_{k+1}^{z,w}=\text{C}\text{OMM}\left(\mathbf{z}_k,\mathbf{H}_{k}^{z},\mathbf{H}_{k}^{z,w}\right)$
		\STATE For any $i\in\mathcal{V}$, draw $\phi_{i,k+1}\stackrel{iid}\sim P_{\phi_i},~\zeta_{i,k+1}\stackrel{iid}\sim P_{\zeta_i}$, and compute function values $G_i(x_{i,k};\phi_{i,k+1})$, $G_i(x_{i,k+1};\phi_{i,k+1})$ and gradients $\nabla F_i(y_{i,k};\zeta_{i,k+1})$, $\nabla G_i(x_{i,k};\phi_{i,k+1})$ and $\nabla G_i(x_{i,k+1};\phi_{i,k+1})$
		\STATE\label{c-x}$\mathbf{x}_{k+1}=\mathbf{x}_k-\alpha_w\left(\breve{\mathbf{x}}_k-\breve{\mathbf{x}}_k^w\right)-\alpha_k\mathbf{U}_{k+1}$
		\STATE$\mathbf{G}_{k+1}=(1-\beta_k)\left(\mathbf{G}_k-G_{k+1,k}\right)+G_{k+1,k+1}$
		\STATE$\hat{\mathbf{G}}_{k+1}=(1-\gamma_k)\left(\hat{\mathbf{G}}_k-\nabla G_{k+1,k}\right)+\nabla G_{k+1,k+1}$
		\STATE\label{c-y}$\mathbf{y}_{k+1}=\mathbf{y}_k-\alpha_w\left(\breve{\mathbf{y}}_k-\breve{\mathbf{y}}_k^w\right)+\mathbf{G}_{k+1}-\mathbf{G}_k$
		 \STATE\label{c-z}$\mathbf{z}_{k+1}=\mathbf{z}_k-\alpha_w\left(\breve{\mathbf{z}}_k-\breve{\mathbf{z}}_k^w\right)+\hat{\mathbf{G}}_{k+1}-\hat{\mathbf{G}}_k$
		\ENDFOR
	\end{algorithmic}
\end{algorithm}\footnotetext{A complete algorithm description from
the agent’s perspective can be found in Appendix \ref{apd:CD-AS-agent}.}
{\footnotesize{*In Lines 3-5 of Algorithm \ref{alg:c-SIA}, set scaling parameter $\alpha$ and compressor $\mathcal{C}$ as $\alpha_x$ and $\mathcal{C}_1$  for $\mathbf{x}_k$, $\alpha_y$ and $\mathcal{C}_2$ for $\mathbf{y}_k$,  $\alpha_z$ and $\mathcal{C}_3$ for $\mathbf{z}_k$ in procedure $\text{COMM}\left(\mathbf{v},\mathbf{H},\mathbf{H}^{w}\right)$.
}}

 \begin{algorithm}[h]
 	\caption{\textbf{procedure} COMM \cite{liu2021linear} } \label{COMMA}
   \begin{algorithmic}[1]
   	\REQUIRE $\mathbf{v}$, $\mathbf{H}$, $\mathbf{H}^{w}$
 	\STATE\label{c-1}$\mathbf{Q}=\mathbf{Compress}\left(\mathbf{v}-\mathbf{H}\right)$
 	\quad \quad \quad \quad \quad \quad \quad \quad $\triangleright$ Compression
 	\STATE\label{c-2}$\breve{\mathbf{v}}=\mathbf{H}+\mathbf{Q}$
 	\STATE\label{c-3}$\breve{\mathbf{v}}^w=\mathbf{H}^w+\tilde{\mathbf{W}}\mathbf{Q}$ \quad \quad \quad \quad \quad \quad \quad \quad \quad \quad \quad \quad $\triangleright$ Communication
 	\STATE\label{c-4}$\mathbf{H}\longleftarrow(1-\alpha)\mathbf{H}+\alpha\breve{\mathbf{v}}$
 	\STATE\label{c-5}$\mathbf{H}^w\longleftarrow(1-\alpha)\mathbf{H}^w+\alpha\breve{\mathbf{v}}^w$
 	\STATE\textbf{Return:} $\breve{\mathbf{v}},\breve{\mathbf{v}}^w,\mathbf{H},\mathbf{H}^w$\\
 \end{algorithmic}
\end{algorithm}

%We incorporate the communication compression into D-ASCGD through an compression procedure developed in \cite{liu2021linear}, namely \textbf{procedure} $\text{COMM}\left(\mathbf{x},\mathbf{H},\mathbf{H}^{w}\right)$ in Algorithm \ref{alg:c-SIA}. In what follows,

  We first explain how  $\text{COMM}\left(\mathbf{v},\mathbf{H},\mathbf{H}^{w}\right)$ in Algorithm \ref{COMMA} works by taking  decision variable $\mathbf{x}_{k}$ as an example.
With inputs including decision variable $\mathbf{x}_{k}$ and its two auxiliary variables $\mathbf{H}_{k}^{x}$ and $\mathbf{H}_{k}^{x,w}$, $\text{COMM}\left(\mathbf{v},\mathbf{H},\mathbf{H}^{w}\right)$ compress
$\mathbf{x}_{k}-\mathbf{H}_{k}^{x}$  as a low-bit variable $\mathbf{Q}$. Next, all agents communicate low-bit variable $\mathbf{Q}$ with their neighbors and obtain $\breve{\mathbf{x}}_k^w=\mathbf{H}_k^{x,w}+\tilde{\mathbf{W}}\mathbf{Q}$. Auxiliary variables $\mathbf{H}_{k}^{x}$ and $\mathbf{H}_{k}^{x,w}$ are updated in Lines \ref{c-4} and  \ref{c-5} to improve the stability of the compression procedure.

%It is worth pointing out that
%\begin{equation*}
%\breve{\mathbf{x}}_k^w=\tilde{\mathbf{W}}\breve{\mathbf{x}}_k
%\end{equation*}
%holds \cite{liu2021linear}, which means all gents can share their decision variable with

Now we explain the relationship between the  Algorithm \ref{alg:c-SIA} and   Algorithm \ref{alg:SIA}.
Denoting $\mathbf{E}_{k+1}^x\define\mathbf{x}_k-\breve{\mathbf{x}}_k$,  the iteration $x_k$  in line \ref{c-x} can be reformulated as
\begin{align}
\label{alg:c-x}&\mathbf{x}_{k+1}
=\left[(1-\alpha_w)\mathbf{I}_{nd}+\alpha_w\tilde{\mathbf{W}}\right]\mathbf{x}_k-\alpha_k\mathbf{U}_{k+1}+\alpha_w\left[\mathbf{I}_{nd}-\tilde{\mathbf{W}}\right]\mathbf{E}_{k+1}^x.
\end{align}
Compared with the iterations of  $\mathbf{x}_k$ in  Algorithm \ref{alg:SIA},  there   is a term $\alpha_w\left[\mathbf{I}_{nd}-\tilde{\mathbf{W}}\right]\mathbf{E}_{k+1}^x$ induced by
the compress errors $\mathbf{E}_{k+1}^x$.
 Similarly,  denote $\mathbf{E}_{k+1}^y\define \mathbf{y}_k-\breve{\mathbf{y}}_k$ and $\mathbf{E}_{k+1}^{z}\define \mathbf{z}_k-\breve{\mathbf{z}}_k$, the iterations $\mathbf{y}_{k+1}$ and $\mathbf{z}_{k+1}$
 in lines \ref{c-y} and \ref{c-z} of Algorithm \ref{alg:c-SIA} can be reformulated as
\begin{align}
\label{alg:c-y}&\mathbf{y}_{k+1}
=\left[(1-\alpha_w)\mathbf{I}_{nd}+\alpha_w\tilde{\mathbf{W}}\right]\mathbf{y}_k+\mathbf{G}_{k+1}-\mathbf{G}_k+\alpha_w\left[\mathbf{I}_{nd}-\tilde{\mathbf{W}}\right]\mathbf{E}_{k+1}^y,\\
\label{alg:c-z}&\mathbf{z}_{k+1}
=\left[(1-\alpha_w)\mathbf{I}_{nd}+\alpha_w\tilde{\mathbf{W}}\right]\mathbf{z}_k+\hat{\mathbf{G}}_{k+1}-\hat{\mathbf{G}}_k+\alpha_w\left[\mathbf{I}_{nd}-\tilde{\mathbf{W}}\right]\mathbf{E}_{k+1}^{z}.
\end{align}
Again, the difference between  (\ref{alg:c-y})-(\ref{alg:c-z}) and
  the iterations of  $\mathbf{y}_{k+1}, \mathbf{z}_{k+1}$ in  Algorithm \ref{alg:SIA},  are the terms induced by  compress errors  $\mathbf{E}_{k+1}^y$ and $\mathbf{E}_{k+1}^y$.

In the next, we establish the convergence  of CD-ASCGD.  The following conditions on  the compressors are needed.

\begin{ass}\label{ass:compressor}
The compressors $\mathcal{C}_1:\mathbb{R}^d\longrightarrow \mathbb{R}^d$, $\mathcal{C}_2:\mathbb{R}^p\longrightarrow \mathbb{R}^p$ and $\mathcal{C}_3:\mathbb{R}^{d\times p}\longrightarrow \mathbb{R}^{d\times p}$ satisfy
\begin{equation}
\label{c1}\mathbb{E}\left[\left\|\frac{\mathcal{C}_1(x)}{r_1}-x\right\|^2\right]\le (1-\psi_1)\|x\|^2,\quad \forall x\in\mathbb{R}^d,
\end{equation}
\begin{equation}
\label{c2}\mathbb{E}\left[\left\|\frac{\mathcal{C}_2(z)}{r_2}-z\right\|^2\right]\le (1-\psi_2)\|z\|^2,\quad \forall z\in\mathbb{R}^p,
\end{equation}
\begin{equation}
\label{c3}\mathbb{E}\left[\normm{\frac{\mathcal{C}_3(y)}{r_3}-y}_F^2\right]\le (1-\psi_3)\normm{y}_F^2,\quad \forall y\in\mathbb{R}^{d\times  p}
\end{equation}
for some constants $\psi_1,\psi_2,\psi_3\in (0,1]$ and $r_1,r_1,r_3\in (0,+\infty)$. Here $\mathbb{E}[\cdot]$ denotes the expectation over the internal randomness of the stochastic compression operator.
\end{ass}
For  vector $x\in\mathbb{R}^d$ or $z\in\mathbb{R}^p$, the class of compressors satisfying (\ref{c1}) or (\ref{c2})  is broad,   such as random quantization \cite{suresh17Limited,seide20141}, sparsification \cite{lin2018deep,Nikita2019sketch},  the norm-sign
compressor \cite{Yi2022Nonconvex,Liao2022compGT}. For matrix $y\in\mathbb{R}^{d\times p}$, we may construct the $dp$-length vector by stacking up the columns of $y$ and then implement predetermined compressors to compress the new constructed vector.

Similar to the analysis of D-ASCGD, we first provide the upper bounds of the local inner-level function estimating errors and the consensus errors of the CD-ASCGD in Lemmas \ref{lem:c-tracking} and \ref{lem:c-consensus} respectively.

\begin{lem}\label{lem:c-tracking}
Under Assumptions \ref{ass-objective}-\ref{ass:compressor},
\begin{align*}
&\mathbb{E}\left[\|\mathbf{G}_{k+1}- \mathbf{g}_{k+1}\|^2\right]\notag\\
&\le(1-\beta_k)^2\mathbb{E}\left[\|\mathbf{G}_k-\mathbf{g}_k\|^2\right]+3\beta_k^2V_g+72(1-\beta_k)^2C_g\alpha_w^2\breve{r}_1\mathbb{E}\left[\left\|\mathbf{x}_k-\mathbf{H}_k^x\right\|^2\right]\notag\\
&\quad+18(1-\beta_k)^2C_g\alpha_k^2C_f\mathbb{E}\left[\normm{\mathbf{z}_k}_F^2\right]+72(1-\beta_k)^2C_g\alpha_w^2\mathbb{E}\left[ \left\|\mathbf{x}_k-\mathbf{1}\otimes \bar{x}_k\right\|^2\right]
\end{align*}
and
\begin{align}
&\mathbb{E}\left[\normm{\hat{\mathbf{G}}_{k+1}- \nabla\mathbf{g}_{k+1}}_F^2\right]\notag\\
&\le(1-\gamma_k)^2\mathbb{E}\left[\normm{\hat{\mathbf{G}}_k-\nabla\mathbf{g}_k}_F^2\right]+3\gamma_k^2V_g^{'}+72(1-\gamma_k)^2pL_g^2\alpha_w^2\breve{r}_1\mathbb{E}\left[\left\|\mathbf{x}_k-\mathbf{H}_k^x\right\|^2\right]\notag\\
&\quad+18(1-\gamma_k)^2pL_g^2\alpha_k^2C_f\mathbb{E}\left[\normm{\mathbf{z}_k}_F^2\right]+72(1-\gamma_k)^2pL_g^2\alpha_w^2\mathbb{E}\left[ \left\|\mathbf{x}_k-\mathbf{1}\otimes \bar{x}_k\right\|^2\right],
\end{align}
where $\breve{r}_1=2r_1^2(1-\psi_1)+2\left(1-r_1\right)^2$, $\psi_1$ and $r_1$ are defined in Assumption \ref{ass:compressor}.
\end{lem}

\begin{lem}\label{lem:c-consensus}
	Suppose (a) Assumptions \ref{ass-objective}-\ref{ass:compressor} hold; (b) $ \{\beta_k\}$ and $\{\gamma_k\}$ are nonincreasing sequences such that $\beta_1\le 1$ and $\gamma_1\le 1$. Then
	\begin{align*}
	\mathbb{E}\left[\left\|\mathbf{x}_{k+1}-\mathbf{1}\otimes\bar{x}_{k+1}\right\|^2\right]
	&\le \frac{1+\rho_w^2}{2}\mathbb{E}\left[\left\|\mathbf{x}_k-\mathbf{1}\otimes\bar{x}_k\right\|^2\right]+\frac{1+\rho_w^2}{1-\rho_w^2}\alpha_k^2C_f\mathbb{E}\left[\normm{\mathbf{z}_k}_F^2\right]\notag\\
	&\quad+4\frac{1+\rho_w^2}{1-\rho_w^2}\alpha_w^2\breve{r}_1\mathbb{E}\left[\left\|\mathbf{x}_k-\mathbf{H}_k^x\right\|^2\right],
	\end{align*}
	{\small\begin{align*}
		&\mathbb{E}\left[\left\|\mathbf{y}_{k+1}-\mathbf{1}\otimes\bar{y}_{k+1}\right\|^2\right]\notag\\
		&\le \frac{1+\rho_w^2}{2}\mathbb{E}\left[\left\|\mathbf{y}_k-\mathbf{1}\otimes\bar{y}_k\right\|^2\right]+4\frac{1+\rho_w^2}{1-\rho_w^2}\left(\beta_k^2\mathbb{E}\left[\left\|\mathbf{G}_k-\mathbf{g}_k\right\|^2\right]+\beta_k^2V_g\right.\notag\\
		&\quad\left.+3C_g\left(4\alpha_w^2\mathbb{E}\left[ \left\|\mathbf{x}_k-\mathbf{1}\otimes \bar{x}_k\right\|^2\right]+\alpha_k^2C_f\mathbb{E}\left[\normm{\mathbf{z}_k}_F^2\right]\right)\right)\notag\\
		 &\quad+48C_g\frac{1+\rho_w^2}{1-\rho_w^2}\alpha_w^2\breve{r}_1\mathbb{E}\left[\left\|\mathbf{x}_k-\mathbf{H}_k^x\right\|^2\right]+4\frac{1+\rho_w^2}{1-\rho_w^2}\alpha_w^2\breve{r}_2\mathbb{E}\left[\left\|\mathbf{y}_k-\mathbf{H}_k^y\right\|^2\right]
		\end{align*}}
	and
	{\small\begin{align*}
		&\mathbb{E}\left[\normm{\mathbf{z}_{k+1}-\mathbf{1}\otimes\bar{z}_{k+1}}^2\right]\notag\\
		&\le \frac{1+\rho_w^2}{2}\mathbb{E}\left[\normm{\mathbf{z}_{k}-\mathbf{1}\otimes\bar{z}_{k}}^2\right]+4\frac{1+\rho_w^2}{1-\rho_w^2}\left(\gamma_k^2\mathbb{E}\left[\normm{\hat{\mathbf{G}}_k-\nabla\mathbf{g}_k}^2_F\right]+\gamma_k^2V_g^{'}\right.\notag\\
		&\quad\left.+3L_g^2\left(4\alpha_w^2\mathbb{E}\left[ \left\|\mathbf{x}_k-\mathbf{1}\otimes \bar{x}_k\right\|^2\right]+\alpha_k^2C_f\mathbb{E}\left[\normm{\mathbf{z}_k}_F^2\right]\right)\right)\notag\\
		&\quad +48L_g^2\frac{1+\rho_w^2}{1-\rho_w^2}\alpha_w^2\breve{r}_1\mathbb{E}\left[\left\|\mathbf{x}_k-\mathbf{H}_k^x\right\|^2\right]+4\frac{1+\rho_w^2}{1-\rho_w^2}\alpha_w^2\breve{r}_3\mathbb{E}\left[\normm{\mathbf{z}_k-\mathbf{H}_k^{z}}^2\right],
		\end{align*}}
	where $\rho_w\define \normm{(1-\alpha_w)\mathbf{I}_{nd}+\alpha_w\tilde{\mathbf{W}}-\frac{\mathbf{1}\mathbf{1}^\intercal}{n}\otimes\mathbf{I}_d}$, $\breve{r}_1=2r_1^2(1-\psi_1)+2\left(1-r_1\right)^2$, $\breve{r}_2=2r_2^2(1-\psi_2)+2\left(1-r_2\right)^2$, $\breve{r}_3=2r_3^2(1-\psi_3)+2\left(1-r_3\right)^2$, and $\psi_1,\psi_2,\psi_3,r_1,r_1,r_3$ are defined in Assumption \ref{ass:compressor}.
\end{lem}

The following lemma characterizes the upper bound of  compression errors.
\begin{lem}\label{lem:compression}
	Let $\alpha_x\in\left(0,\frac{1}{r_1}\right)$, $\alpha_y\in\left(0,\frac{1}{r_2}\right)$, $\alpha_z\in\left(0,\frac{1}{r_3}\right)$ and
	\begin{equation*}
	\alpha_w\in \bigg(0,~ \min\left\{\sqrt{\frac{(\alpha_xr_1\psi_1)^3}{24\breve{r}_1\left(1+\alpha_xr_1\psi_1\right)}},\sqrt{\frac{(\alpha_yr_2\psi_2)^3}{24\breve{r}_2\left(1+\alpha_yr_2\psi_2\right)}},\sqrt{\frac{(\alpha_zr_3\psi_3)^3}{24\breve{r}_3\left(1+\alpha_zr_3\psi_3\right)}}\right\}\bigg],
	\end{equation*}
	where $\breve{r}_1,\breve{r}_2$ and $\breve{r}_3$ are defined in Lemma \ref{lem:c-consensus}, $\psi_1,\psi_2,\psi_3,r_1,r_1,r_3$ are defined in Assumption \ref{ass:compressor}.
	Suppose (a) Assumptions \ref{ass-objective}-\ref{ass:compressor} hold; (b) $ \{\beta_k\}$ and $\{\gamma_k\}$ are nonincreasing sequences such that $\beta_1\le 1$ and $\gamma_1\le 1$. Then 	
	\begin{align}\label{comp-x-0}
	\mathbb{E}\left[\left\|\mathbf{x}_{k+1}-\mathbf{H}_{k+1}^x\right\|^2\right]
	 &\le\left(1-\frac{(\alpha_xr_1\psi_1)^2}{2}\right)\mathbb{E}\left[\left\|\mathbf{x}_{k}-\mathbf{H}_{k}^x\right\|^2\right]+\frac{1+\alpha_xr_1\psi_1}{\alpha_xr_1\psi_1}3\alpha_k^2C_f\mathbb{E}\left[\normm{\mathbf{z}_k}_F^2\right]\notag\\
	&\quad+12\frac{1+\alpha_xr_1\psi_1}{\alpha_xr_1\psi_1}\alpha_w^2\mathbb{E}\left[ \left\|\mathbf{x}_k-\mathbf{1}\otimes \bar{x}_k\right\|^2\right],
	\end{align}
	{\small\begin{align}\label{comp-z-0}
		\mathbb{E}\left[\left\|\mathbf{y}_{k+1}-\mathbf{H}_{k+1}^y\right\|^2\right]
		&\le \left(1-\frac{(\alpha_yr_2\psi_2)^2}{2}\right)\mathbb{E}\left[\left\|\mathbf{y}_{k}-\mathbf{H}_{k}^y\right\|^2\right]+9C_g\frac{1+\alpha_yr_2\psi_2}{\alpha_yr_2\psi_2}\alpha_k^2C_f\mathbb{E}\left[\normm{\mathbf{z}_k}_F^2\right]\notag\\
		&\quad+36C_g\frac{1+\alpha_yr_2\psi_2}{\alpha_yr_2\psi_2}\left(4\alpha_w^2\mathbb{E}\left[ \left\|\mathbf{x}_k-\mathbf{1}\otimes \bar{x}_k\right\|^2\right]+4\alpha_w^2\breve{r}_1\mathbb{E}\left[\left\|\mathbf{x}_k-\mathbf{H}_k^x\right\|^2\right]\right)\notag\\
		&\quad +12\frac{1+\alpha_yr_2\psi_2}{\alpha_yr_2\psi_2}\alpha_w^2\mathbb{E}\left[ \left\|\mathbf{y}_k-\mathbf{1}\otimes \bar{y}_k\right\|^2\right]+36C_g\frac{1+\alpha_yr_2\psi_2}{\alpha_yr_2\psi_2}\alpha_k^2C_f\mathbb{E}\left[\normm{\mathbf{y}_k}_F^2\right]\notag\\
		&\quad+12\frac{1+\alpha_yr_2\psi_2}{\alpha_yr_2\psi_2}\beta_k^2\mathbb{E}\left[\left\|\mathbf{G}_k-\mathbf{g}_k\right\|^2\right]+12\frac{1+\alpha_yr_2\psi_2}{\alpha_yr_2\psi_2}\beta_k^2V_g
		\end{align}}
	and
	{\small\begin{align}\label{comp-y-0}
		\mathbb{E}\left[\left\|\mathbf{z}_{k+1}-\mathbf{H}_{k+1}^{z}\right\|^2\right]
		&\le \left(1-\frac{(\alpha_zr_3\psi_3)^2}{2}\right)\mathbb{E}\left[\normm{\mathbf{z}_{k}-\mathbf{H}_{k}^{z}}^2\right]+9L_g^2\frac{1+\alpha_zr_3\psi_3}{\alpha_zr_3\psi_3}\alpha_k^2C_f\mathbb{E}\left[\normm{\mathbf{z}_k}_F^2\right]\notag\\
		&\quad+9L_g^2\frac{1+\alpha_zr_3\psi_3}{\alpha_zr_3\psi_3}\left(4\alpha_w^2\mathbb{E}\left[ \left\|\mathbf{x}_k-\mathbf{1}\otimes \bar{x}_k\right\|^2\right]+4\alpha_w^2\breve{r}_1\mathbb{E}\left[\left\|\mathbf{x}_k-\mathbf{H}_k^x\right\|^2\right]\right)\notag\\
		&\quad +12\frac{1+\alpha_zr_3\psi_3}{\alpha_zr_3\psi_3}\alpha_w^2\mathbb{E}\left[ \normm{\mathbf{z}_k-\mathbf{1}\otimes \bar{z}_k}^2\right]+36L_g^2\frac{1+\alpha_zr_3\psi_3}{\alpha_zr_3\psi_3}\alpha_k^2C_f\mathbb{E}\left[\normm{\mathbf{y}_k}_F^2\right]\notag\\
		&\quad+12\frac{1+\alpha_zr_3\psi_3}{\alpha_zr_3\psi_3}\gamma_k^2\mathbb{E}\left[\normm{\hat{\mathbf{G}}_k-\nabla\mathbf{g}_k}^2_F\right]+12\frac{1+\alpha_zr_3\psi_3}{\alpha_zr_3\psi_3}\gamma_k^2V_g^{'}.
		\end{align}}
\end{lem}

\begin{lem}\label{lem:c-complex}
	Denote
	{\small\begin{equation}\label{c-notation-1}
		\begin{aligned}
		 &\breve{c}_1=\frac{\frac{1-\rho_w^2}{42}}{48\frac{1+\rho_w^2}{1-\rho_w^2}C_g\alpha_w^2},~\breve{c}_2=\frac{\frac{1-\rho_w^2}{42}}{48\frac{1+\rho_w^2}{1-\rho_w^2}L_g^2\alpha_w^2},~\breve{c}_3=\frac{\frac{1-\rho_w^2}{42}}{72C_g\alpha_w^2},~\breve{c}_4=\frac{\frac{1-\rho_w^2}{42}}{72pL_g^2\alpha_w^2},\\
		 &\breve{c}_5=\frac{\frac{1-\rho_w^2}{42}}{12\frac{1+\alpha_xr_1\psi_1}{\alpha_xr_1\psi_1}\alpha_w^2},~\breve{c}_6=\frac{\frac{1-\rho_w^2}{42}}{144C_g\frac{1+\alpha_yr_2\psi_2}{\alpha_yr_2\psi_2}\alpha_w^2},~\breve{c}_7=\frac{\frac{1-\rho_w^2}{42}}{144L_g^2\frac{1+\alpha_zr_3\psi_3}{\alpha_zr_3\psi_3}\alpha_w^2},\\
		&\breve{c}_8=\left(\frac{1+\rho_w^2}{1-\rho_w^2}C_f+
		12\frac{1+\rho_w^2}{1-\rho_w^2}C_gC_f\breve{c}_1
		+12\frac{1+\rho_w^2}{1-\rho_w^2}L_g^2C_f\breve{c}_2+18C_gC_f\breve{c}_3+18pL_g^2C_f\breve{c}_4\right.\\
		 &\quad\quad\left.+\frac{1+\alpha_xr_1\psi_1}{\alpha_xr_1\psi_1}3C_f\breve{c}_5+36C_g\frac{1+\alpha_yr_2\psi_2}{\alpha_yr_2\psi_2}C_f\breve{c}_6+36L_g^2\frac{1+\alpha_zr_3\psi_3}{\alpha_zr_3\psi_3}C_f\breve{c}_7\right),
		\end{aligned}
		\end{equation}	}
	{\small\begin{equation}\label{c-notation-0}
		\begin{aligned}
		\breve{a}_k&=\max\left\{\frac{2+\rho_w^2}{3}+\frac{3p\breve{c}_8\alpha_k^2}{\breve{c}_2},~(1-\beta_k)^2+6\beta_k^2,~1-\frac{(\alpha_xr_1\psi_1)^2}{4},\right.\\
		&\quad\quad\quad\quad~~~\left.(1-\gamma_k)^2+6p\gamma_k^2+\frac{3p\breve{c}_8\alpha_k^2}{\breve{c}_4},~1-\frac{(\alpha_yr_2\psi_2)^2}{4},~1-\frac{(\alpha_zr_3\psi_3)^2}{4}\right\},\quad~~\\
		\breve{b}_k&=3np\breve{c}_8C_g\alpha_k^2+\left(4\frac{1+\rho_w^2}{1-\rho_w^2}\breve{c}_1V_g+3\breve{c}_3V_g+12\frac{1+\alpha_yr_2\psi_2}{\alpha_yr_2\psi_2}V_g\right)\beta_k^2\\
		&\quad+\left(4\frac{1+\rho_w^2}{1-\rho_w^2}\breve{c}_2V_g^{'}+3\breve{c}_4V_g^{'}+12\frac{1+\alpha_zr_3\psi_3}{\alpha_zr_3\psi_3}V_g^{'}\right)\gamma_k^2,\quad~~
		\end{aligned}
		\end{equation}	}
	{\small\begin{equation}\label{c-notation-2}
		\begin{aligned}
		 \breve{V}_k&=\mathbb{E}\left[\left\|\mathbf{x}_k-\mathbf{1}\otimes\bar{x}_k\right\|^2\right]+\breve{c}_1\mathbb{E}\left[\left\|\mathbf{y}_k-\mathbf{1}\otimes\bar{y}_k\right\|^2\right]+\breve{c}_2\mathbb{E}\left[\normm{\mathbf{z}_{k}-\mathbf{1}\otimes\bar{z}_{k}}^2\right]\\
		 &\quad+\breve{c}_3\mathbb{E}\left[\|\mathbf{G}_k-\mathbf{g}_k\|^2\right]+\breve{c}_4\mathbb{E}\left[\normm{\hat{\mathbf{G}}_k-\nabla\mathbf{g}_k}_F^2\right]+\breve{c}_5\mathbb{E}\left[\left\|\mathbf{x}_{k+1}-\mathbf{H}_{k+1}^x\right\|^2\right]\\
		 &\quad+\breve{c}_6\mathbb{E}\left[\left\|\mathbf{y}_{k+1}-\mathbf{H}_{k+1}^y\right\|^2\right]+\breve{c}_7\mathbb{E}\left[\left\|\mathbf{z}_{k+1}-\mathbf{H}_{k+1}^{z}\right\|^2\right].\qquad\qquad\qquad\qquad\qquad\qquad
		\end{aligned}
		\end{equation}	}
	Suppose (a) Assumptions \ref{ass-objective}$\sim$\ref{ass:compressor} hold, (b) $\alpha_x$, $\alpha_z$, $\alpha_y$  $\alpha_k$, $\beta_k$, $\gamma_k$  and $a_k$ are positive and satisfy 
	\begin{align*}
	&\alpha_x<\frac{1}{r_1},~\alpha_z<\frac{1}{r_2},~\alpha_y<\frac{1}{r_3},~
	\alpha_k<1,~\beta_k<1,~\gamma_k<1,~a_k<1,
	\end{align*}
	(c)  $\alpha_w$ is positive and satisfies
	{\small\begin{align*}
		&4\frac{1+\rho_w^2}{1-\rho_w^2}\alpha_w^2\le \frac{1-\rho_w^2}{6},~\alpha_w\le \min\left\{\sqrt{\frac{(\alpha_xr_1\psi_1)^3}{24\breve{r}_1\left(1+\alpha_xr_1\psi_1\right)}},\sqrt{\frac{(\alpha_zr_2\psi_2)^3}{48\breve{r}_2\left(1+\alpha_zr_2\psi_2\right)}},\sqrt{\frac{(\alpha_yr_3\psi_3)^3}{48\breve{r}_3\left(1+\alpha_yr_3\psi_3\right)}}\right\},\\
		&4\alpha_w^2\left(\frac{1+\rho_w^2}{1-\rho_w^2}\frac{\breve{r}_1}{\breve{c}_5}+12\left(1-\beta_k\right)^2C_g\frac{1+\rho_w^2}{1-\rho_w^2}\breve{r}_1\frac{\breve{c}_1}{\breve{c}_5}+12\left(1-\gamma_k\right)^2L_g^2\frac{1+\rho_w^2}{1-\rho_w^2}\breve{r}_1\frac{\breve{c}_2}{\breve{c}_5}+18(1-\beta_k)^2C_g\breve{r}_1\frac{\breve{c}_3}{\breve{c}_5}\right.\\
		&\quad\left.+18(1-\gamma_k)^2pL_g^2\breve{r}_1\frac{\breve{c}_4}{\breve{c}_5}
		+36C_g\frac{1+\alpha_yr_2\psi_2}{\alpha_yr_2\psi_2}\breve{r}_1\frac{\breve{c}_6}{\breve{c}_5}
		+36L_g^2\frac{1+\alpha_zr_3\psi_3}{\alpha_zr_3\psi_3}\breve{r}_1\frac{\breve{c}_7}{\breve{c}_5}\right)\le \frac{(\alpha_xr_1\psi_1)^2}{4}.
		\end{align*}}
	 Then
	\begin{equation*}
	\breve{V}_{k+1}\le \breve{a}_k \breve{V}_k+\breve{b}_k.
	\end{equation*}	
\end{lem}

Similar to Lemma \ref{lem:complex}, Lemma \ref{lem:compression} is
 a technical result, which characterizes   a recursive inequality relation of the compression errors and the   estimating errors of   hybrid variance reduction method  and
the  dynamic average consensus method.

We are ready to study the  convergence rate of CD-ASCGD.
\begin{thm}\label{thm:c-rate}
	Let $\alpha_k=\frac{s_1}{\sqrt{K}}$,  $\beta_k=\frac{s_2}{\sqrt{K}}$, $\gamma_k=\frac{s_3}{\sqrt{K}}$. Under the conditions of Lemma \ref{lem:c-complex}, 
	\begin{align*}
	&\frac{1}{K}\sum_{k=1}^K\mathbb{E}\left[\|\nabla h(x_{i,k})\|^2\right]\le \frac{8\mathbb{E}\left[h(\bar{x}_1)\right]-\mathbb{E}\left[h(\bar{x}_{K+1})\right]}{s_1\sqrt{K}}+\frac{24pLC_fC_gs_1}{n\sqrt{K}}+\frac{8\breve{c}_9\left( \breve{V}_1+\breve{b}\right)}{K(1-\breve{a}_K)}, ~ \forall i\in\mathcal{V},
	\end{align*}
	where $L=C_gL_f + C_f^{1/2}L_g$,
	\begin{align}
	 \label{c-b}&\breve{b}=3np\breve{c}_8C_gs_1^2+\left(4\frac{1+\rho_w^2}{1-\rho_w^2}\breve{c}_1V_g+3\breve{c}_3V_g+12\frac{1+\alpha_yr_2\psi_2}{\alpha_yr_2\psi_2}V_g\right)s_2^2\\
	 &\quad+\left(4\frac{1+\rho_w^2}{1-\rho_w^2}\breve{c}_2V_g^{'}+3\breve{c}_4V_g^{'}+12\frac{1+\alpha_zr_3\psi_3}{\alpha_zr_3\psi_3}V_g^{'}\right)s_3^2,\\
	&\breve{c}_9=\max\left\{2\left(\frac{C_g^2L_f^2}{n}+\frac{C_fL_g^2}{n}+\frac{L^2}{8}\right),~\frac{2 C_gL_f^2}{n\min\{\breve{c}_1,\breve{c}_3\}},~\frac{\left(4C_f+3LC_f\right)p}{2\min\{\breve{c}_2,\breve{c}_4\}n}\right\}.\notag
	\end{align}
\end{thm}
\begin{proof}
	Similar to the analysis of (\ref{grad-bound-5}) in the proof of Theorem \ref{thm:rate}, we have
	{\small\begin{align}\label{c-grad-bound-5}
		\mathbb{E}\left[h(\bar{x}_{k+1})\right]&\le \mathbb{E}\left[h(\bar{x}_k)\right]-\frac{\alpha_k}{8}\mathbb{E}\left[\|\nabla h(x_{i,k})\|^2\right]+3p\frac{L\alpha_{k}^2}{n}C_fC_g+\breve{c}_9\alpha_k \breve{V}_k, 
		\end{align}}
	where
	\begin{align*}
	\breve{c}_9=\max\left\{2\left(\frac{C_g^2L_f^2}{n}+\frac{C_fL_g^2}{n}+\frac{L^2}{8}\right),~\frac{2 C_gL_f^2}{n\min\{\breve{c}_1,\breve{c}_3\}},~\frac{\left(4C_f+3LC_f\right)p}{2\min\{\breve{c}_2,\breve{c}_4\}n}\right\},
	\end{align*}
	$\breve{c}_1,\cdots, \breve{c}_4$ are defined in (\ref{c-notation-1}), $\breve{V}_k$ is defined in (\ref{c-notation-2}). Reordering the terms of (\ref{c-grad-bound-5}) and summing over $k$ from 1 to $K$,
	\begin{align}
	\sum_{k=1}^K\frac{\alpha_k}{8}\mathbb{E}\left[\|\nabla h(x_{i,k})\|^2\right]
	\label{c-grad-bound-6}&\le \mathbb{E}\left[h(\bar{x}_1)\right]-\mathbb{E}\left[h(\bar{x}_{K+1})\right]+\frac{3pLC_fC_g}{n}\sum_{k=1}^K\alpha_k^2+\breve{c}_9\sum_{k=1}^K\alpha_{k}\breve{V}_k.
	\end{align}
	Note that definitions of $\alpha_k$, $\beta_k$, $\gamma_k$ could guarantee $\breve{a}_k=\breve{a}<1$,\footnote{$\breve{a}$ is dependent on $K$.} $\breve{b}_k=\breve{b}/K$ where $\breve{b}$ is defined in (\ref{c-b}). Then by Lemma \ref{lem:complex},
	\begin{equation*}
	\breve{V}_k\le \breve{a}_{k-1}\breve{V}_{k-1}+b_{k-1}=\breve{a}\breve{V}_{k-1}+\breve{a}/K\le \cdots \le \breve{a}^{k-1} \breve{V}_1+(\breve{a}/K)\sum_{t=0}^{k-2} \breve{a}^t \le \breve{a}^{k-1} \breve{V}_1+\frac{\breve{a}}{K(1-\breve{a})}.
	\end{equation*}
	Substituting the above inequality in to (\ref{c-grad-bound-6}) and multiplying both sides of (\ref{c-grad-bound-6}) by $\frac{8}{s_1\sqrt{K}}$,
	\begin{align}
	\frac{1}{K}\sum_{k=1}^K\mathbb{E}\left[\|\nabla h(x_{i,k})\|^2\right]&\le \frac{8\mathbb{E}\left[h(\bar{x}_1)\right]-\mathbb{E}\left[h(\bar{x}_{K+1})\right]}{s_1\sqrt{K}}+\frac{24pLC_fC_gs_1}{n\sqrt{K}}+\frac{8\breve{c}_9}{K}\sum_{k=1}^K\left(\breve{a}^{k-1} \breve{V}_1+\frac{\breve{b}}{K(1-\breve{a})}\right)\notag\\
	\label{rate-2}&\le \frac{8\mathbb{E}\left[h(\bar{x}_1)\right]-\mathbb{E}\left[h(\bar{x}_{K+1})\right]}{s_1\sqrt{K}}+\frac{24pLC_fC_gs_1}{n\sqrt{K}}+\frac{8\breve{c}_9\left( \breve{V}_1+\breve{b}\right)}{K(1-\breve{a})}.
	\end{align}
	The proof is complete.
\end{proof}

 Similar with D-ASCGD, CD-ASCGD  achieves the optimal  convergence rate  $\mathcal{O}\left(1/\sqrt{K}\right)$.
 The only difference is the third term  $\frac{8\breve{c}_9\left( \breve{V}_1+\breve{b}\right)}{K(1-\breve{a})}$  in (\ref{rate-2})  and the third term $\frac{8c_6(V_1+b)}{K(1-a)}$
 in (\ref{rate-1}).
By the definition of $\breve{a},\breve{V}_1,\breve{b},$ and the fact $\rho_w<1$,     the order of magnitude of  $\frac{8\breve{c}_9\left( \breve{V}_1+\breve{b}\right)}{K(1-\breve{a})}$  is
\begin{equation*}
\mathcal{O}\left(\frac{\max\left\{\frac{1}{1-\rho_w^2},\frac{1}{(\alpha_xr_1\psi_1)^2},\frac{1}{(\alpha_yr_2\psi_2)^2},\frac{1}{(\alpha_zr_3\psi_3)^2},\sqrt{K}\right\}}{K}\right).
\end{equation*}
By the similar analysis,  the order of magnitude of  $\frac{8c_6(V_1+b)}{K(1-a)}$  is
\begin{equation*}
\mathcal{O}\left(\frac{\max\left\{\frac{1}{1-\rho^2},\sqrt{K}\right\}}{K}\right).
\end{equation*}
It is easy to observe that the order of magnitude of the two terms is
 $\mathcal{O}\left(1/\sqrt{K}\right)$.

\section{Experimental results}\label{sec:num}

We consider a multi-agent Markov Decision Processes problem arising in reinforcement learning \cite{yang2022bilevel}
\begin{equation}\label{RL}
\min_{x\in\mathbb{R}^d}\frac{1}{2|\mathcal{S}|}\sum_{s=1}^{|\mathcal{S}|}\left(\phi_s^\intercal x-\frac{1}{n}\sum_{j=1}^n\mathbb{E}_{s^{'}}\left[r^{(j)}_{s,s^{'}}+\gamma\phi_{s^{'}}^\intercal x\big|s\right]\right)^2+\frac{\lambda}{2}\|x\|^2,
\end{equation}
%\begin{equation}\label{RL}
%\min_{x\in\mathbb{R}^d}\frac{1}{2S}\sum_{s=1}^S\left(\phi_s^\intercal x-\frac{1}{n}\sum_{j=1}^n\sum_{s^{'}}P_{s,s^{'}}\left(r^{(j)}_{s,s^{'}}+\gamma\phi_{s^{'}}^\intercal x\right)\right)^2+\frac{\lambda}{2}\|x\|^2
%\end{equation}
where $\mathcal{S}$ is a finite state space,  $\phi_s\in\mathbb{R}^d$ is the state feature, $r^{(j)}_{s,s^{'}}$ denotes the random reward of transition from $s$ to $s^{'}$ for agent $j$, $\gamma\in(0,1)$ is a discount factor, $\lambda$ is the coefficient for the $l_2$-regular.
In the setting of federated learning, each agent $j$ has
	access to its own data with heterogeneous random reward $r^{(j)}_{s,s^{'}}$ for any $s,s^{'}\in\mathcal{S}$. Denote
\begin{align*}
&g_j(x)=\left(x^\intercal,\mathbb{E}_{s^{'}}\left[r^{(j)}_{s,s^{'}}+\gamma\phi_{s^{'}}^\intercal x\big|s=1\right],\cdots,\mathbb{E}_{s^{'}}\left[r^{(j)}_{s,s^{'}}+\gamma\phi_{s^{'}}^\intercal x\big|s=|\mathcal{S}|\right]\right)^\intercal,\\
&f_j(y)=\frac{1}{2|\mathcal{S}|}\sum_{s=1}^{|\mathcal{S}|}\left(\phi_s^\intercal y(1:d)-y(d+s)\right)^2+\frac{\lambda}{2}\left\|y(1:d)\right\|^2,
\end{align*}
where $y(d)$ and $y(1:d)$ are the $(d)$-th component and the first to $d$-th components of vector $y$ respectively. Then problem (\ref{RL}) can be reformulated  as the form of problem (\ref{model}).
\begin{figure}[htbp]
	\centering
	\subfigure[exponential network, n=6]{
		\includegraphics[width=1.9in]{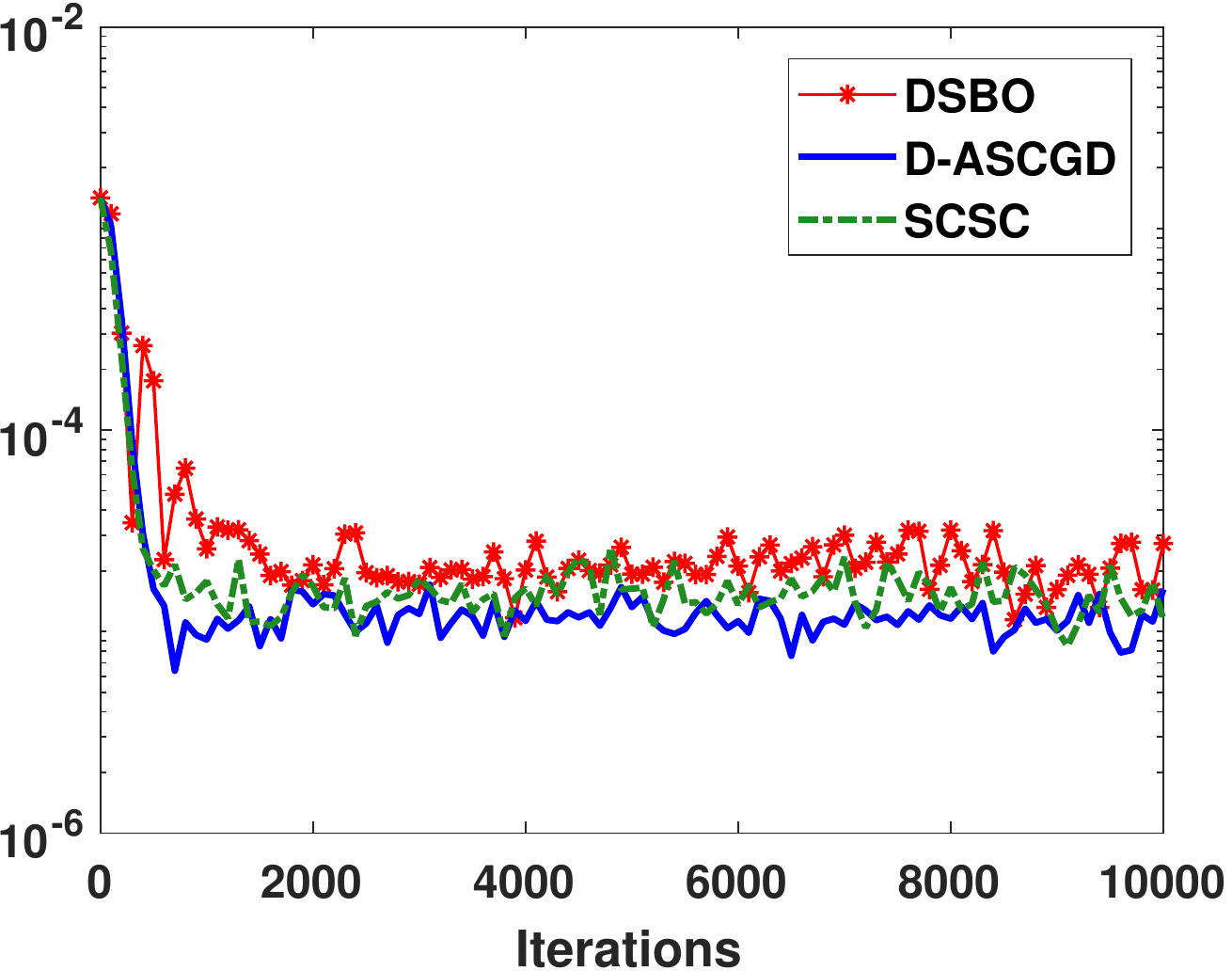}
	}
	\subfigure[exponential network, n=12]{
		\includegraphics[width=1.9in]{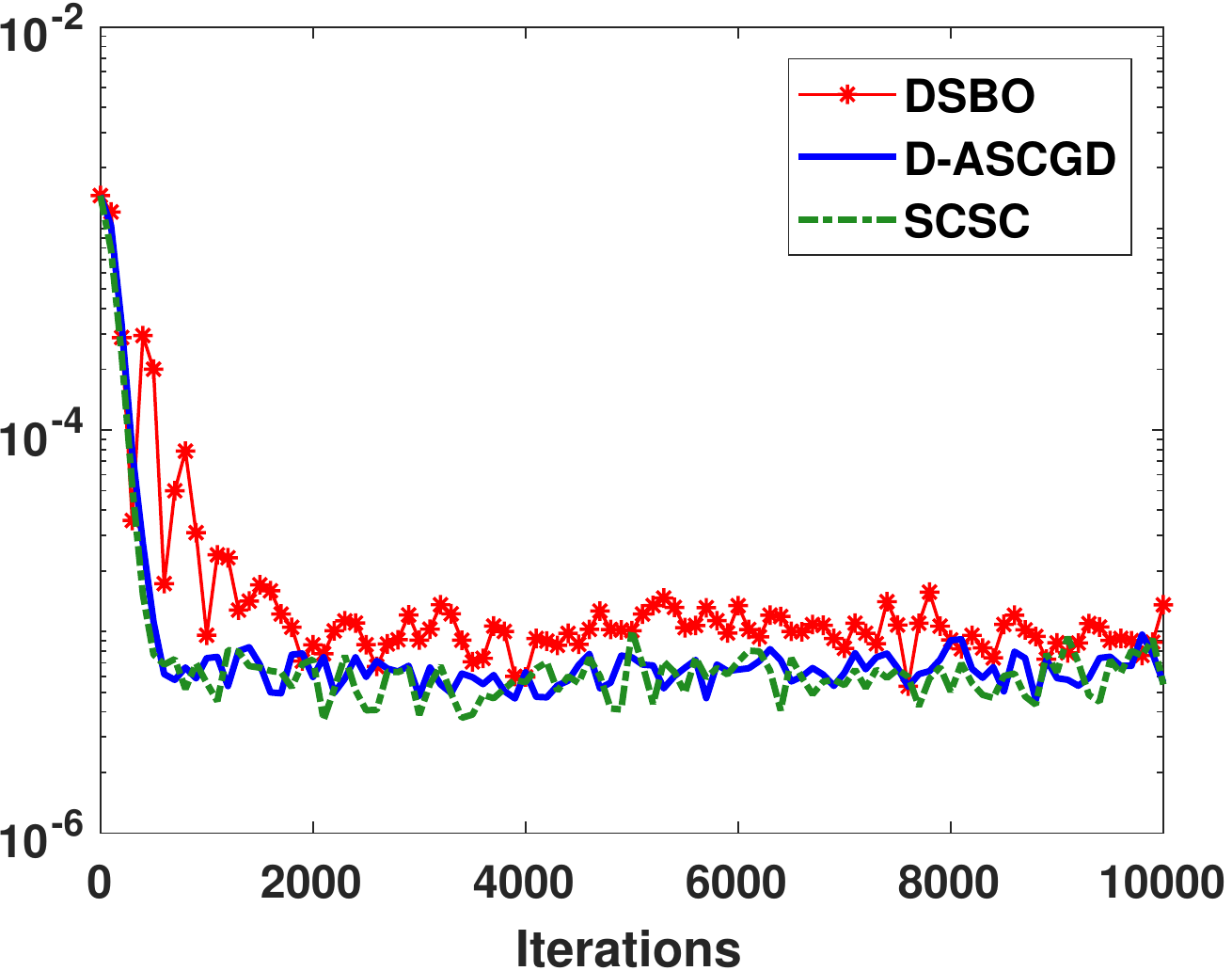}
	}
	\subfigure[exponential network, n=24]{
		\includegraphics[width=1.9in]{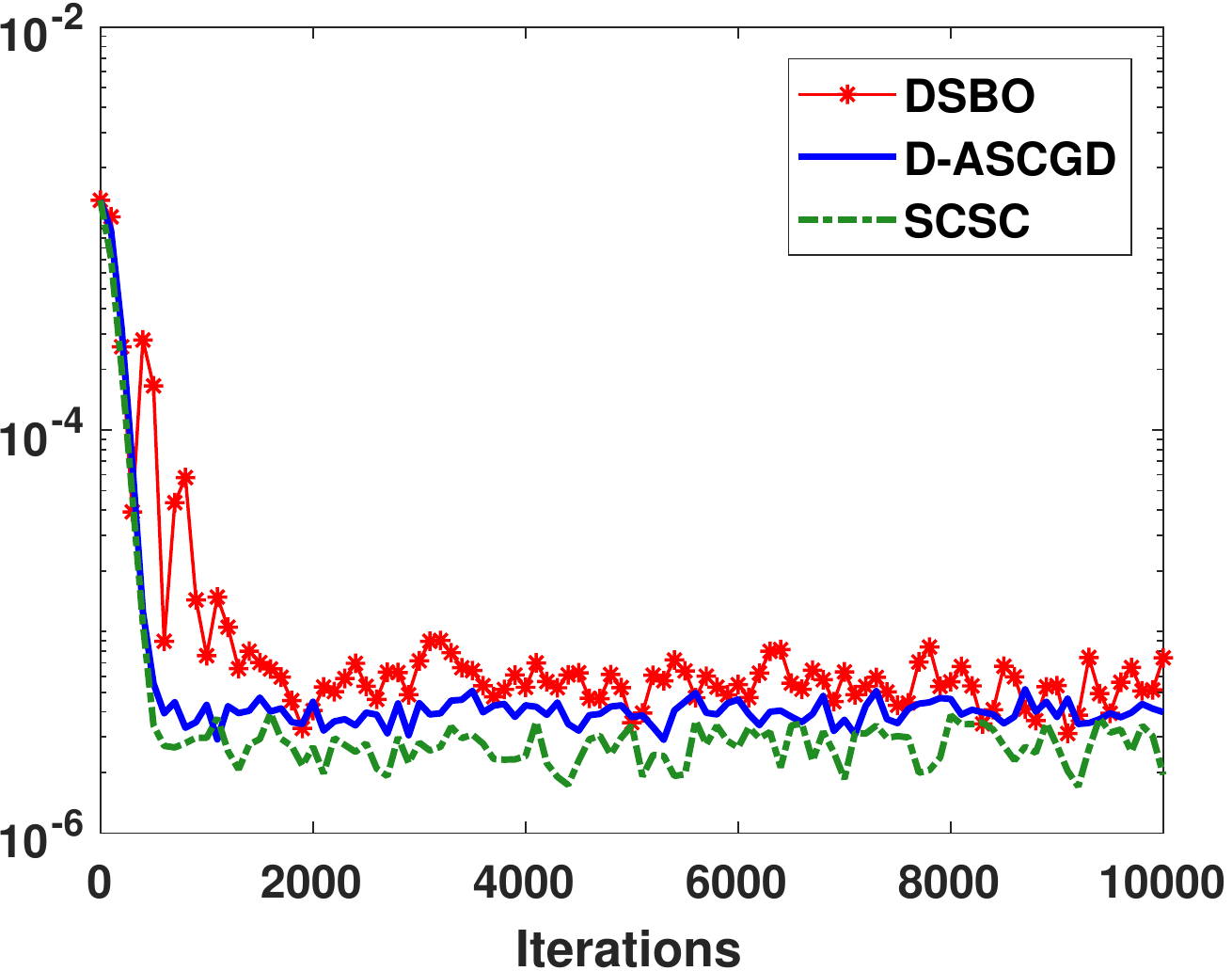}
	}
	\quad    %用 \quad 来换行
	\subfigure[ring network, n=6]{
		\includegraphics[width=1.9in]{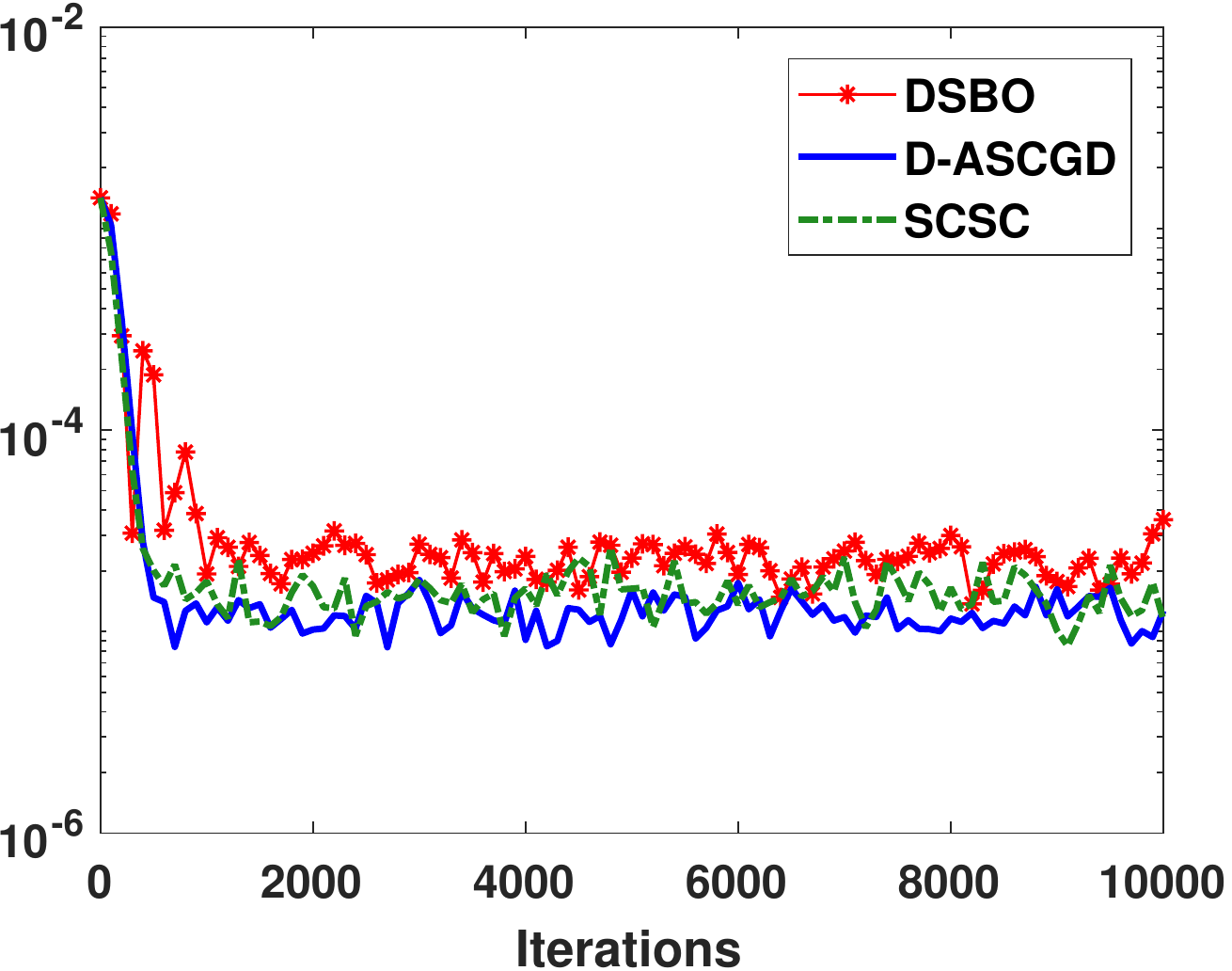}
	}
	\subfigure[ring network, n=12]{
		\includegraphics[width=1.9in]{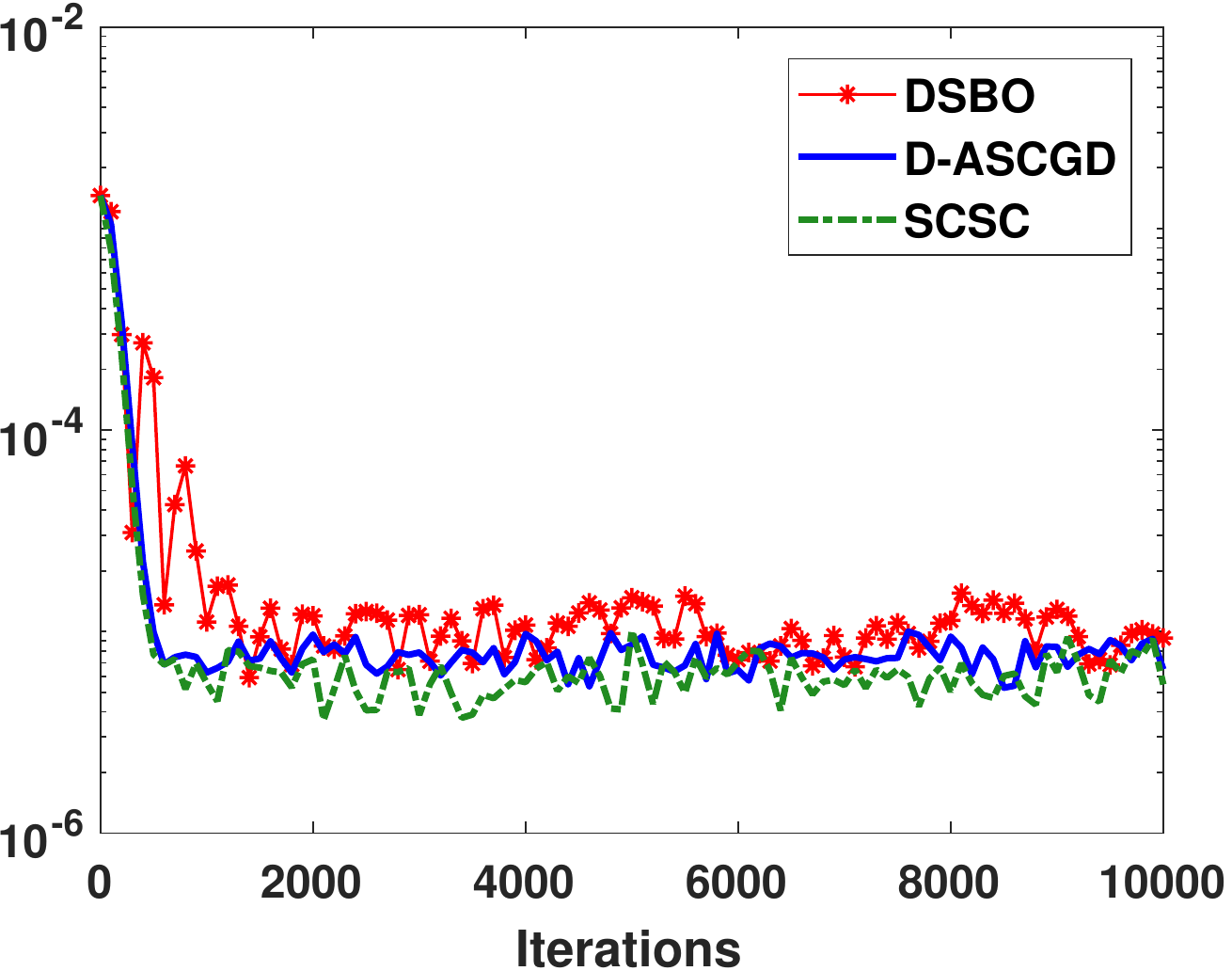}
	}
	\subfigure[ring network, n=24]{
		\includegraphics[width=1.9in]{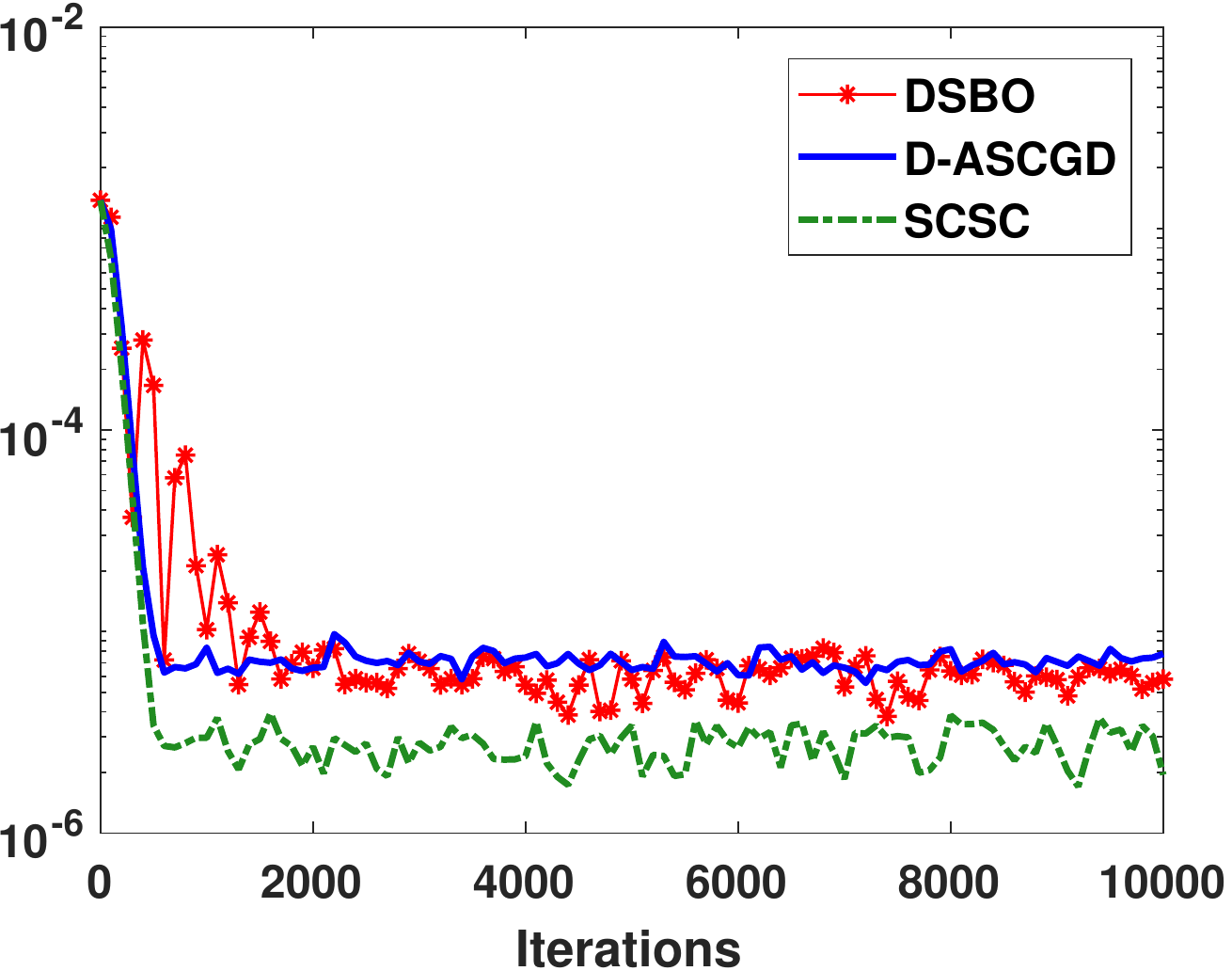}
	}
	\caption{Evolutions of $\frac{1}{n}\sum_{j=1}^n \left(h(x_{j,k})-h(x^*)\right)$   w.r.t the number of iterations.
	}
	\label{fig-5}
\end{figure}

\begin{figure}[htbp]
	\centering
	\subfigure[exponential network, n=6]{
		\includegraphics[width=1.9in]{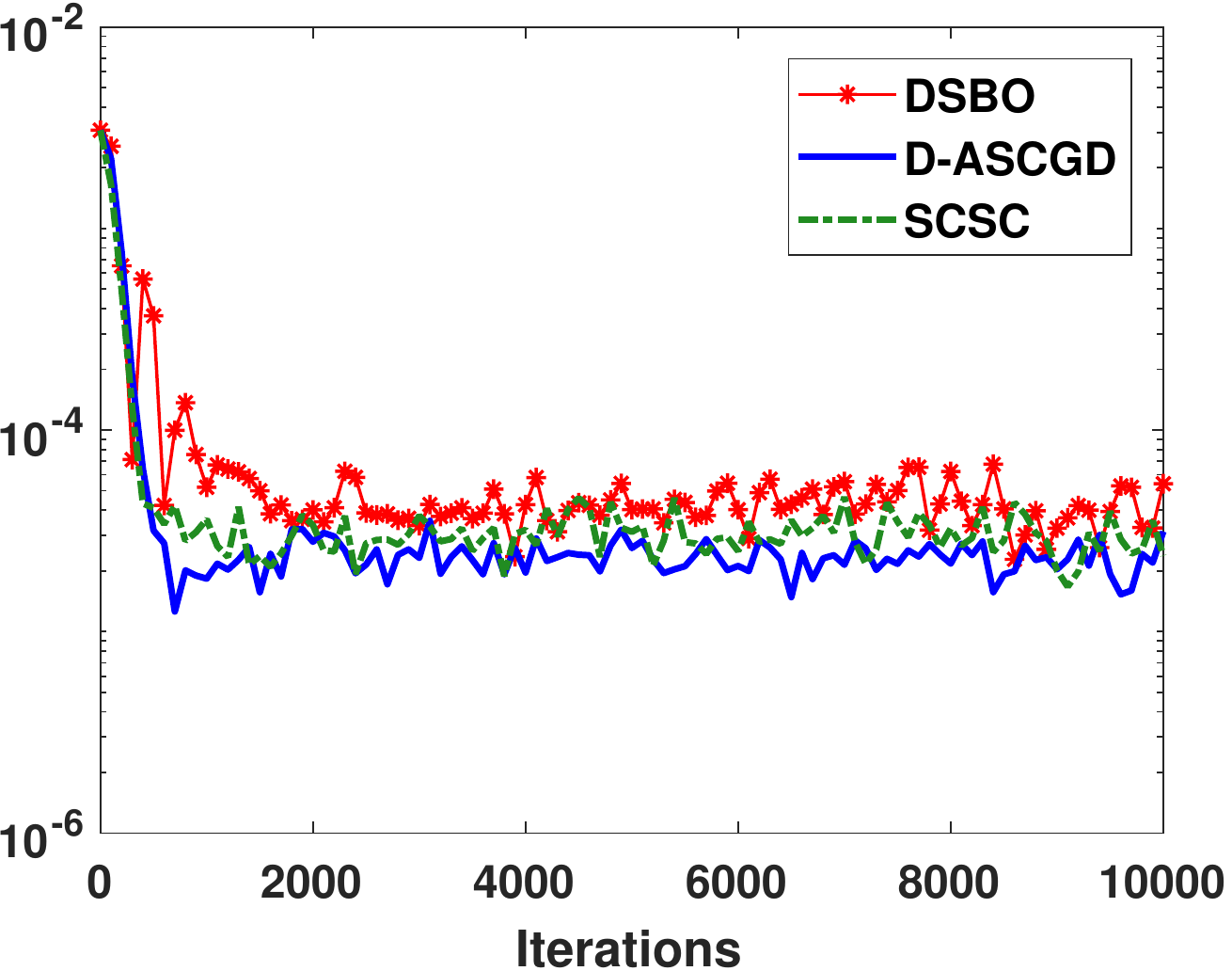}
	}
	\subfigure[exponential network, n=12]{
		\includegraphics[width=1.9in]{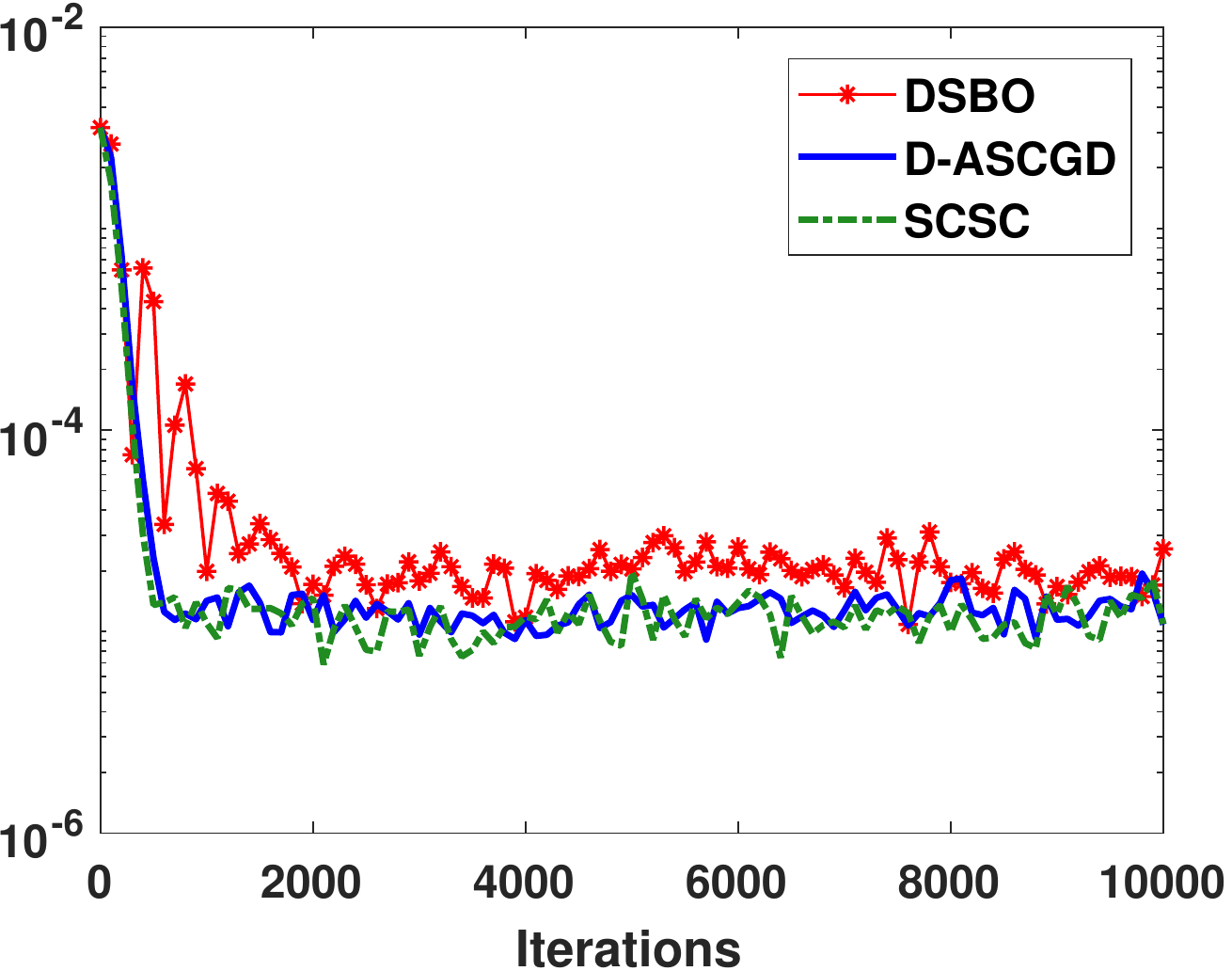}
	}
	\subfigure[exponential network, n=24]{
		\includegraphics[width=1.9in]{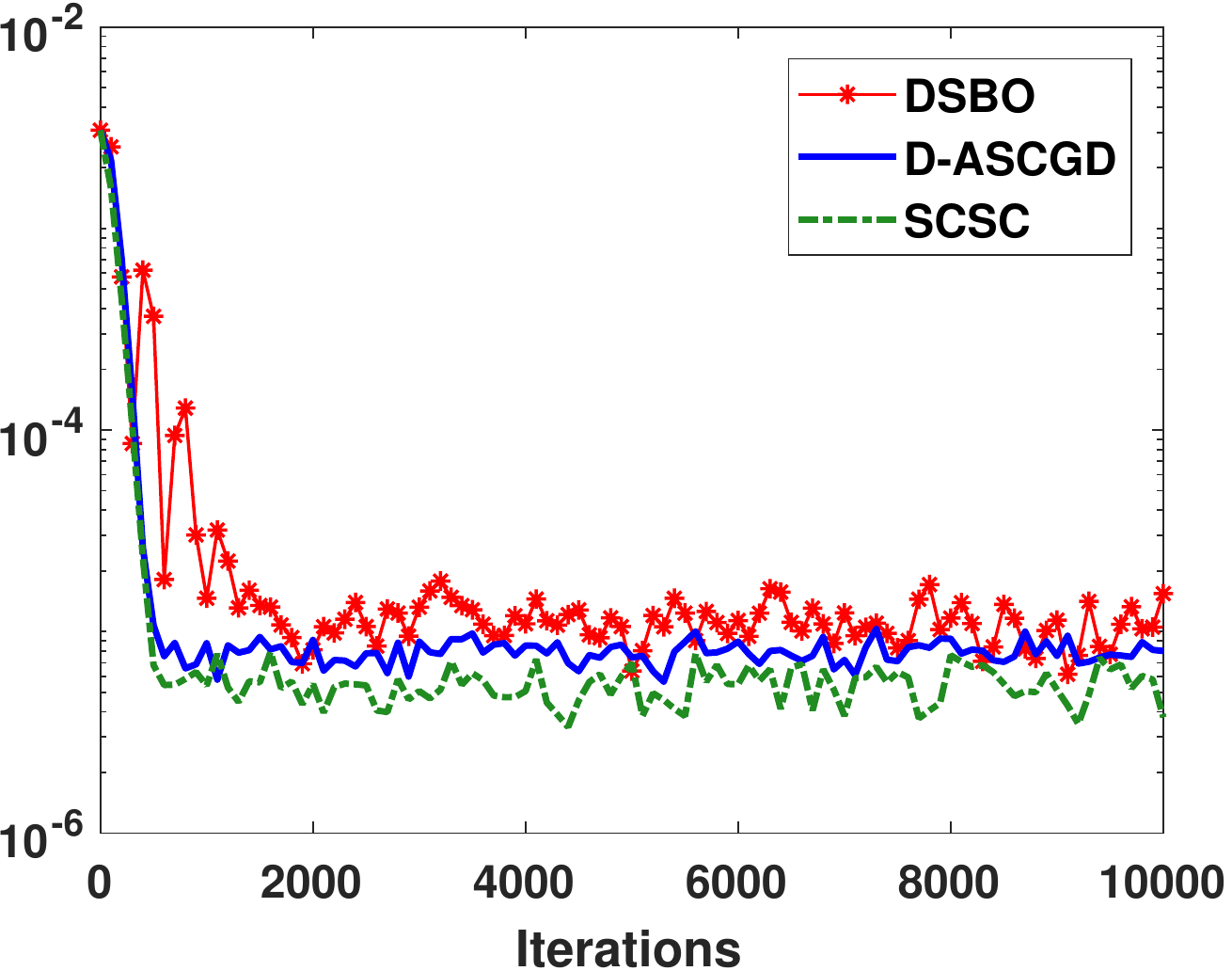}
	}
	\quad    %用 \quad 来换行
	\subfigure[ring network, n=6]{
		\includegraphics[width=1.9in]{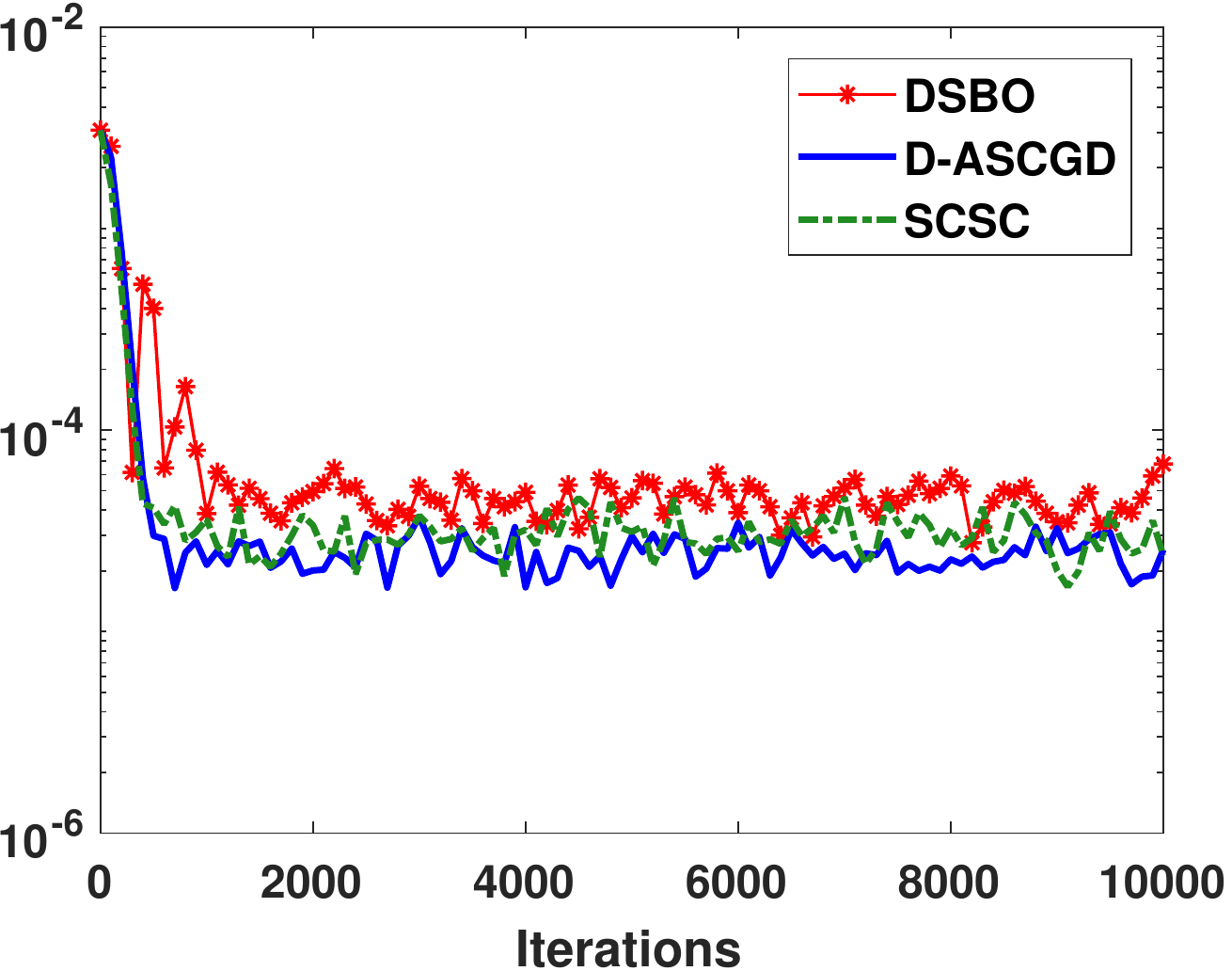}
	}
	\subfigure[ring network, n=12]{
		\includegraphics[width=1.9in]{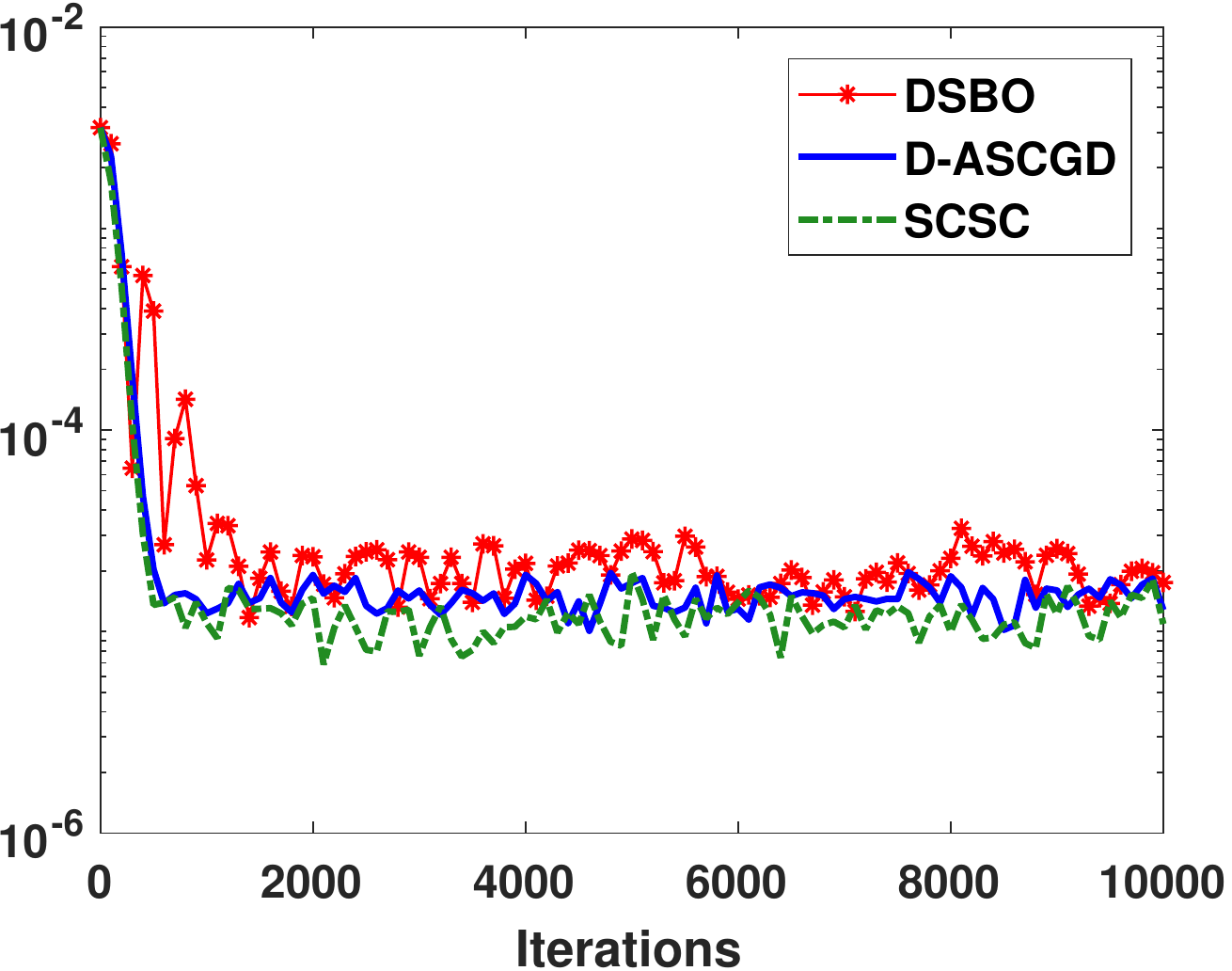}
	}
	\subfigure[ring network, n=24]{
		\includegraphics[width=1.9in]{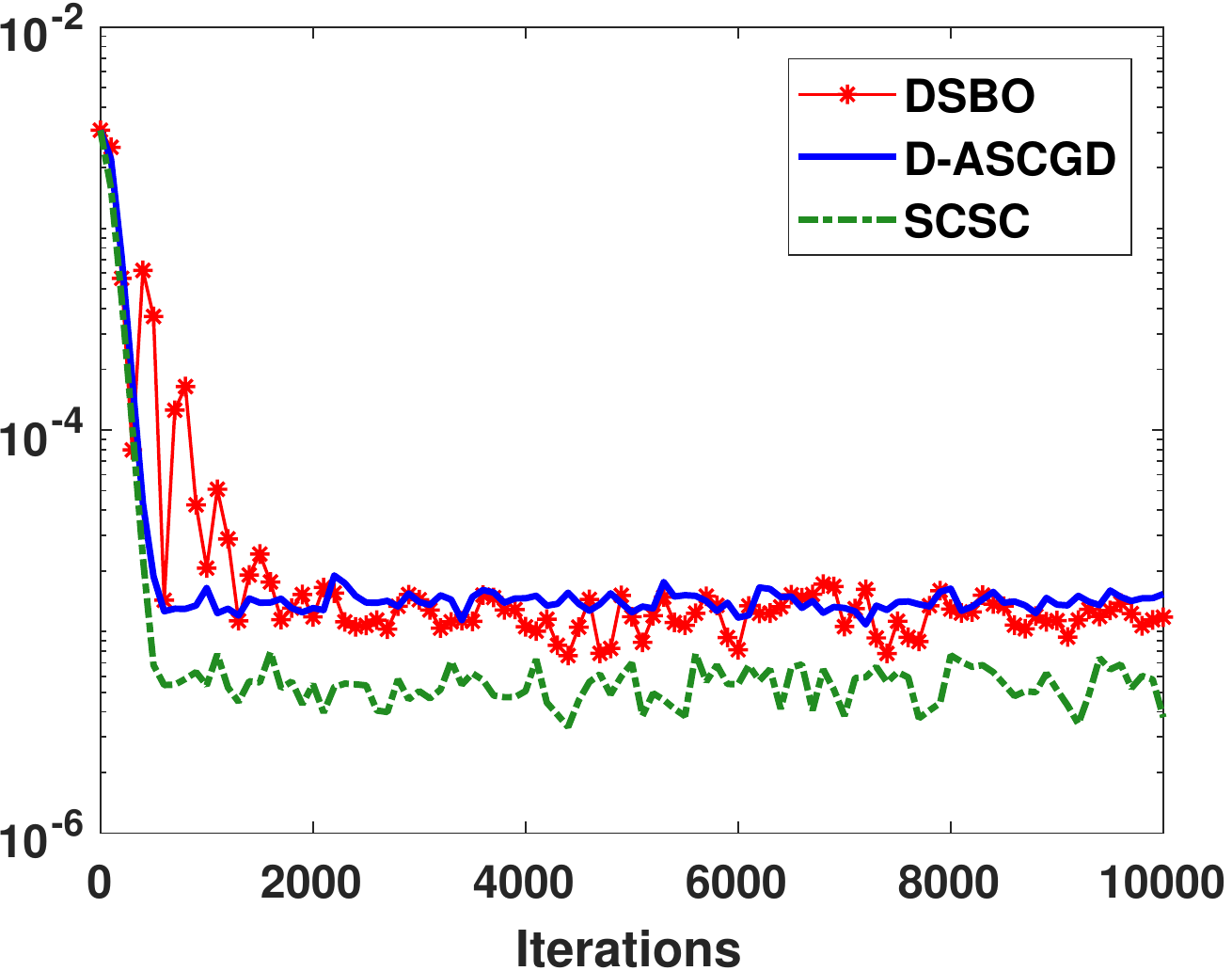}
	}
	\caption{{\small Evolutions of $\frac{1}{n}\sum_{j=1}^n \|\nabla h(x_{j,k})\|^2$  w.r.t to the number of iterations.}}
	\label{fig-6}
\end{figure}
\begin{figure}[h]
\centering
\subfigure[$\frac{1}{n}\sum_{j=1}^n \left(h(x_{j,k})-h(x^*)\right)$]{
	\includegraphics[width=2.9in]{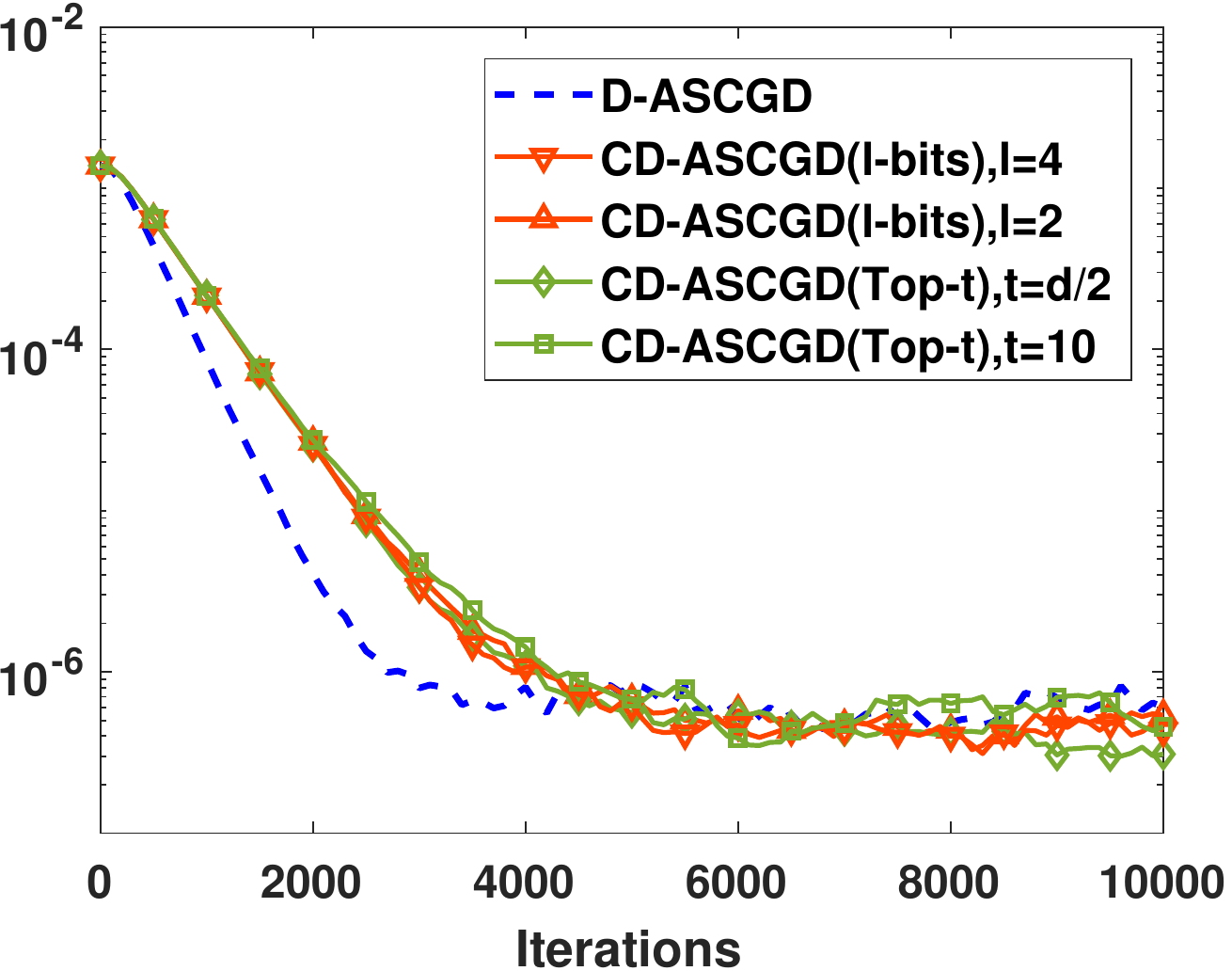}
}
\subfigure[$\frac{1}{n}\sum_{j=1}^n \|\nabla h(x_{j,k})\|^2$]{
	\includegraphics[width=2.9in]{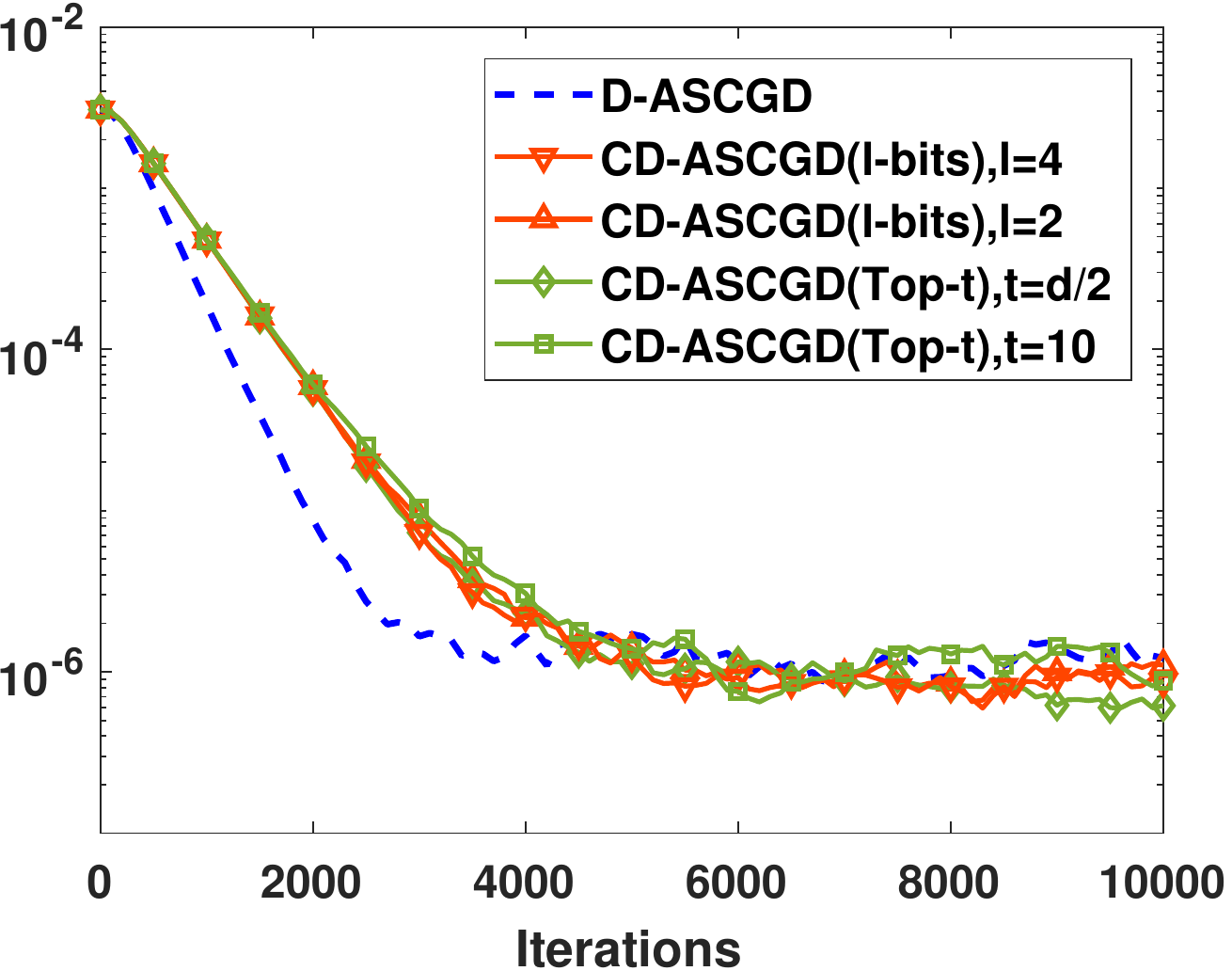}
}
\caption{{\small Evolutions of $\frac{1}{n}\sum_{j=1}^n \left(h(x_{j,k})-h(x^*)\right)$ and $\frac{1}{n}\sum_{j=1}^n \|\nabla h(x_{j,k})\|^2$ w.r.t to  the number of iterations (ring network, n=24).}}
\label{fig-9}
\end{figure}
\begin{figure}[h]
\centering
\subfigure[exponential network, n=6]{
	\includegraphics[width=1.9in]{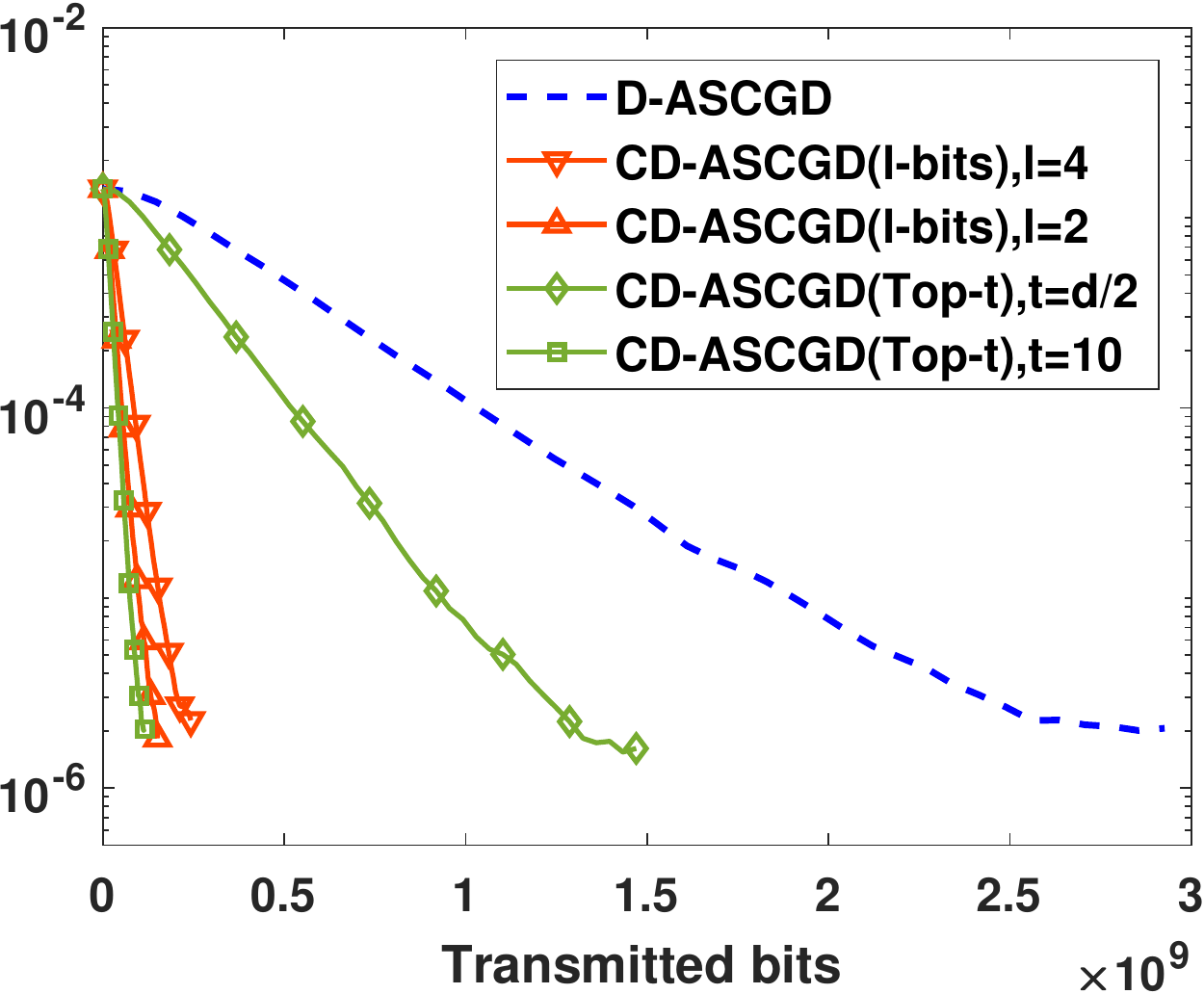}
}
\subfigure[exponential network, n=12]{
	\includegraphics[width=1.9in]{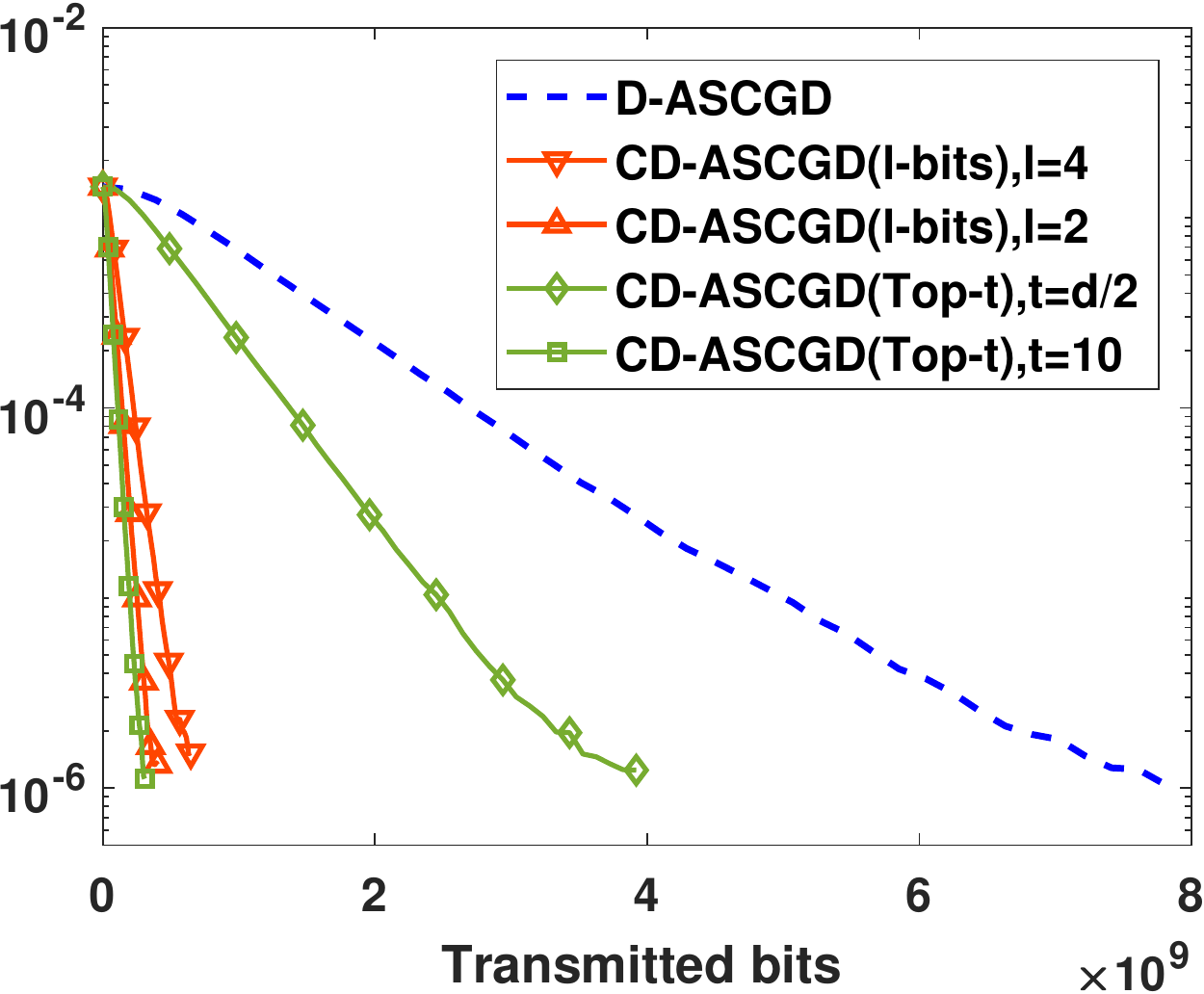}
}
\subfigure[exponential network, n=24]{
	\includegraphics[width=1.9in]{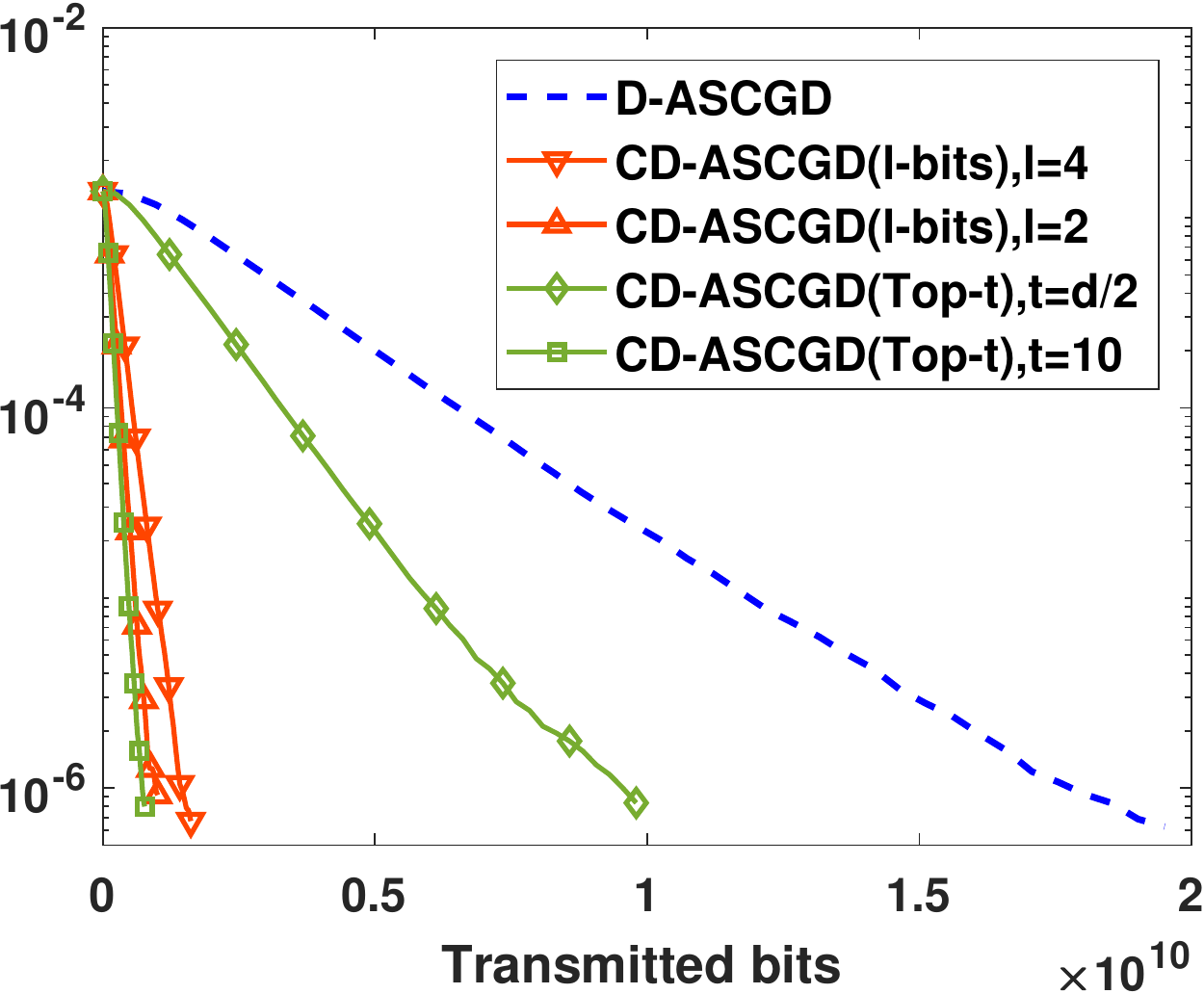}
}
\quad    %用 \quad 来换行
\subfigure[ring network, n=6]{
	\includegraphics[width=1.9in]{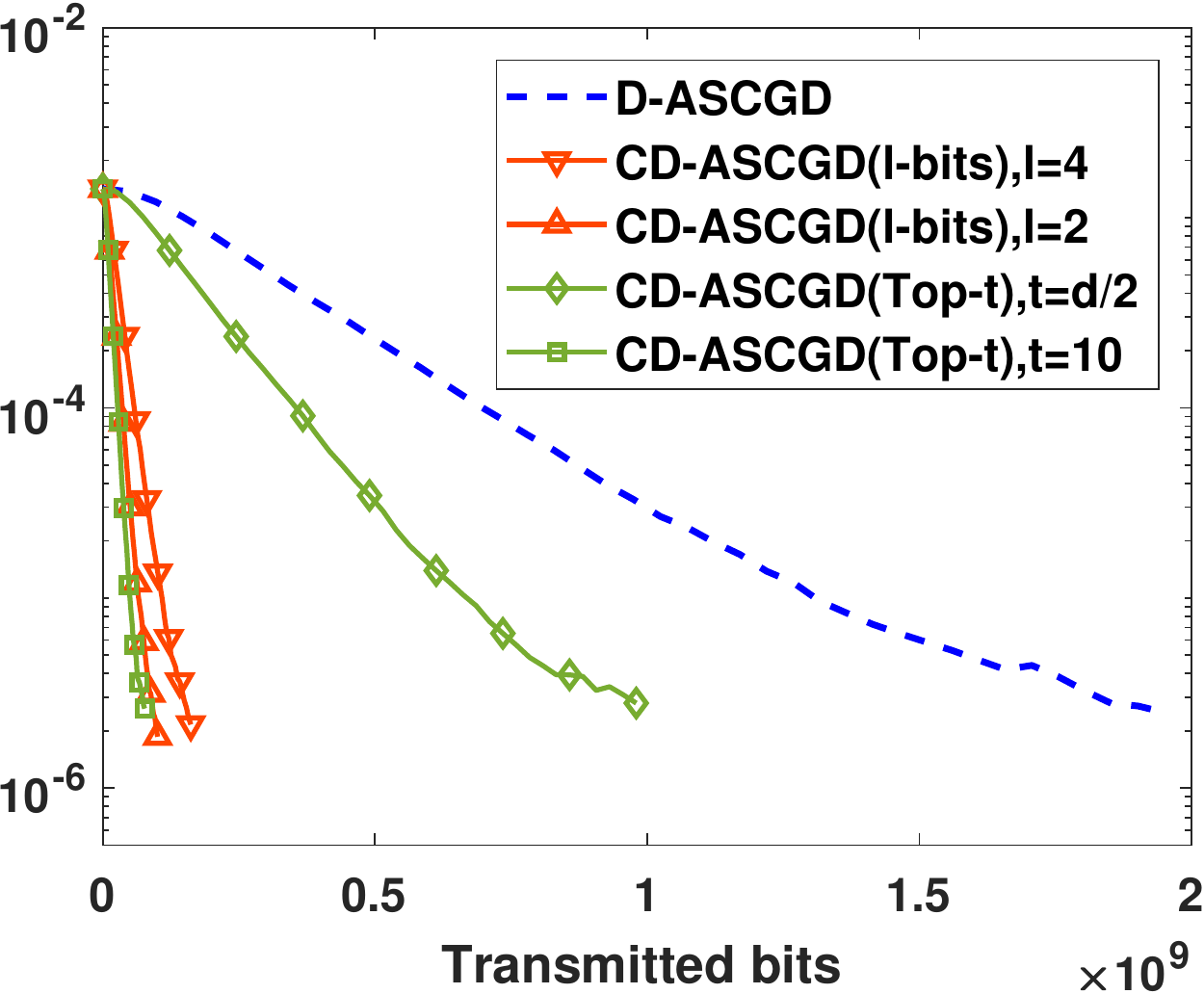}
}
\subfigure[ring network, n=12]{
	\includegraphics[width=1.9in]{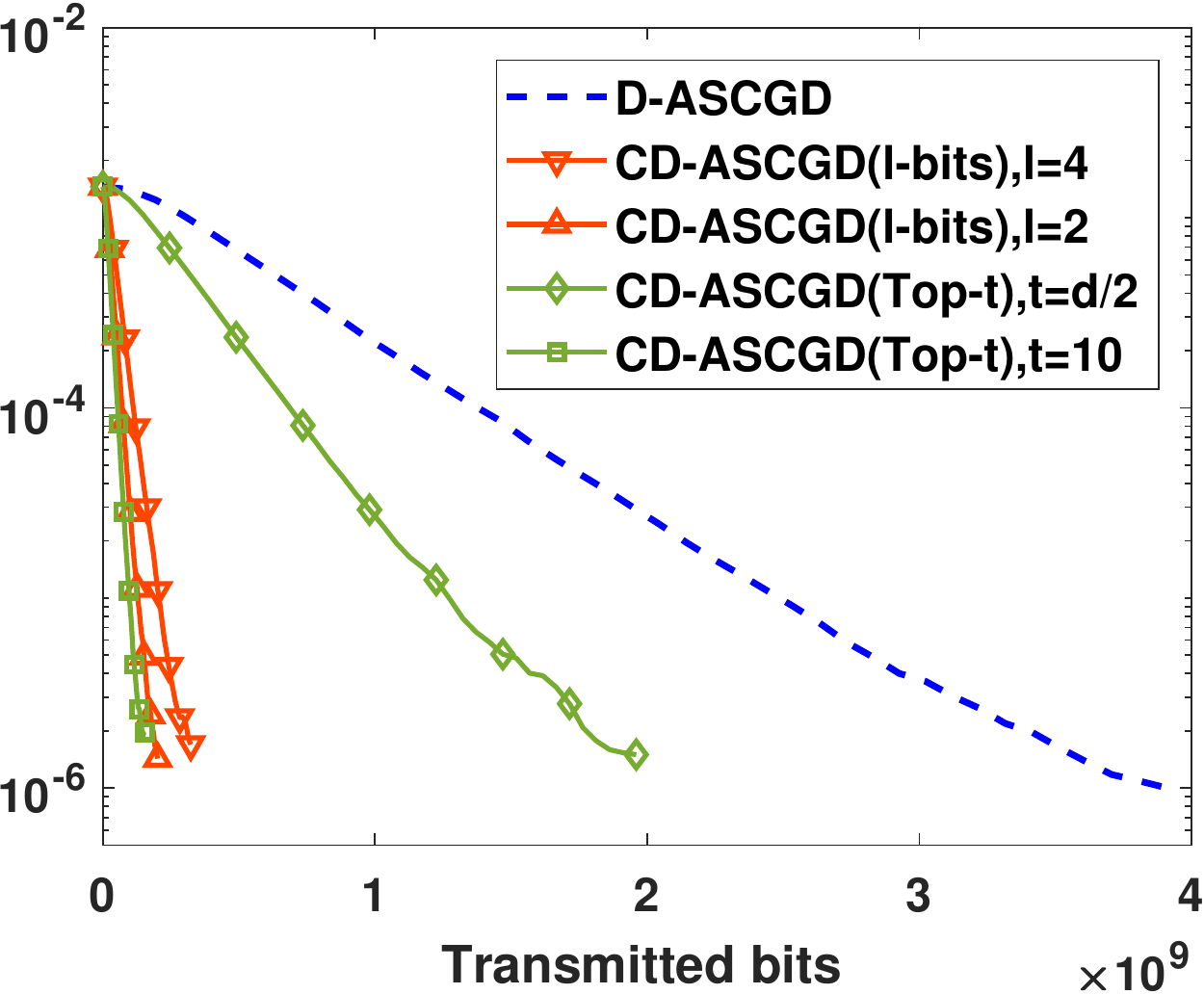}
}
\subfigure[ring network, n=24]{
	\includegraphics[width=1.9in]{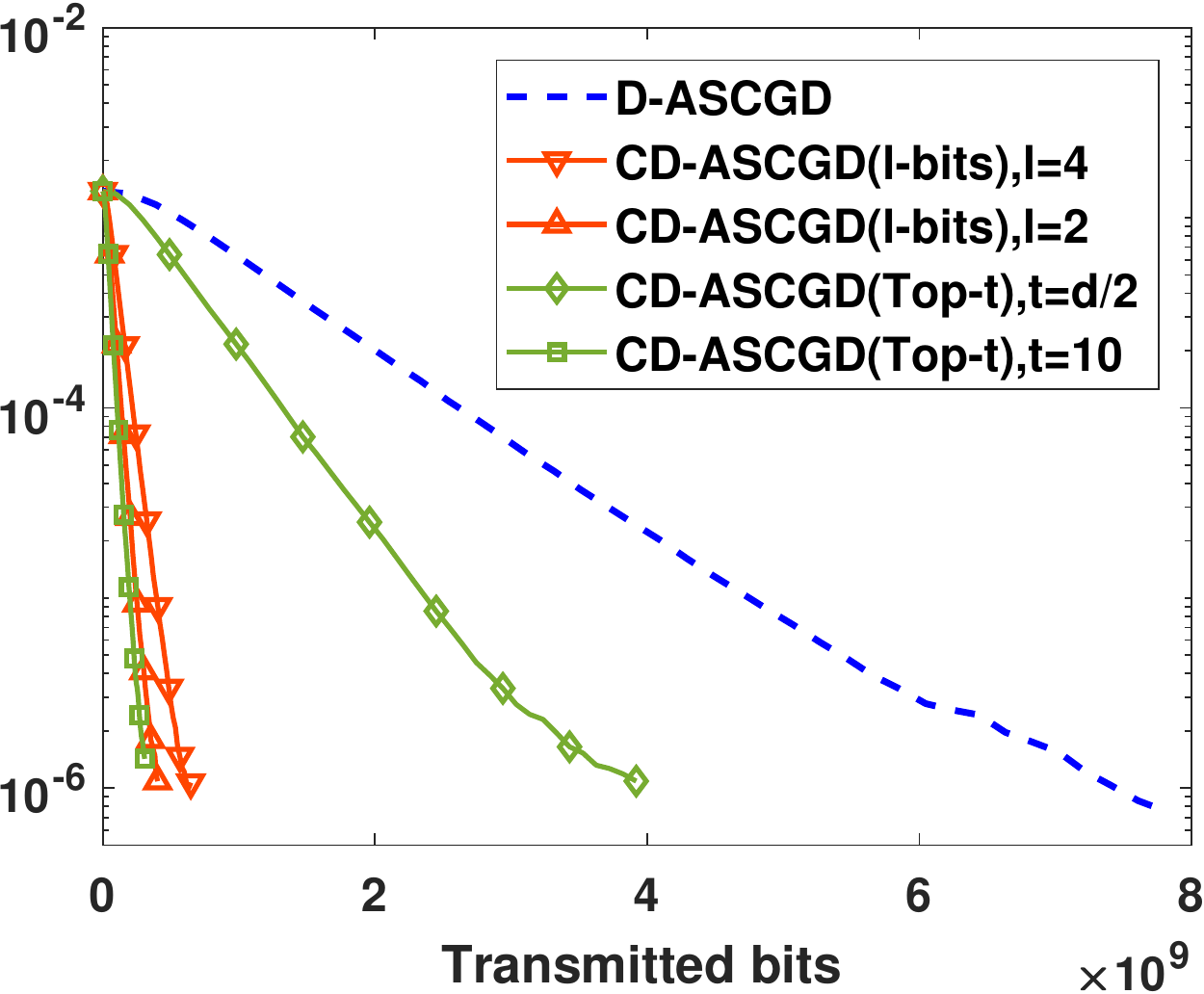}
}
\caption{Evolutions of $\frac{1}{n}\sum_{j=1}^n \left(h(x_{j,k})-h(x^*)\right)$   w.r.t the number of transmitted bits.
}
\label{fig-7}
\end{figure}
\begin{figure}[h]
\centering
\subfigure[exponential network, n=6]{
	\includegraphics[width=1.9in]{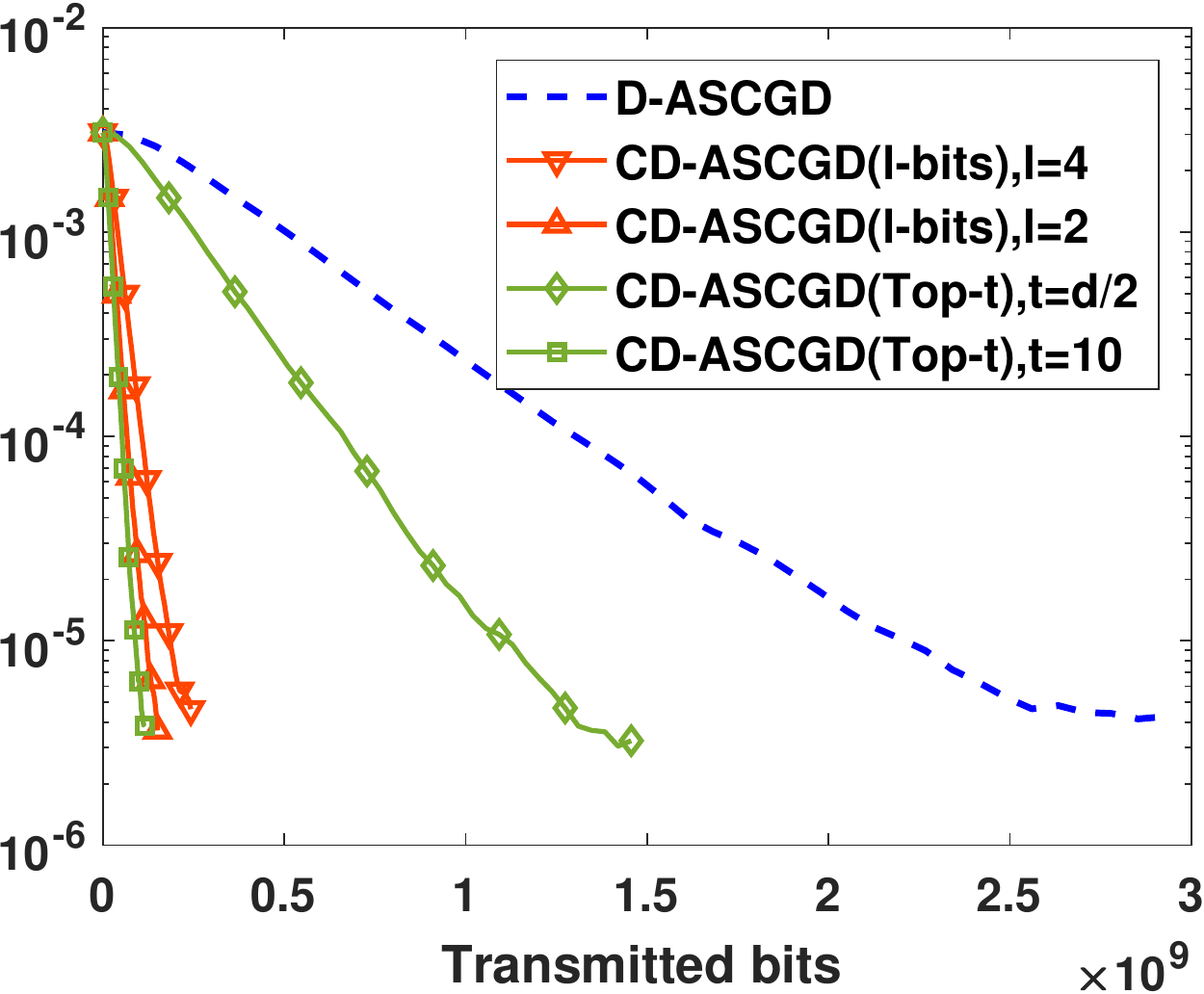}
}
\subfigure[exponential network, n=12]{
	\includegraphics[width=1.9in]{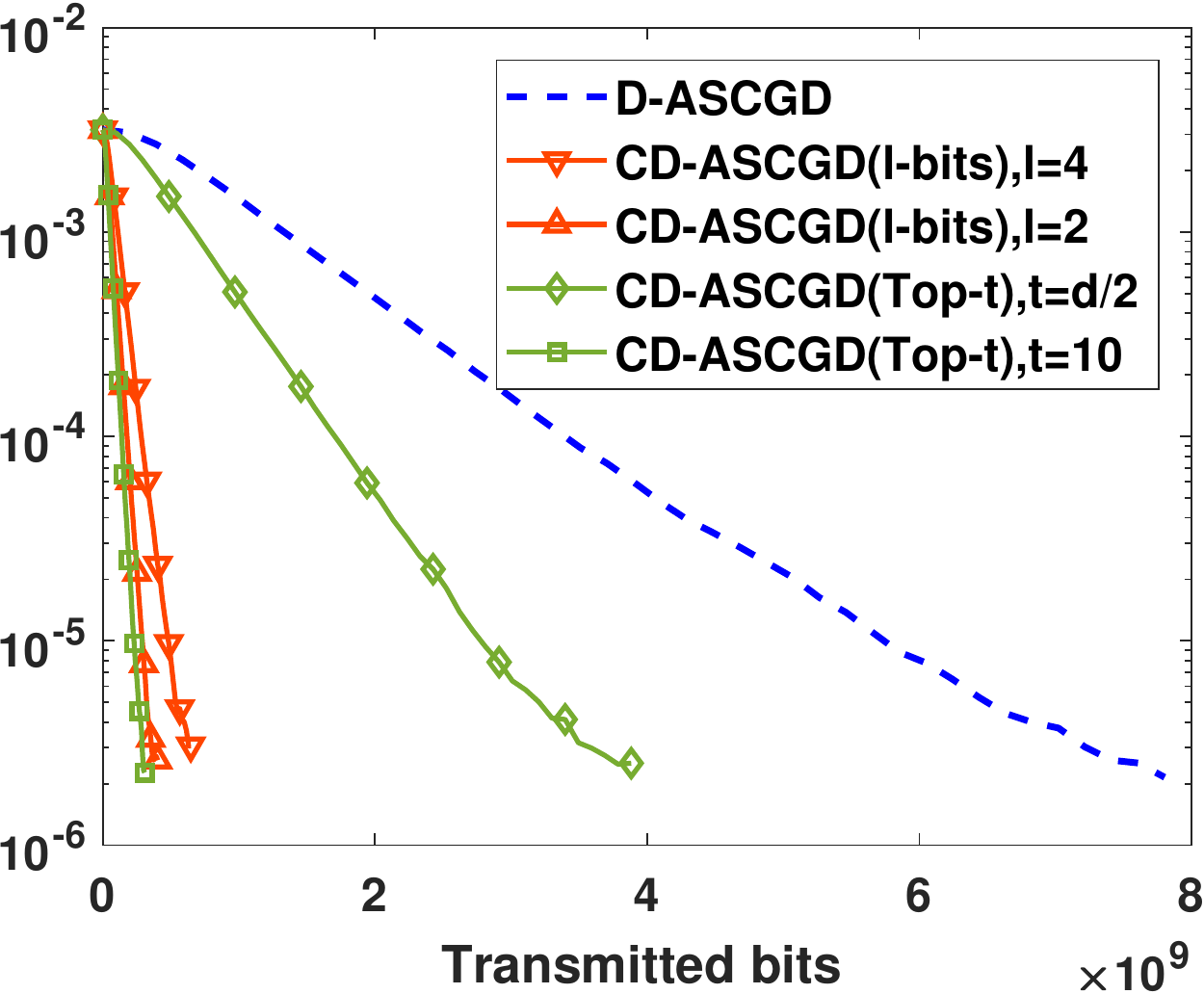}
}
\subfigure[exponential network, n=24]{
	\includegraphics[width=1.9in]{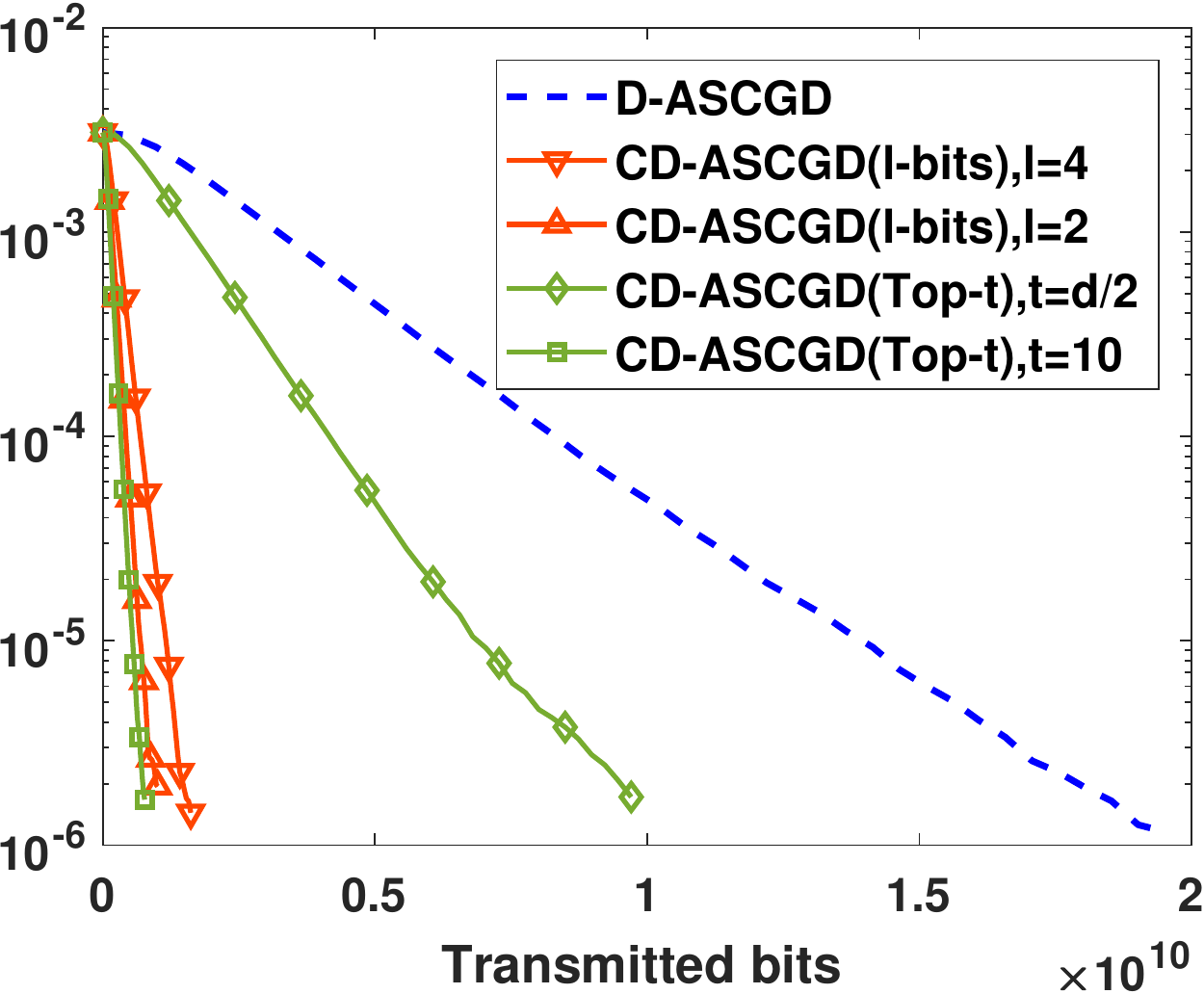}
}
\quad    %用 \quad 来换行
\subfigure[ring network, n=6]{
	\includegraphics[width=1.9in]{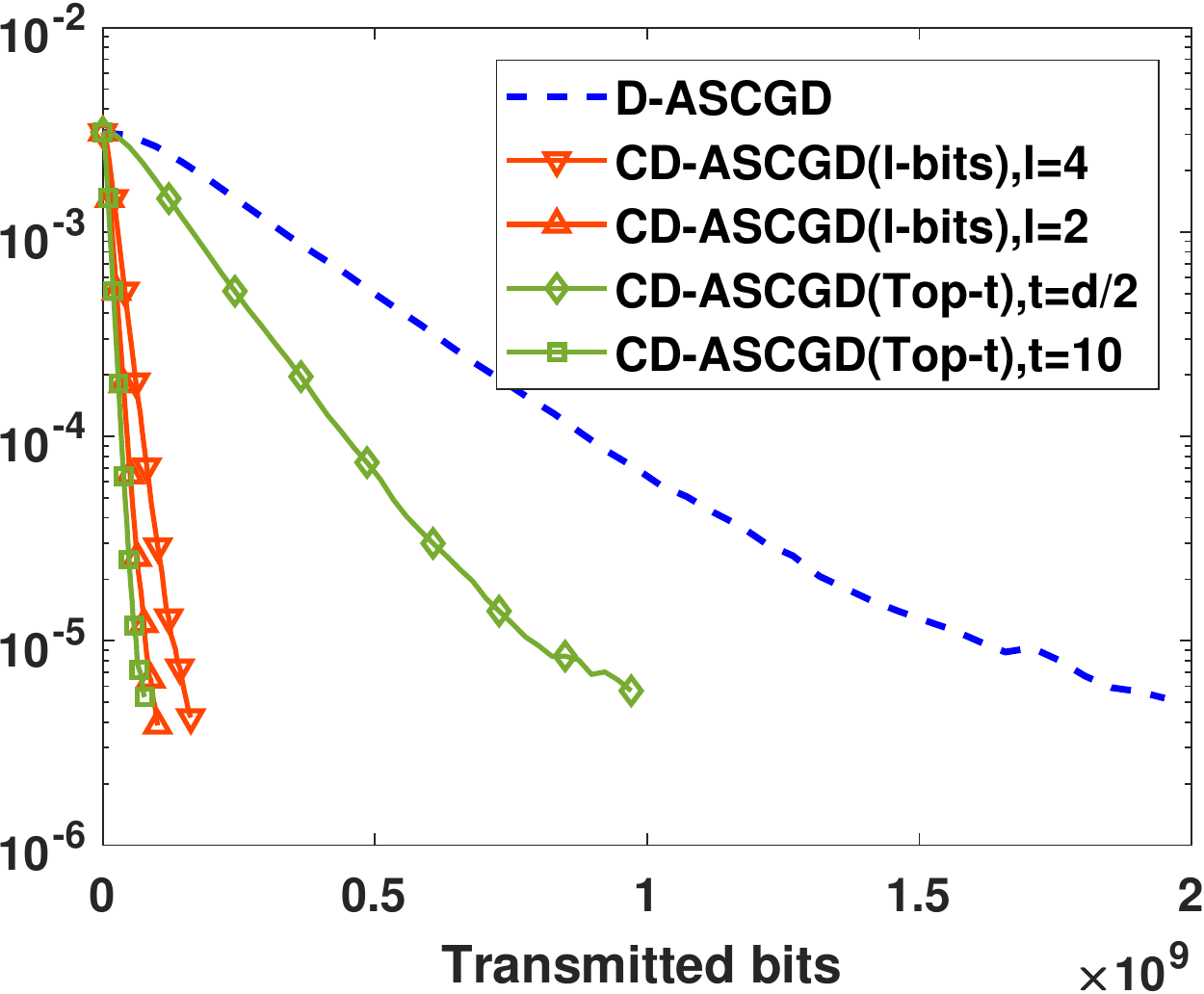}
}
\subfigure[ring network, n=12]{
	\includegraphics[width=1.9in]{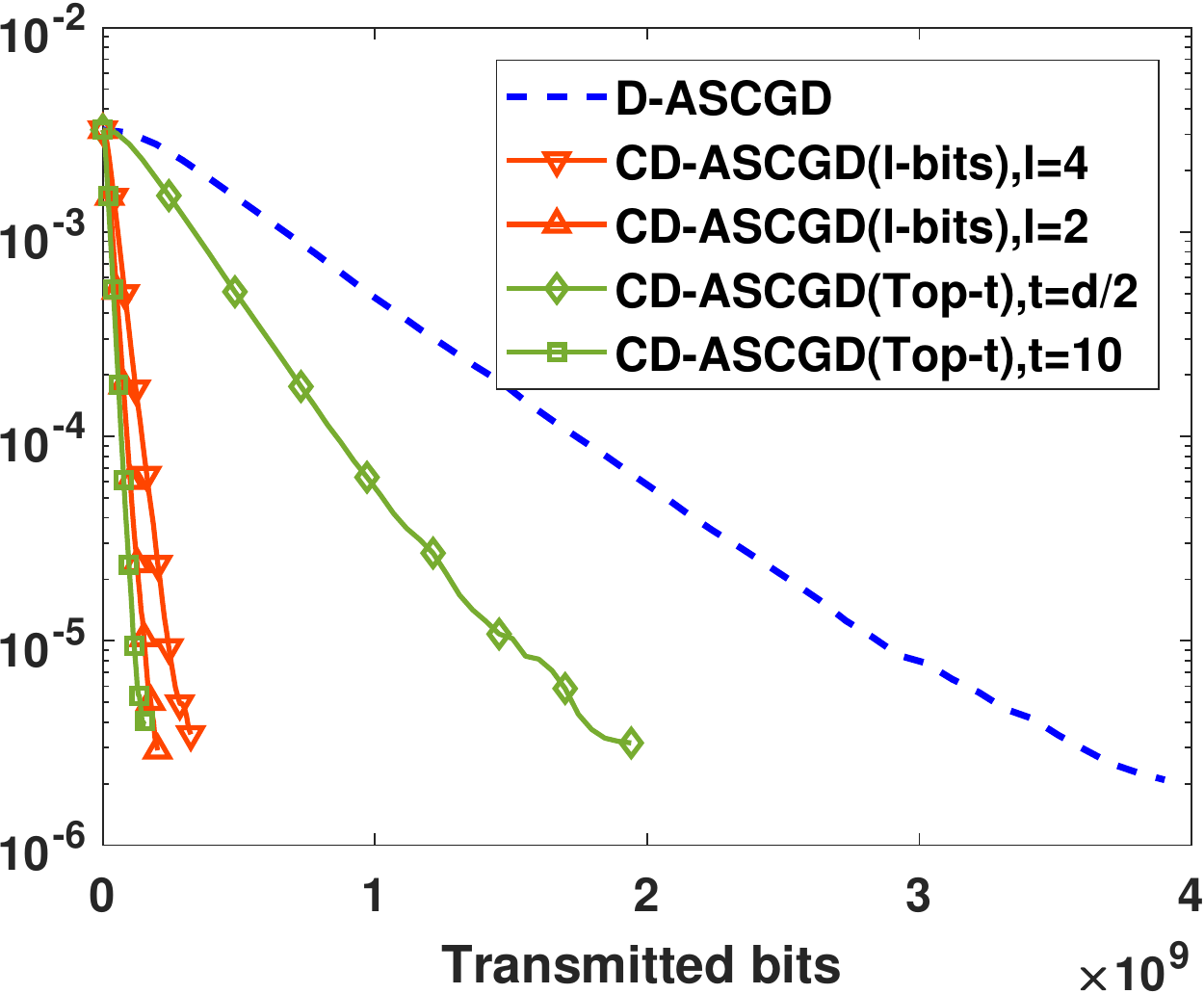}
}
\subfigure[ring network, n=24]{
	\includegraphics[width=1.9in]{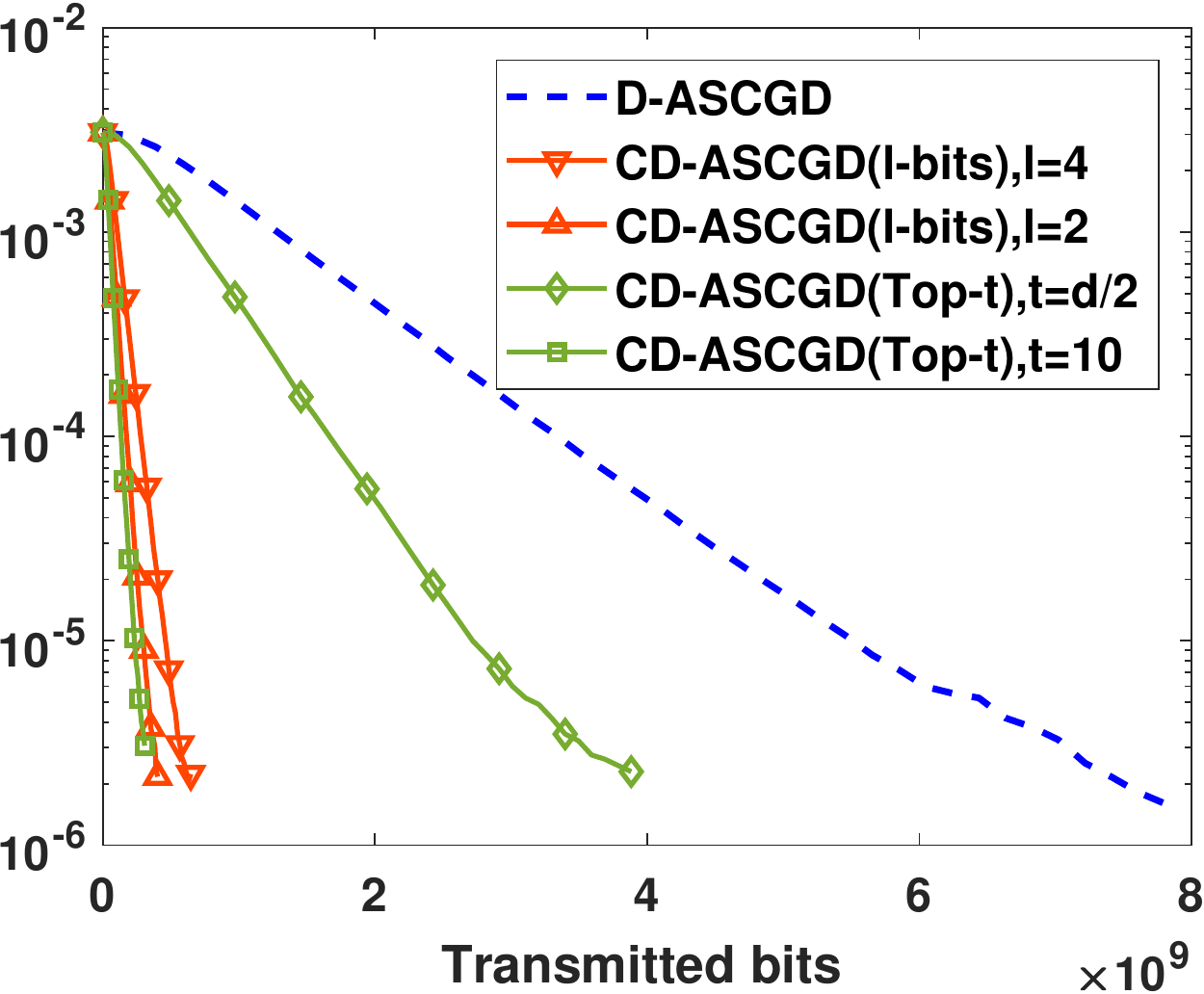}
}
\caption{{\small Evolutions of $\frac{1}{n}\sum_{j=1}^n \|\nabla h(x_{j,k})\|^2$  w.r.t to the number of transmitted bits.}}
\label{fig-8}
\end{figure}
Similar to \cite{yang2022bilevel}, we set regular parameter $\lambda=1$, the number of states $|\mathcal{S}|=100$. For any state $s\in\mathcal{S}$,  the feature $\phi_s\in [0,1]^d$ and  the mean of rewards $\bar{r}^{(i)}_{s,s^{'}}\in [0,1]$ are uniformly distributed. Moreover,  the transition probabilities $P_{s,s^{'}}$ are uniformly generated  and are standardized such that $\sum_{s^{'}=1}^{|\mathcal{S}|}P_{s,s^{'}}=1$. In each simulation, we simulate a random transition from state $s$ to state $s^{'}$ using the transition probability $P_{s,s^{'}}$, generate a random reward $r^{(i)}_{s,s^{'}}\sim N(\bar{r}^{(i)}_{s,s^{'}},1)$ for agent $i$.

We compare D-ASCGD with centralized algorithm SCSC \cite{tianyi2021sol}  and distributed algorithm  DSBO \cite{yang2022bilevel} with stepsizes  $\alpha_k=0.01$, $\beta_k=0.01$  and   $\gamma_k=0.01$.
In each iteration, D-ASCGD and DSBO utilize 5 samples to calculate gradients and function values for each agent and  SCSC utilizes $5n$ samples  to calculate gradients and function value.
We test D-ASCGD and DSBO over exponential networks and ring networks with varying
network sizes $n=6,12,24$.  The weight matrices related to the exponential networks are generated by the way described in \cite{Ying2021Exp}, and the ones related to the ring graphs are set as
\begin{equation*}
w_{i,j}=\left\{
\begin{aligned}
&0.5,\quad \text{if}~ i=j,\\
&0.25,\quad \text{if}~|i-j|=1,\\
&0.25,\quad \text{if}~(i,j)=(1,n)~\text{or}~(i,j)=(n,1),\\
&0,\quad\quad \text{otherwise}.
\end{aligned}
\right.
\end{equation*}

We record the  performance on the  average of the norm square of gradients $\frac{1}{n}\sum_{j=1}^n \|\nabla h(x_{j,k})\|^2$ and the averaged residual $\frac{1}{n}\sum_{j=1}^n \left(h(x_{j,k})-h(x^*)\right)$
in Figures \ref{fig-5} and \ref{fig-6},\footnote{$h(x)$ denotes the objective function of problem (\ref{RL}).} where the optimal solution $x^*$ is obtained by standard gradient descent method.   Similar to \cite{yang2022bilevel}, we run the simulations 10 times for the algorithms and report the average performance. We can observe from Figures \ref{fig-5} and \ref{fig-6} that the three algorithms   converge fast and have the comparable residuals of optimal values and gradients.  Note also that SCSC uses the information of  global function, its performance is slightly  better than
 the distributed  algorithms  D-ASCGD and DSBO.  For the two distributed algorithms D-ASCGD and DSBO, they have the comparable convergence over the networks with different types and sizes, which matches the conclusion  of Theorem \ref{thm:rate}.

For  CD-ASCGD, we consider the following two  compressors.

\begin{itemize}
	\item [$\bullet$]$l$-bits quantizer  \cite{Zhang2021Innov}:
	\begin{equation*}
	\mathcal{C}(x)=\frac{x_{max}}{2^{l-1}}\textbf{sign}(x)\odot \left\lfloor\frac{2^{l-1}|x|}{x_{max}}+u\right\rfloor,
	\end{equation*}
	where $\textbf{sign}(x)$ is the sign function, $\odot$ is the Hadamard product, $|x|$ is the element-wise absolute value of $x$, and $u$ is a random perturbation vector uniformly distributed in $[0,1]^d$, $x_{max}$ refers to the largest absolute value of the elements of $x$. In the test, we choose $l=2,4$ and $b=64$.
	\item [$\bullet$]Top-t sparsifier \cite{Bezno2020biased}:
		\begin{equation*}
		\mathcal{C}(x)=\sum_{l=1}^t[x]_{i_l}\mathbf{e}_{i_l},
		\end{equation*}
		where $\{\mathbf{e}_1,\cdots,\mathbf{e}_d\}$ is the standard basis of $\mathbb{R}^d$ and $i_1,\cdots,i_t$ are the indices of largest $t$ coordinates in magnitude of $x$.   We consider the cases that $t=d/2,10$ ($d>20$).
\end{itemize}

In Figure  \ref{fig-9}, we record the average performance on {\small$\frac{1}{n}\sum_{j=1}^n \|\nabla h(x_{j,k})\|^2$} and {\small$\frac{1}{n}\sum_{j=1}^n \left(h(x_{j,k})-h(x^*)\right)$} with
10 times simulation.  As shown in  \ref{fig-9},  CD-ASCGD  has comparable convergence speeds with D-ASCGD, which is consistent with our theoretical results.
In Figures \ref{fig-7} and \ref{fig-8}, we record the convergence performances of CD-ASCGD and D-ASCGD with respect to the number of bits transmitted between agents. It is easy to observe that CD-ASCGD   transmits  less bits  than D-ASCGD,  which becomes more obvious as the  size of network increasing and the needed transmitted bits of  compressors decreasing.
 Moreover, the  CD-ASCGD needs to transmit more bits  between agents with the increasing of the network size.

\textbf{Acknowledgment.} The research is supported by the NSFC \#11971090 and  Fundamental Research Funds for the Central Universities   DUT22LAB301.

\bibliographystyle{siam}
\bibliography{dsi_mybib}

\section*{Appendix}
\begin{appendices}
\section{}\label{apd:D-NSCGD}
\textbf{Proof of Lemma \ref{lem:tracking}}
\begin{proof}
	By the iteration (\ref{alg:G-1}),
	\begin{align}
	\left\|\mathbf{G}_{k+1}- \mathbf{g}_{k+1}\right\|^2 &=\left\|(1-\beta_k)\left(\mathbf{G}_k-\mathbf{g}_k\right)+(G_{k+1,k+1}-\mathbf{g}_{k+1})+(1-\beta_k)(\mathbf{g}_k-G_{k+1,k})\right\|^2\notag\\
	&=(1-\beta_k)^2\|\mathbf{G}_k-\mathbf{g}_k\|^2+\left\|(G_{k+1,k+1}-\mathbf{g}_{k+1})+(1-\beta_k)(\mathbf{g}_k-G_{k+1,k})\right\|^2\notag\\
	\label{G-bound-0}&\quad + 2\left\langle(1-\beta_k)(\mathbf{G}_k-\mathbf{g}_k),(G_{k+1,k+1}-\mathbf{g}_{k+1})+(1-\beta_k)(\mathbf{g}_k-G_{k+1,k})\right\rangle.
	\end{align}
	Note that
	{\small\begin{equation*}
		\begin{aligned}
		&\mathbb{E}\left[(G_{k+1,k+1}-\mathbf{g}_{k+1})+(1-\beta_k)(\mathbf{g}_k-G_{k+1,k})\right]\\
		 &=\mathbb{E}\left[\mathbb{E}\left[(G_{k+1,k+1}-\mathbf{g}_{k+1})+(1-\beta_k)(\mathbf{g}_k-G_{k+1,k})\bigg|\mathcal{F}_k,\zeta_{1,k+1},\cdots,\zeta_{n,k+1}\right]\right]=\mathbf{0},
		\end{aligned}
		\end{equation*}}
	where
	\begin{equation}\label{sigma-alge}
	\begin{aligned}
	&\mathcal{F}_1=\sigma\left(x_{i,1}, y_{i,1},z_{i,1},G_{i,1},\hat{G}_{i,1}:i\in\mathcal{V}\right),\\
	&\mathcal{F}_k=\sigma\{x_{i,1},y_{i,1},z_{i,1},G_{i,1},\hat{G}_{i,1}, \phi_{i,t},\zeta_{i,t}:i\in\mathcal{V}, 2\le t\le k\}(k\ge2).
	\end{aligned}
	\end{equation}
	Then, taking expectation on both sides of (\ref{G-bound-0}),
	{\small\begin{align*}
		&\mathbb{E}\left[\|\mathbf{G}_{k+1}- \mathbf{g}_{k+1}\|^2\right]= (1-\beta_k)^2\mathbb{E}\left[\|\mathbf{G}_k-\mathbf{g}_k\|^2\right]+\mathbb{E}\left[\left\|(G_{k+1,k+1}-\mathbf{g}_{k+1})+(1-\beta_k)(\mathbf{g}_k-G_{k+1,k})\right\|^2\right].
		\end{align*}}
	By Assumption \ref{ass-objective}(c) and (d),
	\begin{align*}
	&\mathbb{E}\left[\left\|(G_{k+1,k+1}-\mathbf{g}_{k+1})+(1-\beta_k)(\mathbf{g}_k-G_{k+1,k})\right\|^2\right]\notag\\
	&=\mathbb{E}\left[\|(1-\beta_k)(G_{k+1,k+1}-G_{k+1,k})+\beta_k( G_{k+1,k+1}- \mathbf{g}_{k+1})+(1-\beta_k)(\mathbf{g}_k-\mathbf{g}_{k+1})	 \|^2\right]\notag\\
	&\le 3(1-\beta_k)^2\mathbb{E}\left[\|G_{k+1,k+1}-G_{k+1,k}\|^2\right]+3\beta_k^2\mathbb{E}\left[\|G_{k+1,k+1}- \mathbf{g}_{k+1}\|^2\right]+3(1-\beta_k)^2\mathbb{E}\left[\|\mathbf{g}_k-\mathbf{g}_{k+1}	 \|^2\right]\notag\\
	&\le 6(1-\beta_k)^2C_g\mathbb{E}\left[\|\mathbf{x}_{k+1}-\mathbf{x}_k\|^2\right]+3\beta_k^2V_g.
	\end{align*}
	Note that
	\begin{align}
	 \mathbb{E}\left[\|\mathbf{x}_{k+1}-\mathbf{x}_k\|^2\right]&=\mathbb{E}\left[\left\|\left(\tilde{\mathbf{W}}-\mathbf{I}_{nd}\right)\left(\mathbf{x}_k-\mathbf{1}\otimes \bar{x}_k\right)-\alpha_k\mathbf{U}_{k+1}\right\|^2\right]\notag\\
	&\le 2\left(4\mathbb{E}\left[ \left\|\mathbf{x}_k-\mathbf{1}\otimes \bar{x}_k\right\|^2\right]+\alpha_k^2\mathbb{E}\left[\left\|\mathbf{U}_{k+1}\right\|^2\right]\right)\notag\\
	\label{x-x-bound}&\le 2\left(4\mathbb{E}\left[ \left\|\mathbf{x}_k-\mathbf{1}\otimes \bar{x}_k\right\|^2\right]+\alpha_k^2C_f\mathbb{E}\left[\normm{\mathbf{z}_k}_F^2\right]\right),
	\end{align}
	where the equality follows from the row stochasticity of $\mathbf{W}$, the first inequality follows from the facts $\normm{\mathbf{W}-\mathbf{I}_{n}}\le 2$ and $\normm{\mathbf{W}}=1$, the second inequality follows from the definition of $\mathbf{U}_{k+1}$ and Assumption \ref{ass-objective}(c).
	Then
	\begin{align*}
	\mathbb{E}\left[\|\mathbf{G}_{k+1}- \mathbf{g}_{k+1}\|^2\right]
	&\le (1-\beta_k)^2\mathbb{E}\left[\|\mathbf{G}_k-\mathbf{g}_k\|^2\right]+6(1-\beta_k)^2C_g\mathbb{E}\left[\|\mathbf{x}_{k+1}-\mathbf{x}_k\|^2\right]+3\beta_k^2V_g\\
	&\le(1-\beta_k)^2\mathbb{E}\left[\|\mathbf{G}_k-\mathbf{g}_k\|^2\right]+48(1-\beta_k)^2C_g\mathbb{E}\left[ \left\|\mathbf{x}_k-\mathbf{1}\otimes \bar{x}_k\right\|^2\right]\notag\\
	&\quad+12(1-\beta_k)^2C_g\alpha_k^2C_f\mathbb{E}\left[\normm{\mathbf{z}_k}_F^2\right]+3\beta_k^2V_g,
	\end{align*}
    which verifies (\ref{GG-bound}).
	
	Inequality (\ref{Gd-bound}) could be obtained by the similar analysis of (\ref{GG-bound}). The proof is complete.
\end{proof}

\textbf{Proof of Lemma \ref{lem:consensus}}
\begin{proof}
	We first provide the upper bound of consensus error $\mathbb{E}\left[\left\|\mathbf{x}_{k+1}-\mathbf{1}\otimes\bar{x}_{k+1}\right\|^2\right]$.
	By the definition of $\bar{x}_{k+1}$ and the row stochasticity of weight matrix $\mathbf{W}$,
	\begin{equation*}
	 \bar{x}_{k+1}=\left(\frac{\mathbf{1}^\intercal}{n}\otimes\mathbf{I}_d\right)\mathbf{x}_{k+1}=\left(\frac{\mathbf{1}^\intercal}{n}\otimes\mathbf{I}_d\right)\left(\tilde{\mathbf{W}}\mathbf{x}_k-\alpha_k\mathbf{U}_{k+1}\right)=\left(\frac{\mathbf{1}^\intercal}{n}\otimes\mathbf{I}_d\right)\left(\mathbf{x}_k-\alpha_k\mathbf{U}_{k+1}\right).
	\end{equation*}
    Combining  the above equality with the iteration (\ref{alg:x-1}),
	\begin{align}
	&\left\|\mathbf{x}_{k+1}-\mathbf{1}\otimes\bar{x}_{k+1}\right\|^2\notag\\
	 &=\left\|\left(\tilde{\mathbf{W}}-\frac{\mathbf{1}\mathbf{1}^\intercal}{n}\otimes\mathbf{I}_d\right)\mathbf{x}_k-\alpha_k\left(\mathbf{I}_{dn}-\frac{\mathbf{1}\mathbf{1}^\intercal}{n}\otimes\mathbf{I}_d\right)\mathbf{U}_{k+1}\right\|^2\notag\\
	 &=\left\|\left(\tilde{\mathbf{W}}-\frac{\mathbf{1}\mathbf{1}^\intercal}{n}\otimes\mathbf{I}_d\right)\left(\mathbf{x}_k-\mathbf{1}\otimes\bar{x}_k\right)-\alpha_k\left(\mathbf{I}_{dn}-\frac{\mathbf{1}\mathbf{1}^\intercal}{n}\otimes\mathbf{I}_d\right)\mathbf{U}_{k+1}\right\|^2\notag\\
	&\le (1+\tau)\left\|\left(\tilde{\mathbf{W}}-\frac{\mathbf{1}\mathbf{1}^\intercal}{n}\otimes\mathbf{I}_d\right)\left(\mathbf{x}_k-\mathbf{1}\otimes\bar{x}_k\right)\right\|^2+\left(1+\frac{1}{\tau}\right)\left\|\alpha_k\left(\mathbf{I}_{dn}-\frac{\mathbf{1}\mathbf{1}^\intercal}{n}\otimes\mathbf{I}_d\right)\mathbf{U}_{k+1}\right\|^2\notag\\
	&\le (1+\tau)\rho^2\left\|\mathbf{x}_k-\mathbf{1}\otimes\bar{x}_k\right\|^2+\left(1+\frac{1}{\tau}\right)\alpha_k^2\normm{\mathbf{I}_{n}-\frac{\mathbf{1}\mathbf{1}^\intercal}{n}}^2\left\|\mathbf{U}_{k+1}\right\|^2\notag\\
	\label{x-bound}&= (1+\tau)\rho^2\left\|\mathbf{x}_k-\mathbf{1}\otimes\bar{x}_k\right\|^2+\left(1+\frac{1}{\tau}\right)\alpha_k^2\left\|\mathbf{U}_{k+1}\right\|^2,
	\end{align}
	where $\rho\define \normm{\mathbf{W}-\frac{\mathbf{1}\mathbf{1}^\intercal}{n}}<1$ \cite[Lemma 4]{Li2022Aggregative}, the second equality follows from the row stochasticity of $\mathbf{W}$, the first inequality follows from that $(a+b)^2\le (1+\tau)a^2+(1+1/\tau)b^2$ for any $\tau>0$, the last equality follows from the fact $\normm{\mathbf{I}_{n}-\frac{\mathbf{1}\mathbf{1}^\intercal}{n}}=1$. Taking expectation on both sides of inequality (\ref{x-bound}),
	\begin{align}
	\mathbb{E}\left[\left\|\mathbf{x}_{k+1}-\mathbf{1}\otimes\bar{x}_{k+1}\right\|^2\right]
	&\le (1+\tau)\rho^2\mathbb{E}\left[\left\|\mathbf{x}_k-\mathbf{1}\otimes\bar{x}_k\right\|^2\right]+\left(1+\frac{1}{\tau}\right)\alpha_k^2\mathbb{E}\left[\left\|\mathbf{U}_{k+1}\right\|^2\right]\notag\\
	&\le (1+\tau)\rho^2\mathbb{E}\left[\left\|\mathbf{x}_k-\mathbf{1}\otimes\bar{x}_k\right\|^2\right]+\left(1+\frac{1}{\tau}\right)\alpha_k^2C_f\mathbb{E}\left[\normm{\mathbf{z}_k}_F^2\right]\notag,
	\end{align}
	where the last inequality follows from the definition of $\mathbf{U}_{k+1}$ and Assumption \ref{ass-objective}(c). Setting $\tau=\frac{1-\rho^2}{2\rho^2}$, we obtain (\ref{xx-bound}).
	
	Next, we provide the upper bound of consensus error $\mathbb{E}\left[\left\|\mathbf{y}_{k+1}-\mathbf{1}\otimes\bar{y}_{k+1}\right\|^2\right]$. Similar to  the analysis of (\ref{x-bound}), it follows from the definition of $\bar{y}_k$ and iteration (\ref{alg:z-1}) that
	\begin{align}
	\left\|\mathbf{y}_{k+1}-\mathbf{1}\otimes\bar{y}_{k+1}\right\|^2
	\label{z-bound}&\le (1+\tau)\rho^2\left\|\mathbf{y}_k-\mathbf{1}\otimes\bar{y}_k\right\|^2+\left(1+\frac{1}{\tau}\right)\left\|\mathbf{G}_{k+1}-\mathbf{G}_k\right\|^2.
	\end{align}
	By iteration (\ref{alg:G-1}),
    \begin{align*}
	&\mathbf{G}_{k+1}-\mathbf{G}_k\\
	&=-\beta_k\mathbf{G}_k+\left(1-\beta_k\right)\left(G_{k+1,k+1}-G_{k+1,k}\right)+\beta_kG_{k+1,k+1}\\
	 &=-\beta_k\left(\mathbf{G}_k-\mathbf{g}_k\right)+\left(1-\beta_k\right)\left(G_{k+1,k+1}-G_{k+1,k}\right)+\beta_k\left(G_{k+1,k+1}-\mathbf{g}_{k+1}\right)+\beta_k\left(\mathbf{g}_{k+1}-\mathbf{g}_{k}\right).
	\end{align*}
	Then
	\begin{align*}
	&\left\|\mathbf{y}_{k+1}-\mathbf{1}\otimes\bar{y}_{k+1}\right\|^2\\
	&\le (1+\tau)\rho^2\left\|\mathbf{y}_k-\mathbf{1}\otimes\bar{y}_k\right\|^2+4\left(1+\frac{1}{\tau}\right)\left(\beta_k^2\left\|\mathbf{G}_k-\mathbf{g}_k\right\|^2\right.\notag\\
	&\quad\left.+\left(1-\beta_k\right)^2\left\|G_{k+1,k+1}-G_{k+1,k}\right\|^2+\beta_k^2\left\|G_{k+1,k+1}-\mathbf{g}_{k+1}\right\|^2+\beta_k^2\left\|\mathbf{g}_{k+1}-\mathbf{g}_{k}\right\|^2\right).
	\end{align*}
	Taking expectation on both sides of the inequality above,
	\begin{align}
	&\mathbb{E}\left[\left\|\mathbf{y}_{k+1}-\mathbf{1}\otimes\bar{y}_{k+1}\right\|^2\right]\\
	&\le (1+\tau)\rho^2\mathbb{E}\left[\left\|\mathbf{y}_k-\mathbf{1}\otimes\bar{y}_k\right\|^2\right]+4\left(1+\frac{1}{\tau}\right)\left(\beta_k^2\mathbb{E}\left[\left\|\mathbf{G}_k-\mathbf{g}_k\right\|^2\right]\right.\notag\\
	 &\quad\left.+\left(1-\beta_k\right)^2\mathbb{E}\left[\left\|G_{k+1,k+1}-G_{k+1,k}\right\|^2\right]+\beta_k^2\mathbb{E}\left[\left\|G_{k+1,k+1}-\mathbf{g}_k\right\|^2+\beta_k^2\left\|\mathbf{g}_{k+1}-\mathbf{g}_{k}\right\|^2\right]\right)\notag\\
	&\le (1+\tau)\rho^2\mathbb{E}\left[\left\|\mathbf{y}_k-\mathbf{1}\otimes\bar{y}_k\right\|^2\right]+4\left(1+\frac{1}{\tau}\right)\left(\beta_k^2\mathbb{E}\left[\left\|\mathbf{G}_k-\mathbf{g}_{k+1}\right\|^2\right]\right.\notag\\
	&\quad\left.+C_g\mathbb{E}\left[\left\|\mathbf{x}_{k+1}-\mathbf{x}_{k}\right\|^2\right]+\beta_k^2V_g\right)\notag\\
	&\le (1+\tau)\rho^2\mathbb{E}\left[\left\|\mathbf{y}_k-\mathbf{1}\otimes\bar{y}_k\right\|^2\right]+4\left(1+\frac{1}{\tau}\right)\left(\beta_k^2\mathbb{E}\left[\left\|\mathbf{G}_k-\mathbf{g}_k\right\|^2\right]\right.\notag\\
	\label{zz-bound}&\quad\left.+2C_g\left(4\mathbb{E}\left[ \left\|\mathbf{x}_k-\mathbf{1}\otimes \bar{x}_k\right\|^2\right]+\alpha_k^2C_f\mathbb{E}\left[\normm{\mathbf{z}_k}_F^2\right]\right)+\beta_k^2V_g\right),
	\end{align}
	where the second inequality follows from Assumption \ref{ass-objective}(c), (d) and the fact $\beta_k\ge1$, the last inequality follows from
	  (\ref{x-x-bound}). Setting $\tau=\frac{1-\rho^2}{2\rho^2}$, we obtain (\ref{zz-bound-1}).
	
 Inequality	(\ref{yy-bound}) could be obtained by the similar analysis of (\ref{zz-bound-1}). The proof is complete.
\end{proof}

\textbf{Proof of Lemma \ref{lem:complex}}
\begin{proof}
	Multiplying $c_1,\cdots, c_4$ on both sides of (\ref{zz-bound-1}), (\ref{yy-bound}), (\ref{GG-bound}) and (\ref{Gd-bound}) respectively, and then adding up them to (\ref{xx-bound}), we have
	{\small\begin{align*}
		&V_{k+1}\\
		&\le \frac{2+\rho^2}{3}\mathbb{E}\left[\left\|\mathbf{x}_k-\mathbf{1}\otimes\bar{x}_k\right\|^2\right]+\frac{1+\rho^2}{2}c_1\mathbb{E}\left[\left\|\mathbf{y}_k-\mathbf{1}\otimes\bar{y}_k\right\|^2\right]+\frac{1+\rho^2}{2}c_2\mathbb{E}\left[\normm{\mathbf{z}_k-\mathbf{1}\otimes\bar{z}_k}^2\right]\\
		&\quad +\left((1-\beta_k)^2+6\frac{1-\rho^2}{1+\rho^2}\beta_k^2\right)c_3\mathbb{E}\left[\|\mathbf{G}_k-\mathbf{g}_k\|^2\right]+\left((1-\gamma_k)^2+6pL_g\frac{1-\rho^2}{1+\rho^2}\gamma_k^2\right)c_4\mathbb{E}\left[\normm{\hat{\mathbf{G}}_k-\nabla\mathbf{g}_k}_F^2\right]\\
		&\quad+ c_5\alpha_k^2\mathbb{E}\left[\normm{\mathbf{z}_k}_F^2\right]+3\left(\frac{1+\rho^2}{1-\rho^2}V_gc_1+V_gc_3\right)\beta_k^2+3\left(\frac{1+\rho^2}{1-\rho^2}2V_g^{'}c_2+V_g^{'}c_4\right)\gamma_k^2.
		\end{align*}}
Note that
	\begin{align*}
	&c_5\alpha_k^2\mathbb{E}\left[\normm{\mathbf{z}_k}_F^2\right]\notag\\
	 &=c_5\alpha_k^2\mathbb{E}\left[\normm{\mathbf{z}_k-\mathbf{1}\otimes\bar{z}_k+\mathbf{1}\otimes\bar{z}_k-\mathbf{1}\otimes\left(\frac{1}{n}\sum_{j=1}^n\nabla g_j(x_{j,k})\right)+\mathbf{1}\otimes\left(\frac{1}{n}\sum_{j=1}^n\nabla g_j(x_{j,k})\right)}_F^2\right]\notag\\
	&\le 3c_5\alpha_k^2\left(\mathbb{E}\left[\normm{\mathbf{z}_k-\mathbf{1}\otimes\bar{z}_k}_F^2\right]+n\mathbb{E}\left[\normm{\bar{z}_k-\frac{1}{n}\sum_{j=1}^n\nabla g_j(x_{j,k})}_F^2\right]+\sum_{j=1}^n\mathbb{E}\left[\normm{\nabla g_j(x_{j,k})}_F^2\right]\right)\notag\\
	&\le 3c_5\alpha_k^2\left(\mathbb{E}\left[\normm{\mathbf{z}_k-\mathbf{1}\otimes\bar{z}_k}_F^2\right]+\mathbb{E}\left[\normm{\hat{\mathbf{G}}_k-\nabla\mathbf{g}_k}_F^2\right]+npC_g\right)\notag\\
	 &\le 3pc_5\alpha_k^2\left(\mathbb{E}\left[\normm{\mathbf{z}_k-\mathbf{1}\otimes\bar{z}_k}^2\right]+\mathbb{E}\left[\normm{\hat{\mathbf{G}}_k-\nabla\mathbf{g}_k}^2\right]+nC_g\right).
	\end{align*}
	Then
	{\small\begin{align*}
		V_{k+1}&\le \frac{2+\rho^2}{3}\mathbb{E}\left[\left\|\mathbf{x}_k-\mathbf{1}\otimes\bar{x}_k\right\|^2\right]+\frac{1+\rho^2}{2}c_1\mathbb{E}\left[\left\|\mathbf{y}_k-\mathbf{1}\otimes\bar{y}_k\right\|^2\right]\\
		&\quad +\left(\frac{1+\rho^2}{2}+\frac{3pc_5}{c_2}\alpha_k^2\right)c_2\mathbb{E}\left[\normm{\mathbf{z}_k-\mathbf{1}\otimes\bar{z}_k}^2\right]+\left((1-\beta_k)^2+6\frac{1-\rho^2}{1+\rho^2}\beta_k^2\right)c_3\mathbb{E}\left[\|\mathbf{G}_k-\mathbf{g}_k\|^2\right]\\
		 &\quad+\left((1-\gamma_k)^2+6pL_g\frac{1-\rho^2}{1+\rho^2}\gamma_k^2+\frac{3pc_5}{c_4}\alpha_k^2\right)c_4\mathbb{E}\left[\normm{\hat{\mathbf{G}}_k-\nabla\mathbf{g}_k}_F^2\right]\\
		&\quad+ 3npc_5C_g\alpha_k^2+3\left(\frac{1+\rho^2}{1-\rho^2}V_gc_1+V_gc_3\right)\beta_k^2+3\left(\frac{1+\rho^2}{1-\rho^2}2V_g^{'}c_2+V_g^{'}c_4\right)\gamma_k^2\\
		&\le a_k V_k+b_k.
		\end{align*}}
	The proof is complete.	
\end{proof}

\textbf{Proof of Lemma \ref{lem:c-tracking}}
\begin{proof}
	The proof is similar to Lemma \ref{lem:tracking}.
\end{proof}
\textbf{Proof of Lemma \ref{lem:c-consensus}}
\begin{proof}
	The proof is similar to Lemma \ref{lem:consensus}.
\end{proof}

\textbf{Proof of Lemma \ref{lem:compression}}
\begin{proof}
	We first provide the upper bound of compression error $\mathbb{E}\left[\left\|\mathbf{x}_{k+1}-\mathbf{H}_{k+1}^x\right\|^2\right]$. By the Line \ref{c-4} of Algorithm \ref{COMMA},
	\begin{equation*}
	\mathbf{H}_{k+1}^x=\mathbf{H}_{k}^x+\alpha_x\mathcal{C}\left(\mathbf{x}_k-\mathbf{H}_k^x\right),
	\end{equation*}
	and then
	\begin{align}
	 &\mathbb{E}\left[\left\|\mathbf{x}_{k+1}-\mathbf{H}_{k+1}^x\right\|^2\right]\notag\\&=\mathbb{E}\left[\left\|\mathbf{x}_{k+1}-\mathbf{x}_{k}+\mathbf{x}_{k}-\left(\mathbf{H}_{k}^x+\alpha_x\mathcal{C}\left(\mathbf{x}_k-\mathbf{H}_k^x\right)\right)\right\|^2\right]\notag\\
	 &\le(1+\tau)\mathbb{E}\left[\left\|\mathbf{x}_{k}-\left(\mathbf{H}_{k}^x+\alpha_x\mathcal{C}\left(\mathbf{x}_k-\mathbf{H}_k^x\right)\right)\right\|^2\right]+\left(1+\frac{1}{\tau}\right)\mathbb{E}\left[\left\|\mathbf{x}_{k+1}-\mathbf{x}_{k}\right\|^2\right]\notag\\
	 &=(1+\tau)\mathbb{E}\left[\left\|\alpha_xr_1\left(\mathbf{x}_{k}-\mathbf{H}_{k}^x-\frac{\mathcal{C}\left(\mathbf{x}_k-\mathbf{H}_k^x\right)}{r_1}\right)+(1-\alpha_xr_1)\left(\mathbf{x}_{k}-\mathbf{H}_{k}^x\right)\right\|^2\right]\notag\\
	&\quad+\left(1+\frac{1}{\tau}\right)\mathbb{E}\left[\left\|\mathbf{x}_{k+1}-\mathbf{x}_{k}\right\|^2\right]\notag\\
	 &\le(1+\tau)\mathbb{E}\left[\alpha_xr_1\left\|\mathbf{x}_{k}-\mathbf{H}_{k}^x-\frac{\mathcal{C}\left(\mathbf{x}_k-\mathbf{H}_k^x\right)}{r_1}\right\|^2+(1-\alpha_xr_1)\left\|\mathbf{x}_{k}-\mathbf{H}_{k}^x\right\|^2\right]\notag\\
	&\quad+\left(1+\frac{1}{\tau}\right)\mathbb{E}\left[\left\|\mathbf{x}_{k+1}-\mathbf{x}_{k}\right\|^2\right]\notag\\
	 %&\le(1+\tau)\mathbb{E}\left[\left(\alpha_xr_1\left(1-\psi_1\right)+(1-\alpha_xr_1)\right)\left\|\mathbf{x}_{k}-\mathbf{H}_{k}^x\right\|^2\right]\\
	%&\quad+\left(1+\frac{1}{\tau}\right)\mathbb{E}\left[\left\|\mathbf{x}_{k+1}-\mathbf{x}_{k}\right\|^2\right]\\
	 \label{comp-x-1}&\le(1+\tau)\mathbb{E}\left[(1-\alpha_xr_1\psi_1)\left\|\mathbf{x}_{k}-\mathbf{H}_{k}^x\right\|^2\right]+\left(1+\frac{1}{\tau}\right)\mathbb{E}\left[\left\|\mathbf{x}_{k+1}-\mathbf{x}_{k}\right\|^2\right],
	\end{align}
	where the first inequality follows from that $(a+b)^2\le (1+\tau)a^2+(1+1/\tau)b^2$ for any $\tau>0$, the second inequality follows from the convexity of $\|\cdot\|^2$ and the last inequality follows from Assumption  \ref{ass:compressor}. For the term $\mathbb{E}\left[\left\|\mathbf{x}_{k+1}-\mathbf{x}_{k}\right\|^2\right]$ on the right hand side of above inequality,
	\begin{align}
	&\mathbb{E}\left[\|\mathbf{x}_{k+1}-\mathbf{x}_k\|^2\right]\notag\\
	&=\mathbb{E}\left[\left\|\alpha_w\left[\tilde{\mathbf{W}}-\mathbf{I}_{nd}\right](\mathbf{x}_k-\mathbf{1}\otimes \bar{x}_k)-\alpha_k\mathbf{U}_{k+1}+\alpha_w\left[\mathbf{I}_{nd}-\tilde{\mathbf{W}}\right]\mathbf{E}_{k+1}^x\right\|^2\right]\notag\\
	&\le 3\left(4\alpha_w^2\mathbb{E}\left[ \left\|\mathbf{x}_k-\mathbf{1}\otimes \bar{x}_k\right\|^2\right]+\alpha_k^2\mathbb{E}\left[\left\|\mathbf{U}_{k+1}\right\|^2\right]+4\alpha_w^2\mathbb{E}\left[\left\|\mathbf{E}_{k+1}^y\right\|^2\right]\right)\notag\\
	\label{x-x}&\le 3\left(4\alpha_w^2\mathbb{E}\left[ \left\|\mathbf{x}_k-\mathbf{1}\otimes \bar{x}_k\right\|^2\right]+\alpha_k^2C_f\mathbb{E}\left[\normm{\mathbf{z}_k}_F^2\right]+4\alpha_w^2\breve{r}_1\mathbb{E}\left[\left\|\mathbf{x}_k-\mathbf{H}_k^x\right\|^2\right]\right),
	\end{align}
	where the first equality follows from (\ref{alg:c-x}) and the row stochasticity of $\mathbf{W}$, the first inequality follows from the fact $\normm{\mathbf{W}-\mathbf{I}_{n}}\le 2$, the second inequality follows from  Assumption \ref{ass-objective}(c), the fact $\mathbf{E}_{k+1}^x=\mathbf{x}_k-\mathbf{H}_k^x-\mathcal{C}\left(\mathbf{x}_k-\mathbf{H}_k^x\right)$ and (\ref{c1}).
	Substituting the above inequality into (\ref{comp-x-1}) and setting $\tau=\alpha_xr_1\psi_1$,
	\begin{align*}
	\mathbb{E}\left[\left\|\mathbf{x}_{k+1}-\mathbf{H}_{k+1}^x\right\|^2\right]
	%	 &\le\left[(1+\tau)(1-\alpha_xr_1\psi_1)+12\left(1+\frac{1}{\tau}\right)\alpha_w^2\breve{r}_1\right]\mathbb{E}\left[\left\|\mathbf{x}_{k}-\mathbf{H}_{k}^x\right\|^2\right]\notag\\
	%	&\quad+3\left(1+\frac{1}{\tau}\right)\left(4\alpha_w^2\mathbb{E}\left[ \left\|\mathbf{x}_k-\mathbf{1}\otimes \bar{x}_k\right\|^2\right]+\alpha_k^2C_f\mathbb{E}\left[\normm{\mathbf{z}_k}_F^2\right]\right)\\
	 &\le\left(1-(\alpha_xr_1\psi_1)^2+12\frac{1+\alpha_xr_1\psi_1}{\alpha_xr_1\psi_1}\alpha_w^2\breve{r}_1\right)\mathbb{E}\left[\left\|\mathbf{x}_{k}-\mathbf{H}_{k}^x\right\|^2\right]\notag\\
	&\quad+\frac{1+\alpha_xr_1\psi_1}{\alpha_xr_1\psi_1}\left(12\alpha_w^2\mathbb{E}\left[ \left\|\mathbf{x}_k-\mathbf{1}\otimes \bar{x}_k\right\|^2\right]+3\alpha_k^2C_f\mathbb{E}\left[\normm{\mathbf{z}_k}_F^2\right]\right).
	\end{align*}
	Noting that $\alpha_w^2\le\frac{(\alpha_xr_1\psi_1)^3}{24\breve{r}_1\left(1+\alpha_xr_1\psi_1\right)}$, (\ref{comp-x-0}) holds.

	Next, we provide the upper bound of compression error $\mathbb{E}\left[\left\|\mathbf{y}_{k+1}-\mathbf{H}_{k+1}^y\right\|^2\right]$.
	By the Line \ref{c-4} of Algorithm \ref{COMMA},
	\begin{equation*}
	\mathbf{H}_{k+1}^y=\mathbf{H}_{k}^y+\alpha_y\mathcal{C}\left(\mathbf{y}_k-\mathbf{H}_k^y\right),
	\end{equation*}
	and then
	\begin{align}
	 \mathbb{E}\left[\left\|\mathbf{y}_{k+1}-\mathbf{H}_{k+1}^y\right\|^2\right]&=\mathbb{E}\left[\left\|\mathbf{y}_{k+1}-\mathbf{y}_{k}+\mathbf{y}_{k}-\left(\mathbf{H}_{k}^y+\alpha_y\mathcal{C}\left(\mathbf{y}_k-\mathbf{H}_k^y\right)\right)\right\|^2\right]\notag\\
	 &\le(1+\tau)\mathbb{E}\left[\left\|\mathbf{y}_{k}-\left(\mathbf{H}_{k}^y+\alpha_y\mathcal{C}\left(\mathbf{y}_k-\mathbf{H}_k^y\right)\right)\right\|^2\right]+\left(1+\frac{1}{\tau}\right)\mathbb{E}\left[\left\|\mathbf{y}_{k+1}-\mathbf{y}_{k}\right\|^2\right]\notag\\
	 &=(1+\tau)\mathbb{E}\left[\left\|\alpha_yr_1\left(\mathbf{y}_{k}-\mathbf{H}_{k}^y-\frac{\mathcal{C}\left(\mathbf{y}_k-\mathbf{H}_k^y\right)}{r_1}\right)+(1-\alpha_yr_1)\left(\mathbf{y}_{k}-\mathbf{H}_{k}^y\right)\right\|^2\right]\notag\\
	&\quad+\left(1+\frac{1}{\tau}\right)\mathbb{E}\left[\left\|\mathbf{y}_{k+1}-\mathbf{y}_{k}\right\|^2\right]\notag\\
	 &\le(1+\tau)\mathbb{E}\left[\alpha_yr_1\left\|\mathbf{y}_{k}-\mathbf{H}_{k}^y-\frac{\mathcal{C}\left(\mathbf{y}_k-\mathbf{H}_k^y\right)}{r_1}\right\|^2+(1-\alpha_yr_1)\left\|\mathbf{y}_{k}-\mathbf{H}_{k}^y\right\|^2\right]\notag\\
	&\quad+\left(1+\frac{1}{\tau}\right)\mathbb{E}\left[\left\|\mathbf{y}_{k+1}-\mathbf{y}_{k}\right\|^2\right]\notag\\
	 %&\le(1+\tau)\mathbb{E}\left[\left(\alpha_yr_1\left(1-\psi_1\right)+(1-\alpha_yr_1)\right)\left\|\mathbf{y}_{k}-\mathbf{H}_{k}^y\right\|^2\right]\\
	%&\quad+\left(1+\frac{1}{\tau}\right)\mathbb{E}\left[\left\|\mathbf{y}_{k+1}-\mathbf{y}_{k}\right\|^2\right]\\
	 \label{comp-z-1}&\le(1+\tau)\mathbb{E}\left[(1-\alpha_yr_1\psi_1)\left\|\mathbf{y}_{k}-\mathbf{H}_{k}^y\right\|^2\right]+\left(1+\frac{1}{\tau}\right)\mathbb{E}\left[\left\|\mathbf{y}_{k+1}-\mathbf{y}_{k}\right\|^2\right],
	\end{align}
	where the first inequality follows from that $(a+b)^2\le (1+\tau)a^2+(1+1/\tau)b^2$ for any $\tau>0$, the second inequality follows from the convexity of $\|\cdot\|^2$ and the last inequality follows from Assumption  \ref{ass:compressor}. Similar to (\ref{x-x}),
	{\small\begin{align}
		\mathbb{E}\left[\|\mathbf{y}_{k+1}-\mathbf{y}_k\|^2\right]
		&\le 3\left(4\alpha_w^2\mathbb{E}\left[ \left\|\mathbf{y}_k-\mathbf{1}\otimes \bar{y}_k\right\|^2\right]+4\alpha_w^2\breve{r}_2\mathbb{E}\left[\left\|\mathbf{y}_k-\mathbf{H}_k^y\right\|^2\right]\right)\notag\\
		\label{z-z-bound}&\quad+9C_g\left(4\alpha_w^2\mathbb{E}\left[ \left\|\mathbf{x}_k-\mathbf{1}\otimes \bar{x}_k\right\|^2\right]+\alpha_k^2C_f\mathbb{E}\left[\normm{\mathbf{z}_k}_F^2\right]+4\alpha_w^2\breve{r}_1\mathbb{E}\left[\left\|\mathbf{x}_k-\mathbf{H}_k^x\right\|^2\right]\right),
		\end{align}}
	Substituting (\ref{z-z-bound}) into (\ref{comp-z-1}) and setting $\tau=\alpha_yr_2\psi_2$,
	\begin{align*}
	\mathbb{E}\left[\left\|\mathbf{y}_{k+1}-\mathbf{H}_{k+1}^y\right\|^2\right]
	&\le \left(1-(\alpha_yr_2\psi_2)^2+12\frac{1+\alpha_xr_1\psi_1}{\alpha_xr_1\psi_1}\alpha_w^2\breve{r}_2\right)\mathbb{E}\left[\left\|\mathbf{y}_{k}-\mathbf{H}_{k}^y\right\|^2\right]\notag\\
	&\quad+9C_g\frac{1+\alpha_yr_2\psi_2}{\alpha_yr_2\psi_2}\left(4\alpha_w^2\mathbb{E}\left[ \left\|\mathbf{x}_k-\mathbf{1}\otimes \bar{x}_k\right\|^2\right]+4\alpha_w^2\breve{r}_1\mathbb{E}\left[\left\|\mathbf{x}_k-\mathbf{H}_k^x\right\|^2\right]\right)\notag\\
	&\quad +12\frac{1+\alpha_yr_2\psi_2}{\alpha_yr_2\psi_2}\alpha_w^2\mathbb{E}\left[ \left\|\mathbf{y}_k-\mathbf{1}\otimes \bar{y}_k\right\|^2\right]+9C_g\frac{1+\alpha_yr_2\psi_2}{\alpha_yr_2\psi_2}\alpha_k^2C_f\mathbb{E}\left[\normm{\mathbf{z}_k}_F^2\right].
	\end{align*}
	Noting that $\alpha_w^2\le\frac{(\alpha_yr_2\psi_2)^3}{24\breve{r}_2\left(1+\alpha_yr_2\psi_2\right)}$,  (\ref{comp-z-0}) holds.
	
Inequality	(\ref{comp-y-0}) could be obtained by the similar analysis of (\ref{comp-z-0}).
\end{proof}
\textbf{Proof of Lemma \ref{lem:c-complex}}
\begin{proof}
	The proof is similar to Lemma \ref{lem:complex}.
\end{proof}

\newpage
\section{}\label{CD-NSCGD}\label{apd:CD-AS-agent}
Here we present CD-ASCGD method from each agent’s perspective.
\begin{algorithm}[H]
	\caption{D-ASCGD method from agents' view:}\label{alg:CSIA-e}
	\begin{algorithmic}[1]
		\REQUIRE  initial values $x_{i,1},H_{i,1}^x\in\mathbb{R}^{d}$,  $H_{i,1}^y\in\mathbb{R}^{p}$, $H_{i,1}^{z}\in\mathbb{R}^{d\times p}$, $y_{i,1}=G_{i,1}\in\mathbb{R}^{p}$, $z_{i,1}=\hat{G}_{i,1}\in\mathbb{R}^{d\times p}$; stepsizes $\alpha_k>0$, $\beta_k>0,\gamma_k>0$; scaling parameters $\alpha_w\in(0,)$; nonnegative weight matrix $\mathbf{W}=\{w_{ij}\}_{1\le i,j\le n}\in \mathbb{R}^{n\times n}$    %%input
		%\ENSURE Optimum set $\mathbb{X}$  %%output
		\STATE $H_{i,1}^{x,w}=\sum_{j=1}^n w_{ij} H_{j,1}^x$,~\small$H_{i,1}^{y,w}=\sum_{j=1}^n w_{ij} H_{j,1}^y$,~\small$H_{i,1}^{z,w}=\sum_{j=1}^n w_{ij} H_{j,1}^{z}$
		\FOR {$k=1,2,\cdots$}
		\STATE $q_{i,k}^x=\mathbf{Compress}\left(x_{i,k}-H_{i,k}^x\right)$\quad \quad \quad \quad \quad \quad \quad \quad $\triangleright$ Compression
		\STATE $\breve{x}_{i,k}=H_{i,k}^x+q_{i,k}^x$
		\STATE $\breve{x}_{i,k}^w=H_{i,k}^{x,w}+\sum_{j=1}^n w_{ij}q_{j,k}^x$\quad \quad \quad \quad \quad \quad \quad \quad \quad \quad \quad \quad $\triangleright$ Communication
		\STATE $H_{i,k+1}^x=(1-\alpha_x)H_{i,k+1}^x+\alpha_x \breve{x}_{i,k}$
		\STATE $H_{i,k+1}^{x,w}=(1-\alpha_x)H_{i,k+1}^{x,w}+\alpha_x \breve{x}_{i,k}^w$
		\STATE $q_{i,k}^y=\mathbf{Compress}\left(y_{i,k}-H_{i,k}^y\right)$\quad \quad \quad \quad \quad \quad \quad \quad $\triangleright$ Compression
		\STATE $\breve{y}_{i,k}=H_{i,k}^y+q_{i,k}^y$
		\STATE $\breve{y}_{i,k}^w=H_{i,k}^{y,w}+\sum_{j=1}^n w_{ij}q_{j,k}^y$\quad \quad \quad \quad \quad \quad \quad \quad \quad \quad \quad \quad $\triangleright$ Communication
		\STATE $H_{i,k+1}^y=(1-\alpha_y)H_{i,k+1}^y+\alpha_y \breve{y}_{i,k}$
		\STATE $H_{i,k+1}^{y,w}=(1-\alpha_y)H_{i,k+1}^{y,w}+\alpha_y \breve{y}_{i,k}^w$
		\STATE $q_{i,k}^{z}=\mathbf{Compress}\left(z_{i,k}-H_{i,k}^{z}\right)$\quad \quad \quad \quad \quad \quad \quad \quad $\triangleright$ Compression
		\STATE $\breve{z}_{i,k}=H_{i,k}^{z}+q_{i,k}^{z}$
		\STATE $\breve{z}_{i,k}^w=H_{i,k}^{z,w}+\sum_{j=1}^n w_{ij}q_{j,k}^{z}$\quad \quad \quad \quad \quad \quad \quad \quad \quad \quad \quad \quad $\triangleright$ Communication
		\STATE $H_{i,k+1}^{z}=(1-\alpha_z)H_{i,k+1}^{z}+\alpha_z \breve{z}_{i,k}$
		\STATE $H_{i,k+1}^{z,w}=(1-\alpha_z)H_{i,k+1}^{z,w}+\alpha_z \breve{z}_{i,k}^w$
		\STATE Draw $\phi_{i,k+1}\stackrel{iid}\sim P_{\phi_i},~\zeta_{i,k+1}\stackrel{iid}\sim P_{\zeta_i}$, and compute function values	 $G_i(x_{i,k};\phi_{i,k+1})$, and gradients $\nabla F_i(y_{i,k};\zeta_{i,k+1})$, $\nabla G_i(x_{i,k};\phi_{i,k+1})$ and
			$\nabla G_i(x_{i,k+1};\phi_{i,k+1})$
		\STATE $x_{i,k+1}=x_{i,k}-\alpha_w\left(\breve{x}_{i,k}-\breve{x}_{i,k}^w\right)-\alpha_kz_{i,k}\nabla F_i(y_{i,k};\zeta_{i,k+1})$
		\STATE $G_{i,k+1}=(1-\beta_k)\left(G_{i,k}-G_i(x_{i,k};\phi_{i,k+1})\right)+ G_i(x_{i,k+1};\phi_{i,k+1})$
		\STATE $\hat{G}_{i,k+1}=(1-\gamma_k)\left(\hat{G}_{i,k}-\nabla G_i(x_{i,k};\phi_{i,k+1})\right)+\nabla G_i(x_{i,k+1};\phi_{i,k+1})$
		\STATE $y_{i,k+1}=y_{i,k}-\alpha_w\left(\breve{y}_{i,k}-\breve{y}_{i,k}^w\right)+G_{i,k+1}-G_{i,k}$
		\STATE $z_{i,k+1}=z_{i,k}-\alpha_w\left(\breve{z}_{i,k}-\breve{z}_{i,k}^w\right)+\hat{G}_{i,k+1}-\hat{G}_{i,k}$
		\ENDFOR
	\end{algorithmic}
\end{algorithm}

\end{appendices}

\end{document}